\documentclass[twoside,12pt]{article} 
\usepackage[margin=1.0in]{geometry}

\usepackage{amssymb,amsmath,amsthm}
\usepackage{mathtools}
\usepackage{bm,lscape}
\usepackage{bbm}
\usepackage{algpseudocode}
\usepackage[ruled,vlined]{algorithm2e}
\usepackage{mathrsfs,dsfont}
\usepackage{wasysym}
\RequirePackage[
    colorlinks=true,
    linkcolor=blue,
    urlcolor=blue,
    citecolor=blue]{hyperref}
\RequirePackage{natbib}
\usepackage[utf8]{inputenc}
\usepackage[english]{babel}

\usepackage{graphicx}
\usepackage{color}
\usepackage{rotating}
\usepackage{authblk}
\usepackage{appendix}

\allowdisplaybreaks

\allowdisplaybreaks

\newcommand{\R}{\mathbb{R}}

\newcommand{\N}{\mathbb{N}}

\newcommand{\ddr}{\mathrm{d}}

\newtheorem{thm}{Theorem}[section]
\newtheorem{lem}[thm]{Lemma}
\newtheorem{defi}[thm]{Definition}

\newtheorem{prp}[thm]{Proposition}

\newtheorem{cor}[thm]{Corollary}
\newtheorem{remark}[thm]{Remark}

\providecommand{\keywords}[1]
{
  \small	
  \textbf{\textit{Keywords:}} #1
}

\begin{document}

\title{Laws of large numbers and central limit theorem for Ewens-Pitman model}


\author[1]{Claudia Contardi\thanks{claudia.contardi01@universitadipavia.it}}
\author[1]{Emanuele Dolera\thanks{emanuele.dolera@unipv.it}}
\author[3]{Stefano Favaro\thanks{stefano.favaro@unito.it}}
\affil[1]{\small{Department of Mathematics, University of Pavia, Italy}}
\affil[2]{\small{Department of Economics and Statistics \& Collegio Carlo Alberto, University of Torino, Italy}}

\maketitle

\begin{abstract}
The Ewens-Pitman model is a distribution for random partitions of the set $\{1,\ldots,n\}$, with $n\in\mathbb{N}$, indexed by parameters $\alpha \in [0,1)$ and $\theta>-\alpha$, such that $\alpha=0$ is the Ewens model in population genetics. The large $n$ asymptotic behaviour of the number $K_{n}$ of blocks in the Ewens-Pitman random partition has been extensively investigated in terms of almost-sure and Gaussian fluctuations, which show that $K_{n}$ scales as $\log n$ and $n^{\alpha}$ depending on whether $\alpha=0$ or $\alpha\in(0,1)$, providing non-random and random limiting behaviours, respectively. In this paper, we study the large $n$ asymptotic behaviour of $K_{n}$ when the parameter $\theta$ is allowed to depend linearly on $n\in\mathbb{N}$, a non-standard asymptotic regime first considered for $\alpha=0$ in Feng (\textit{The Annals of Applied Probability}, \textbf{17}, 2007). In particular, for $\alpha\in[0,1)$ and $\theta=\lambda n$, with $\lambda>0$, we establish a law of large numbers (LLN) and a central limit theorem (CLT) for $K_{n}$, which show that $K_{n}$ scales as $n$, providing non-random limiting behaviours. Depending on whether $\alpha=0$ or $\alpha\in(0,1)$, our results rely on different arguments. For $\alpha=0$ we rely on the representation of $K_{n}$ as a sum of independent, but not identically distributed, Bernoulli random variables, which leads to a refinement of the CLT in terms of a Berry-Esseen theorem. Instead, for $\alpha\in(0,1)$, we rely on a compound Poisson construction of $K_{n}$, leading to prove LLNs, CLTs and Berry-Esseen theorems for the number of blocks of the negative-Binomial compound Poisson random partition, which are of independent interest.
\end{abstract}



\section{Introduction}

The Ewens-Pitman model for random partitions first appeared in \citet{Pit(95)} as a two-parameter generalization of the Ewens model in population genetics \citep{Ewe(72),Cra(16)}. For $n\in\mathbb{N}$, let $\Pi_{n}$ be a random partition of $\{1,\ldots,n\}$ into $K_{n}\in\{1,\ldots,n\}$ blocks of sizes or frequencies $\mathbf{N}_{n}=(N_{1,n},\ldots,N_{K_{n},n})\in\mathbb{N}^{K_{n}}$ such that $n=\sum_{1\leq i\leq K_{n}}N_{i,n}$. For $\alpha\in[0,1)$ and $\theta>-\alpha$, the Ewens-Pitman model assigns to $\Pi_{n}$ the probability 
\begin{equation}\label{epsm}
P[K_{n}=k,\mathbf{N}_{n}=(n_{1},\ldots,n_{k})]=\frac{1}{k!}{n\choose n_{1},\ldots,n_{k}}\frac{[\theta]_{(k,\alpha)}}{[\theta]_{(n)}}\prod_{i=1}^{n}[1-\alpha]_{(n_{i}-1)},
\end{equation}
where $[x]_{(n,a)}$ denotes the rising factorial of $x$ of order $n$ and increment $a$, i.e. $[x]_{(n,a)} :=\prod_{0\leq i\leq n-1}(x+ia)$ and $[x]_{(n)} :=[x]_{(n,1)}$. The distribution \eqref{epsm} admits a sequential construction in terms of the Chinese restaurant process \citep{Pit(95),Zab(05)} and a Poisson process construction by random sampling the two-parameter Poisson-Dirichlet distribution \citep{Per(92),Pit(97)}; see also \citet{Dol(20),Dol(21)} for a construction through the negative-Binomial compound Poisson model for random partitions \citep{Cha(07)}. For $\alpha=0$ the Ewens-Pitman model reduces to the Ewens model, arising by random sampling the Poisson-Dirichlet distribution \citep{Kin(75)}. The Ewens-Pitman model plays a critical role in a variety of research areas, e.g., population genetics, Bayesian statistics, combinatorics, machine learning and statistical physics. See \citet[Chapter 3]{Pit(06)} for an overview of the Ewens-Pitman model and generalizations thereof.

There have been several works on the large $n$ asymptotic behaviour of $K_{n}$ under the Ewens-Pitman model, which show different scalings and limits for $K_{n}$ depending on whether $\alpha=0$ or $\alpha\in(0,1)$; see \citet[Chapter 3 and Chapter 4]{Pit(06)} and references therein. Denote the almost sure and weak convergence by $\stackrel{a.s.}{\longrightarrow}$ and $\stackrel{w}{\longrightarrow}$, respectively. For $\alpha=0$ and $\theta>0$, as $n\rightarrow+\infty$ there hold
\begin{equation}\label{llnk_DP}
\frac{K_{n}}{\log n}\stackrel{a.s.}{\longrightarrow}\theta
\end{equation}
and 
\begin{equation}\label{cltk_DP}
\sqrt{\log n}\left(\frac{K_{n}}{\log n}-\theta\right)\stackrel{w}{\longrightarrow}\sqrt{\theta}\,\mathcal{N}(0,1),
\end{equation}
where $\mathcal{N}(0,1)$ denotes the standard Gaussian random variable; see \citet[Theorem 2.3]{Kor(73)}. For $\alpha\in(0,1)$ and $\theta>-\alpha$, as $n\rightarrow+\infty$ there hold 
\begin{equation}\label{as_limit}
\frac{K_{n}}{n^{\alpha}}\stackrel{a.s.}{\longrightarrow}S_{\alpha,\theta}
\end{equation}
and 
\begin{equation}\label{gaussian_limit}
\sqrt{n^{\alpha}}\left(\frac{K_{n}}{n^{\alpha}}-S_{\alpha,\theta}\right)\stackrel{w}{\longrightarrow}\sqrt{\tilde{S}_{\alpha,\theta}}\,\mathcal{N}(0,1),
\end{equation}
where $S_{\alpha,\theta}$ and $\tilde{S}_{\alpha,\theta}$ denote scaled Mittag-Leffler random variables \citep{Zol(86),Bin(89)}, sharing the same distribution, and $\tilde{S}_{\alpha,\theta}$ is independent of $\mathcal{N}(0,1)$; see \citet[Theorem 3.8]{Pit(06)} and \citet[Theorem 2.8]{Ber(24)}.

\begin{remark} 
Beyond the almost-sure and Gaussian fluctuations displayed in \eqref{llnk_DP}-\eqref{gaussian_limit}, $K_{n}$ has been investigated with respect to large and moderate deviations \citep{Fen(98),Fav(14),Fav(18)} and laws of iterated logarithm \citep{Ber(24)}. Non-asymptotic results for $K_{n}$ have been established in terms of Berry-Esseen theorems \citep{Dol(20)} and concentration inequalities \citep{Per(22)}.
\end{remark}

\subsection{Our contributions}

Under the Ewens-Pitman model with $\alpha\in[0,1)$ and $\theta>0$, we study the large $n$ asymptotic behaviour of $K_{n}$ when the parameter $\theta$ is allowed to depend linearly on $n\in\mathbb{N}$. This is a non-standard asymptotic regime first considered for $\alpha=0$ in Feng \citet{Fen(07)}. More precisely, for $\alpha\in[0,1)$ and $\theta=\lambda n$, with $\lambda>0$, we establish a law of large numbers (LLN) and a central limit theorem (CLT) for $K_{n}$. Denoting by $\stackrel{p}{\longrightarrow}$ the convergence in probability, the next theorem states the main results of the paper.

\begin{thm}\label{thm_main}
For $n\in\mathbb{N}$, let $K_{n}$ be the number of partition blocks under the Ewens-Pitman model with parameter $\alpha\in[0,1)$ and $\theta=\lambda n$, with $\lambda>0$. If 
\begin{displaymath}
\mathfrak{m}_{\alpha,\lambda}:=\begin{cases} \frac{\lambda}{\alpha}\left[\left(1+\frac{1}{\lambda}\right)^{\alpha}-1\right] & \text{ for } \alpha \in (0, 1)\\[0.4cm]
\lambda \log \left(1+\frac{1}{\lambda}\right) &  \text{ for } \alpha = 0
\end{cases} 
\end{displaymath}
and
\begin{displaymath}
\mathfrak{s}_{\alpha,\lambda}^{2}:= \begin{cases}  \frac{\lambda}{\alpha}\left[\left(1+\frac{1}{\lambda}\right)^{2\alpha}\left(1-\frac{\alpha}{1+\lambda}\right)-\left(1+\frac{1}{\lambda}\right)^{\alpha}\right]& \text{ for } \alpha \in (0, 1)\\[0.4cm]
\lambda \log \left(1 + \frac{1}{\lambda}\right)  - \frac{\lambda}{\lambda +1}&  \text{ for } \alpha = 0,
\end{cases}
\end{displaymath}
then, as $n\rightarrow+\infty$ there hold:
\begin{itemize}
\item[i)]
\begin{equation}\label{mom_m}
\mathbb{E}\left[K_n\right] =n\mathfrak{m}_{\alpha, \lambda} + O(1)
\end{equation}
and
\begin{equation}\label{mom_v}
\operatorname{Var}(K_n) =n \mathfrak{s}_{\alpha, \lambda}^2 + O(1);
\end{equation}
\item[ii)]
\begin{equation}\label{lln}
\frac{K_{n}}{n}\stackrel{p}{\longrightarrow}\mathfrak{m}_{\alpha, \lambda};
\end{equation}
\item[iii)]
\begin{equation}\label{clt}
\frac{K_{n}-n\mathfrak{m}_{\alpha, \lambda}}{\sqrt{n\mathfrak{s}_{\alpha, \lambda}^{2}}}\stackrel{\text{w}}{\longrightarrow}\mathcal{N}(0,1).
\end{equation}
\end{itemize}
Further, for any $\lambda>0$ there hold that $\mathfrak{m}_{0, \lambda} = \lim_{\alpha \to 0^+} \mathfrak{m}_{\alpha, \lambda}$ and $\mathfrak{s}^2_{0, \lambda} = \lim_{\alpha \to 0^+} \mathfrak{s}^2_{\alpha, \lambda}$.
\end{thm}

Theorem \ref{thm_main} provides counterparts of the almost-sure fluctuations \eqref{llnk_DP} and \eqref{as_limit}, as well as of the Gaussian fluctuations \eqref{cltk_DP} and \eqref{gaussian_limit}, under the regime $\theta=\lambda n$, with $\lambda>0$. By comparing the LLN \eqref{lln} with the almost-sure fluctuations \eqref{llnk_DP} and \eqref{as_limit}, we observe how the regime $\theta=\lambda n$ affects the large $n$ asymptotic behaviour of $K_{n}$ in terms of both the scalings and the limiting behaviours. While the almost-sure fluctuations scale as $\log n$ and $n^{\alpha}$, depending on whether $\alpha=0$ or $\alpha\in(0,1)$, the LLN \eqref{lln} scales as $n$ for $\alpha\in[0,1)$, this being the ``usual" scaling for a LLN.  Further, while the almost-sure fluctuations have a non-random and a random limiting behaviour depending on whether $\alpha=0$ or $\alpha\in(0,1)$, which determine the non-random and random centering in \eqref{cltk_DP} and \eqref{gaussian_limit}, respectively, the LLN \eqref{lln} has a non-random limiting behaviour for $\alpha\in[0,1)$, determining a non-random centering in the CLT \eqref{clt}.

The LLN \eqref{lln} relies on combining Chebyshev's inequality with the asymptotic expansions \eqref{mom_m} and \eqref{mom_v}. These expansions, in turn, 
are obtained from the distribution of $K_{n}$, which follows by marginalizing \eqref{epsm} with $\theta=\lambda n$, for $\lambda>0$. A stronger version of the LLN \eqref{lln} will be also established, 
showing that as $n\rightarrow+\infty$ \begin{equation}\label{lln_strong}
\frac{K_{n}}{n}\stackrel{a.s.}{\longrightarrow}\mathfrak{m}_{\alpha, \lambda}.
\end{equation}
While the LLN \eqref{lln} relies on the sole (marginal) distribution of $K_{n}$, the strong LLN \eqref{lln_strong} requires to consider the finite-dimensional laws of the sequence $\{K_n\}_{n \geq 1}$. That is, for each $n\geq1$, it requires to look at $K_{n}$ as the number of blocks of the random partition of $\{1,\ldots,n\}$ induced by random sampling the two-parameter Poisson-Dirichlet distribution with $\alpha\in[0,1)$ and $\theta=\lambda n$, for $\lambda>0$.

Depending on whether $\alpha=0$ or $\alpha\in(0,1)$, the CLT \eqref{clt} relies on different arguments. For $\alpha=0$, the CLT relies on the representation of $K_{n}$ as a sum of independent, but not identically distributed, Bernoulli random variables, leading to a refinement of the CLT in terms of a Berry-Esseen theorem with respect to the Kolmogorov metric $\left\|\cdot \right\|_\infty$. In particular, if $F_n$ and $\Phi$ denote the cumulative distribution functions of $n^{-1/2}\mathfrak{s}_{0, \lambda}^{-1}(K_n - n \mathfrak{m}_{0, \lambda})$ and $\mathcal{N}(0,1)$, respectively, then we will show that there exist a constant $C>0$ and $\bar{n} \in \mathbb{N}$ such that
\begin{equation}\label{eq:be_dir}
\left\|F_{n}- \Phi \right\|_\infty \le \frac{C \,\log(n)}{n^{1/8}}
\end{equation}
holds for every $n \ge \bar{n}$; the constant $C$ is obtained constructively, and it can be made explicit by gathering equations in the proof of the Berry-Esseen inequality \eqref{eq:be_dir}. For $\alpha\in(0,1)$ the CLT relies on the compound Poisson construction of $K_{n}$, leading to LLNs, CLTs and Berry-Esseen theorems for the number of blocks of the negative-Binomial compound Poisson random partition, which are of independent interest \citep{Cha(07),Dol(21)}. 

\subsection{Related work}

There exists a rich literature on the large $\theta$ asymptotic behaviour of the two-parameter Poisson-Dirichlet distribution, assuming $\alpha\in[0,1)$ and $\theta>0$. In population genetics, the Poisson-Dirichlet distribution (i.e. $\alpha=0$) describes the distribution of gene frequencies in a large neutral population at a locus, with the parameter $\theta>0$ taking on the interpretation of  the scaled population mutation rate \citep[Chapter 2]{Feng(10)}. The genetic interpretation of $\theta$ has motivated several works on the large $\theta$ asymptotic behaviour of the Poisson-Dirichlet distribution, and statistics thereof, providing Gaussian fluctuations and large (and moderate) deviations \citep{Wat(77),Gri(79),Joy(02),Daw(06),Fen(07),Fen(08)}. In particular, \citet{Fen(07)} first considered the regime $\theta=\lambda n$, with $\lambda>0$, in the study of the large $n$ asymptotic behaviour of the number $K_{n}$ of blocks in the Ewens random partition, providing a large deviation principle for $K_{n}$.  Some of the large $\theta$ asymptotic results developed for the Poisson-Dirichlet distribution have been extended to two-parameter Poisson-Dirichlet distribution (i.e. $\alpha\in(0,1)$), in terms of both Gaussian fluctuations and large deviations \citep{Fen(07a),Fen(10)}. Instead, nothing is known on the large $\theta$ asymptotic behaviour of the number $K_{n}$ of blocks in the Ewens-Pitman random partition, not even an extension of the large deviation principle of \citet[Theorem 4.1. and Theorem 4.4.]{Fen(07)}. The present paper covers this gap in the literature, providing  a LLN and a CLT for $K_{n}$ in the regime $\theta=\lambda n$, with $\lambda>0$.

\subsection{Organization of the paper}

The paper is structured as follows. Section \ref{sec2} contains the proof of Theorem \ref{thm_main} for $\alpha\in(0,1)$, with technical lemmas and propositions in Appendix \ref{app2}. Section \ref{sec3} contains the proof of Theorem \ref{thm_main} for $\alpha=0$ and the proof of the Berry-Esseen theorem \eqref{eq:be_dir}, with technical lemmas and propositions in Appendix \ref{app3}. Section \ref{sec4} contains the proof of the strong LLN  \eqref{lln_strong}. In Section \ref{sec5} we discuss some directions of future research, including the use of Theorem \ref{thm_main} in Bayesian nonparametric estimation of the unseen. 


\section{Proof of Theorem \ref{thm_main} for $\alpha\in(0,1)$}\label{sec2}

The proof of Theorem \ref{thm_main} is structured as follows: i) in Section \ref{sec21} we prove the asymptotic expansions \eqref{mom_m} and \eqref{mom_v}; ii) in Section \ref{sec22} we prove the LLN \eqref{lln} and \eqref{lln_strong}; iii) in Section \ref{sec23} we prove the CLT \eqref{clt}. Technical lemmas and propositions are deferred to Appendix \ref{app2}. 

\subsection{Asymptotic expansions \eqref{mom_m} and \eqref{mom_v}}\label{sec21}
To prove \eqref{mom_m}, we combine $\mathbb{E}[K_{n}]$ in the regime $\theta=\lambda n$ (see Equation \eqref{eq: mean_kn_gen_final}) with the asymptotic expansion for the Gamma function \citep[Equation 1]{TE(51)}. 
Precisely, this yields
\begin{align*}
\mathbb{E}\left[K_n\right]&= n \, \left\{\frac{\lambda}{\alpha} \left[-1 + \left(\frac{\lambda+1}{\lambda}\right)^\alpha + O \left(\frac{1}{n}\right) \right] \right\} = n\mathfrak{m}_{\alpha, \lambda}+ O(1),
\end{align*}
with
\begin{displaymath}
\mathfrak{m}_{\alpha, \lambda}=\frac{\lambda}{\alpha}\left[\left(1+\frac{1}{\lambda}\right)^{\alpha}-1\right].
\end{displaymath}
See Appendix \ref{app21} for details. Similarly, to prove \eqref{mom_v}, we combine $\operatorname{Var}(K_{n})$ in the regime $\theta=\lambda n$ (see Equation \eqref{eq: var_kn_gen_final}) with the asymptotic expansion for the Gamma function \citep[Equation 1]{TE(51)}. Precisely, this yields
\begin{align*}
&\operatorname{Var}(K_{n})\\
&\quad=\left[\left(\frac{\lambda n}{\alpha}\right)^2 + \frac{\lambda n}{\alpha}\right]\left[1 - 2 \left(\frac{\lambda+1}{\lambda} \right)^\alpha +\left(\frac{\lambda+1}{\lambda}\right)^\alpha  \frac{\alpha(\alpha-1)}{\lambda(\lambda+1) n } + \left(\frac{\lambda+1}{\lambda}\right)^{2\alpha} \right. \\
 &\quad\quad\quad \left.-\left(\frac{\lambda+1}{\lambda}\right)^{2\alpha}  \frac{\alpha(2\alpha-1)}{\lambda(\lambda+1) n} + O\left(\frac{1}{n^2}\right)\right]\\
&\quad\quad + \left(\frac{\lambda n}{\alpha}\right)  \left[-1 + \left(\frac{\lambda+1}{\lambda} \right)^\alpha+ O\left(\frac{1}{n}\right)\right] \\
&\quad\quad-\left(\frac{\lambda n}{\alpha}\right)^2  \left[ 1+ \left(\frac{\lambda+1}{\lambda}\right)^{2\alpha} - \left(\frac{\lambda+1}{\lambda}\right)^{2\alpha}  \frac{\alpha(\alpha -1)}{\lambda(\lambda+1) n} -2 \left(\frac{\lambda+1}{\lambda} \right)^\alpha \right. \\
&\quad\quad\quad \left.+ \left(\frac{\lambda+1}{\lambda} \right)^\alpha \frac{\alpha(\alpha-1)}{\lambda(\lambda+1) n} + O\left(\frac{1}{n^2}\right) \right]\\
&\quad= \left(\frac{\lambda n}{\alpha}\right) \cdot   \left(\frac{\lambda+1}{\lambda}\right)^\alpha\left[-1 + \left(\frac{\lambda+1}{\lambda}\right)^\alpha\cdot \left(1- \frac{\alpha}{\lambda+1}\right) \right] +O(1) \\
&\quad=n \mathfrak{s}_{\alpha, \lambda}^2 + O(1),
\end{align*}
with
\begin{displaymath}
\mathfrak{s}_{\alpha, \lambda}^2=\frac{\lambda}{\alpha}\left[\left(1+\frac{1}{\lambda}\right)^{2\alpha}\left(1-\frac{\alpha}{1+\lambda}\right)-\left(1+\frac{1}{\lambda}\right)^{\alpha}\right].
\end{displaymath}
See Appendix \ref{app21} for details. This completes the proof of \eqref{mom_m} and \eqref{mom_v}, respectively.

\subsection{LLN \eqref{lln}}\label{sec22}
To prove the LLN \eqref{lln}, it is useful to consider the following identity:
\begin{equation} \label{split_slutsky}
\frac{K_n - n \mathfrak{m}_{\alpha, \lambda}}{n} = \frac{K_n - \mathbb{E} \left[ K_n\right] }{n} + \frac{\mathbb{E} \left[ K_n\right] - n \mathfrak{m}_{\alpha, \lambda}}{n} .
\end{equation}
The second term in the right-end side is deterministic, and converges to 0 in view of \eqref{mom_m}. Then, fix $\varepsilon >0$ and combine Chebyshev's inequality 
with \eqref{mom_v} to obtain
\begin{align*}
P\left[\left| \frac{K_n - \mathbb{E}[K_n]}{n} \right|> \varepsilon \right]& = P \left[\left| K_n - \mathbb{E}[K_n] \right| > n\varepsilon \right]\le \frac{\operatorname{Var}(K_n)}{n^2 \varepsilon^2} = O\left(\frac{1}{n}\right)
\end{align*}
as $n \rightarrow +\infty$. The proof of the LLN \eqref{lln} is completed in view of \eqref{split_slutsky}.

\subsection{CLT \eqref{clt}}\label{sec23}

To prove the CLT \eqref{clt}, we recall the compound Poisson construction of $K_{n}$ \citep[Lemma 1 and Proposition 1]{Dol(20)}. For $\alpha\in(0,1)$ and $\theta>0$ let $S_{\alpha,\theta}$ be the scaled Mittag-Leffler random variable in \eqref{as_limit}. More precisely, $S_{\alpha,\theta}$ is a positive random variable with density function 
\begin{displaymath}
f_{S_{\alpha,\theta}}(s)=\frac{\Gamma(\theta)}{\Gamma\left(\frac{\theta}{\alpha}\right)}s^{\frac{\theta-1}{\alpha}-1}f_{\alpha}(s^{-1/\alpha})\qquad s>0,
\end{displaymath}
where $f_{\alpha}$ is the positive $\alpha$-stable density function \citep{Poll(46)}. Further information about $f_{\alpha}$ can be found in \citet[Sections 2.2 and 2.4]{Zol(86)}, while
\citet[Chapter 0]{Pit(06)} contains some details on $S_{\alpha,\theta}$. In particular, there hold
\begin{equation}
    \mathbb{E}\left[S_{\alpha, \theta}\right] =\frac{\theta}{\alpha}\frac{\Gamma (\theta )}{\Gamma(\alpha + \theta)} 
    \label{eq: mean stable}
\end{equation}
and
\begin{equation}
    \operatorname{Var}\left[S_{\alpha, \theta}\right] =  \frac{\theta}{\alpha} \left(\frac{\theta}{\alpha} +1\right) \frac{\Gamma(\theta)}{\Gamma(\theta+2\alpha)} - \left(\frac{\theta}{\alpha}\right)^2  \left(\frac{\Gamma(\theta)}{\Gamma(\theta+\alpha)}\right)^2.
    \label{eq: var stable}
\end{equation}

For $\alpha\in(0,1)$, $z>0$ and  $n\in\mathbb{N}$, we introduce a random variable $R(\alpha,n,z)$ with values in $\{1,\ldots,n\}$, whose distribution has probability mass function given by
\begin{equation}\label{dist_cha}
P[R(\alpha,n,z) = k] = \frac{\mathscr{C} (n, k; \alpha)z^k}{\sum_{j = 1}^n \mathscr{C} (n, j; \alpha)z^j} \qquad k\in\{1,\ldots,n\},
\end{equation}
where 
\begin{displaymath}
\mathscr{C}(n, k; \alpha)=\frac{1}{k!}\sum_{i=0}^{k}(-1)^{i}{k\choose i}[-i\alpha]_{(n)} \geq 0
\end{displaymath}
is the generalized factorial coefficient, with the proviso that $\mathscr{C}(0,0;\alpha):=1$ and $\mathscr{C}(n,0;\alpha):=0$ for $n\geq1$. See \citet[Chapter 2]{Cha(05)} for detailed account on $\mathscr{C}$. Specifically, \eqref{dist_cha} provides the distribution of the number of blocks in a random partition of the set $\{1,\ldots,n\}$ under the negative-binomial compound Poisson model \citep[Example 3.2]{Cha(07)}.  

Let $G_{a,b}$ be a Gamma random variable with shape $a>0$ and scale $b>0$, namely $G_{a,b}$ has density function $f_{a,b}(x) = [b^a/\Gamma(a)] x^{a-1} e^{-b x}$, for $x>0$.
For $\alpha\in(0,1)$ and $\theta>0$, \citet[Proposition 1]{Dol(20)} shows that for any $n\in\mathbb{N}$
\begin{equation}\label{DF_prop}
K_n \stackrel{d}{=} R\left(\alpha, n, Z_{\theta, n}\right),
\end{equation}
and
\begin{displaymath}
Z_{\theta, n} := S_{\alpha, \theta} G_{ \theta + n, 1}^\alpha
\end{displaymath}
such that $S_{\alpha, \theta}$ and $G_{\theta + n, 1}$ are independent random variables, and independent of $R(\alpha, n, z)$, for any $z>0$. See also \citet{Dol(21)} for details on \eqref{DF_prop}. Here and throughout,
$\stackrel{d}{=}$ denotes identity in distribution.  

The proof of the CLT \eqref{clt} relies on the distributional identity \eqref{DF_prop}. In particular, hereafter we consider \eqref{DF_prop} in the regime $\theta = \lambda n$. That is, for any $n\in\mathbb{N}$
\begin{equation}\label{eq:Kn_repres}
K_n \stackrel{d}{=} R\left(\alpha, n, Z_n\right),
\end{equation}
where
\begin{displaymath}
Z_n:= Z_{\lambda n, n} = S_{\alpha, \lambda n}G_{(\lambda +1)n, 1}^\alpha
\end{displaymath}
such that $S_{\alpha, \lambda n}$ and $G_{(\lambda +1)n, 1}$ are independent random variables, and independent of $R(\alpha, n, z)$, for any $z>0$. To simplify the notation, for any $n\in\mathbb{N}$ and $z>0$ we set
\begin{displaymath}
R_n(z) :=R(\alpha, n, nz).
\end{displaymath}
Hereafter, we investigate the large $n$ asymptotic behaviours of $Z_{n}$ and $R_n(z)$, as well as their interplay with respect to the asymptotic expansions displayed in \eqref{mom_m} and \eqref{mom_v}. In particular, the next proposition provides a LLN and CLT for $Z_n$.

\begin{prp}[LLN and CLT for $Z_n$]\label{prop:clt_Z} 
For any $\alpha\in(0,1)$ and $\lambda>0$, set
\begin{displaymath}
z_0 := \frac{\lambda}{\alpha} \left(\frac{\lambda+1}{\lambda}\right)^\alpha
\end{displaymath}
and
\begin{displaymath}
\Sigma^2 := \frac{\lambda}{\alpha} \left(\frac{\lambda+1}{\lambda}\right)^{2\alpha} \left(1- \frac{\alpha}{\lambda+1} \right).
\end{displaymath}
As $n \to +\infty$, there hold:
\begin{itemize}
\item[i)]
    \begin{equation}\label{mom_Zn1}
        \mathbb{E}[Z_n]= n z_0 + O(1) 
        \end{equation}
        and
            \begin{equation}\label{mom_Zn2}
\operatorname{Var}(Z_n)= n \Sigma^2 + O(1);
    \end{equation}
\item[ii)]
\begin{equation}
\label{eq:lln_zn}
\frac{Z_n}{n} \stackrel{p}{\longrightarrow} z_0;
\end{equation}
\item[iii)]
\begin{equation}
\label{eq:clt_zn}
\frac{Z_n - n z_0}{\sqrt{n}} \stackrel{w}{\longrightarrow} \mathcal{N}\left(0,\Sigma^2\right).
\end{equation}
\end{itemize}
\end{prp}

See Appendix \ref{sec: Zn} for the proof of Proposition \ref{prop:clt_Z}. Then, we state an analogous result for the variable $R_n(z)$, enriched with a Berry-Esseen estimate for the CLT.

\begin{prp}[LLN and Berry-Esseen theorem for $R_n(z)$]\label{prop:clt_R}
For any $\alpha\in(0,1)$ and $z>0$, let $\mu:  (0, +\infty) \to \mathbb{R}$ and $\sigma :  (0, +\infty) \to \mathbb{R}$ be functions defined as
\begin{displaymath}
\mu(z): = z\left(1-\frac{1}{\tau(z)}\right)
\end{displaymath}
and
 \begin{displaymath}
\sigma^2(z): = z \left( 1-\frac{1}{\tau(z)} - \frac{\alpha}{\alpha z + (1-\alpha) \tau(z)}\right) > 0
\end{displaymath}
where, for any $z>0$, $\tau(z)$ denotes the unique real, positive solution to the equation
\begin{equation} 
\label{eq:tauParis}
\tau(z)^\frac{1}{\alpha} = \frac{\tau(z)}{\alpha z}+1.
\end{equation}
As $n \to +\infty$, there hold:
\begin{itemize}
\item[i)]
    \begin{equation}\label{mom_Rnz1}
        \mathbb{E}[R_n(z)]= n\mu(z) + O(1) 
        \end{equation}
        and
            \begin{equation}\label{mom_Rnz2}
\operatorname{Var}(R_n(z))= n \sigma^2(z) + O(1);
    \end{equation}
\item[ii)]
\begin{equation}
\label{eq:lln_Rn}
\frac{R_n(z)}{n} \stackrel{p}{\longrightarrow} \mu(z).
\end{equation}
\end{itemize}
Finally, set
\begin{displaymath}
W_n(z): = \frac{R_n(z) - n \mu(z)}{\sqrt{n \sigma^2(z)}}
\end{displaymath}
and denote by $F_{W_n(z)}$ its distribution function.
Then, there exists a continuous function $C : (0, +\infty) \rightarrow (0, +\infty)$ such that, for any choice of $\zeta_0 $ and $\zeta_1$ that satisfy $0< \zeta_0 < z_0 <\zeta_1 <+\infty$, 
with the same $z_0$ as in Proposition \ref{prop:clt_Z}, the inequality
\begin{equation}
\label{eq:be_rn}
\left\|F_{W_n(z)}- \Phi \right\|_\infty \le \frac{C(z) \, \log(n)}{n^{1/8}}
\end{equation}
holds for every $z \in \left[\zeta_0, \zeta_1\right]$ and every $n \ge \bar{n}(\zeta_0, \zeta_1)$. In particular, the Berry-Esseen bound \eqref{eq:be_rn} implies that $W_n(z)\stackrel{w}{ \longrightarrow }\mathcal{N}(0,1)$ as $n \rightarrow + \infty$, for every fixed $z >0$.
\end{prp}

See Appendix \ref{sec:BE_Rn} for the proof of Proposition \ref{prop:clt_R}. The next proposition shows how Proposition \ref{prop:clt_Z} and Proposition \ref{prop:clt_R} interplay with the asymptotic expansions \eqref{mom_m} and \eqref{mom_v}.

\begin{prp}\label{obs:limits_are_coherent}
Under the assumptions of Proposition \ref{prop:clt_Z} and Proposition \ref{prop:clt_R}, there hold
\begin{displaymath}
\mu(z_0 ) = \mathfrak{m}_{\alpha, \lambda}
\end{displaymath}
and
\begin{displaymath}
\sigma^2(z_0) + \Sigma^2\left(\mu'(z_0)\right)^2 = \mathfrak{s}_{\alpha, \lambda}^2
\end{displaymath}
\end{prp}
\begin{proof}
Set $\tau_0: = \left(\frac{\lambda+1}{\lambda}\right)^\alpha$. First, 
we prove that $ \tau (z_0) = \tau_0$ by checking that $\tau_0$ is a (and therefore the unique) positive solution to \eqref{eq:tauParis} when $z = z_0$. In particular, it holds
\begin{displaymath}
 \frac{\tau_0}{\alpha z_0} +1 = \frac{\alpha}{\alpha \lambda} \left(\frac{\lambda+1}{\lambda}\right)^\alpha \left(\frac{\lambda}{\lambda+1}\right)^\alpha +1 = \frac{1}{\lambda} + 1 = \frac{\lambda+1}{\lambda}= \tau_0^\frac{1}{\alpha}.
\end{displaymath}
The first  identity then follows by evaluating $\mu$ at $z_0$.
For the second identity, differentiate both sides of \eqref{eq:tauParis} with respect to $z$ and rearrange to obtain that for $z>0$
\begin{displaymath}
\tau'(z) = - \frac{\tau(z)}{z D(z)},
\end{displaymath}
where
\begin{equation}
\label{eq:defD}
D(z) := z\tau(z)^{\frac{1-\alpha}{\alpha}} -1 = \frac{\alpha z + (1-\alpha) \tau(z)}{\alpha \tau(z)}.
\end{equation}
Whence,
\begin{displaymath}
\mu'(z) = z \frac{\tau(z)'}{\tau(z)^2} + 1-\frac{1}{\tau(z)}= - \frac{1}{\tau(z)D(z)} + 1-\frac{1}{\tau(z)}. 
\end{displaymath}
A simple computation yields
    \begin{displaymath}
         D(z_0)  
         = - \left( 1- \frac{\lambda+1}{\alpha}\right)
     \end{displaymath}
and, in turn,
     \begin{displaymath}
         \mu'(z_0) = 1- \frac{\lambda+1}{\tau_0 (\lambda+1-\alpha)}.
    \end{displaymath}
Thus,we can write
      \begin{align*}
      &   \sigma^2(z_0) + \Sigma^2 \cdot \left(\mu'(z_0)\right)^2 \\
      & \quad \quad =   \frac{\lambda}{\alpha} \tau_0 \left( 1-\frac{1}{\tau_0} - \frac{\alpha}{ (\lambda+ 1-\alpha) \tau_0}\right)  - \frac{\lambda}{\alpha} \tau_0^2 \left( 1- \frac{\alpha}{\lambda+1}\right) \left(1 - \frac{\lambda+1}{\tau_0 (\lambda+1-\alpha)}\right)^2\\
 & \quad \quad =   \frac{\lambda}{\alpha} \left\{ \left(\frac{\lambda+1}{\lambda}\right)^{2\alpha} \left(1- \frac{\alpha}{\lambda+1} \right) - \left(\frac{\lambda+1}{\lambda}\right)^\alpha \right\} \\
 & \quad \quad  = \mathfrak{s}_{\alpha, \lambda}^2.
         \end{align*}
\end{proof}

Now, we show how Propositions \ref{prop:clt_Z}, \ref{prop:clt_R} and \ref{obs:limits_are_coherent} can be used to prove the CLT \eqref{clt}. Denoting by $F_{n}$ the cumulative distribution function of the random variable
\begin{displaymath}
\frac{K_n - n\mathfrak{m}_{\alpha, \lambda}}{\sqrt{n\mathfrak{s}_{\alpha, \lambda}^2}}, 
\end{displaymath}
we prove that, for any $x \in \R$, 
\begin{displaymath}
\lim_{n \rightarrow + \infty } F_n(x) := \lim_{ n \rightarrow + \infty } P\left[K_n \le n \mathfrak{m}_{\alpha, \lambda} + \sqrt{n\mathfrak{s}_{\alpha, \lambda}^2} x \right] = \Phi(x).
\end{displaymath}
Denote by $\mu_{\frac{Z_n}{n}}$ the probability distribution of $Z_n/n$. 
By standard properties of conditional probability, we can rewrite the distributional identity \eqref{eq:Kn_repres} as
\begin{align*}
F_n(x) & =\int_0^{+\infty} P\left[R_n(z) \le n \mathfrak{m}_{\alpha, \lambda} + \sqrt{n\mathfrak{s}_{\alpha, \lambda}^2}x\right] \ \mu_{\frac{Z_n}{n}}(\mathrm{d} z) \\
& = \int_0^{+\infty} P\left[W_n(z) \le \frac{\sqrt{n}\, \left[\mathfrak{m}_{\alpha, \lambda} - \mu(z)\right] + \mathfrak{s}_{\alpha, \lambda}x}{\sigma(z)}\right] \ \mu_{\frac{Z_n}{n}}(\mathrm{d} z).
\end{align*}
Whence, $F_n(x) = \mathcal{I}^{(n)}_1(x) + \mathcal{I}^{(n)}_2(x)$, where
\begin{displaymath}
\mathcal{I}^{(n)}_1(x) := \int_0^{+\infty}\Phi\left( \frac{\sqrt{n}\, \left[\mathfrak{m}_{\alpha, \lambda} - \mu(z)\right] + \mathfrak{s}_{\alpha, \lambda}x}{\sigma(z)}\right)\ \mu_{\frac{Z_n}{n}}(\mathrm{d} z) 
\end{displaymath} 
and
\begin{align*}
\mathcal{I}^{(n)}_2(x) &:= \int_0^{+\infty} \left\{F_{W_n(z)}\left( \frac{\sqrt{n}\, \left[\mathfrak{m}_{\alpha, \lambda} - \mu(z)\right] + \mathfrak{s}_{\alpha, \lambda}x}{\sigma(z)}\right) \right.\\
&\quad\quad\quad\quad\quad \left. - \Phi\left( \frac{\sqrt{n}\, \left[\mathfrak{m}_{\alpha, \lambda} - \mu(z)\right] + \mathfrak{s}_{\alpha, \lambda}x}{\sigma(z)}\right)\right\}\ \mu_{\frac{Z_n}{n}}(\mathrm{d} z).
\end{align*}
The proof of the CLT \eqref{clt} is completed by showing that, for every $x\in\mathbb{R}$
\begin{equation}\label{rv_1}
\lim_{n \to + \infty} \mathcal{I}^{(n)}_1(x) =\Phi(x)
\end{equation}
and
\begin{equation}\label{rv_2}
\lim_{n \to + \infty} \mathcal{I}^{(n)}_2(x)=0.
\end{equation}

To prove \eqref{rv_1}, we premise two technical lemmas.
\begin{lem}
\label{lem:lem1_ter}
If $Y$ is a Gaussian random variable with mean $0$ and variance $\Sigma^2$, then
\begin{displaymath}
\mathbb{E} \left[\Phi\left(\frac{\mu'(z_0) Y + \mathfrak{s}_{\alpha, \lambda}x}{\sigma(z_0)} \right)\right] = \Phi(x)
\end{displaymath}
holds for every $x\in \mathbb{R}$.
\end{lem}
\begin{proof}
Introduce a standard Gaussian random variable $Z$, independent of $Y$. By standard properties of conditional probability, it holds that
\begin{displaymath}
\Phi\left(\frac{\mu'(z_0) Y + \mathfrak{s}_{\alpha, \lambda} x}{\sigma(z_0)} \right) = P \left[ Z \le \frac{\mu'(z_0) Y + \mathfrak{s}_{\alpha, \lambda}x}{\sigma(z_0)} \ \bigg| \ Y \right],
\end{displaymath}
which implies
\begin{displaymath}
\mathbb{E}\left[ \Phi\left(\frac{\mu'(z_0) Y + \mathfrak{s}_{\alpha, \lambda}x}{\sigma(z_0)} \right)\right]= P \left[ \frac{\sigma(z_0)Z - \mu'(z_0) Y}{\mathfrak{s}_{\alpha, \lambda}} \le x\right].
\end{displaymath}
Notice that $\mathfrak{s}_{\alpha, \lambda}^{-1}[\sigma(z_0)Z - \mu'(z_0) Y]$ has Gaussian distribution with mean $0$ and variance
\begin{displaymath}
 \frac{\sigma(z_0)^2+ (\mu'(z_0))^2 \Sigma^2 }{\mathfrak{s}_{\alpha, \lambda}^2} =1,
\end{displaymath}
thanks to the last identity in Proposition \ref{obs:limits_are_coherent}. This completes the proof.
\end{proof}

\begin{lem}
\label{lem:lemma2_ter}
Let $\psi \in C^0([0, +\infty)) \cap C^1((0, +\infty))$, with bounded (first) derivative. Then, as $n \to + \infty$ there holds
\begin{displaymath}
\sqrt{n} \left[\psi \left(\frac{Z_n}{n}\right) - \psi(z_0) \right] \stackrel{\text{w}}{\longrightarrow}\mathcal{N}\left(0, (\psi'(z_0))^2 \Sigma^2\right)
\end{displaymath}
\end{lem}
\begin{proof}
From the fundamental theorem of calculus, we write that
\begin{equation}\label{part21}
\sqrt{n} \left[\psi \left(\frac{Z_n}{n}\right) - \psi(z_0) \right] = \sqrt{n} \left(\frac{Z_n}{n} - z_0\right)\int_0^1 \psi'\left(z_0 + t \left[ \frac{Z_n}{n} - z_0 \right]\right) \mathrm{d}t.
\end{equation}
According to \eqref{eq:lln_zn}, $n^{-1}Z_{n}-z_{0} \stackrel{p}{\longrightarrow}0$, as $n\rightarrow+\infty$. Since $\psi'$ is bounded, a combination of the dominated convergence theorem with the continuous mapping 
for the convergence in probability entails that, as $n\rightarrow+\infty$
\begin{equation}\label{part22}
\int_0^1 \psi'\left(z_0 + t \left[ \frac{Z_n}{n} - z_0 \right]\right) \mathrm{d}t \stackrel{p}{\longrightarrow} \psi'(z_0).
\end{equation}
By taking into account \eqref{eq:clt_zn} and \eqref{part22}, the conclusion follows from a straightforward application of Slutsky's theorem to the right-hand side of \eqref{part21}. 
\end{proof}

The next proposition combines the two preceding lemmata to obtain the desired behaviour for $\mathcal{I}_1^{(n)}$ in \eqref{rv_1}.
\begin{prp}\label{lem:lem3_ter} 
For every $x \in \mathbb{R}$, 
\begin{displaymath}
\lim_{n \to + \infty} \mathcal{I}^{(n)}_1(x) :=\lim_{n \rightarrow+\infty} \mathbb{E} \left[ \Phi \left(\frac{\sqrt{n} \left[\mathfrak{m}_{\alpha, \lambda} - \mu\left(\frac{Z_n}{n}\right)\right] 
+\mathfrak{s}_{\alpha, \lambda}x}{\sigma\left(\frac{Z_n}{n}\right)}\right)\right] =\Phi(x).
\end{displaymath}
\end{prp}
\begin{proof}
According to Proposition \ref{obs:limits_are_coherent}, we have $\mathfrak{m}_{\alpha, \lambda} = \mu(z_0)$. Now, it is useful to observe that the function $\mu$ considered in Proposition \ref{prop:clt_R} belongs to $C^0([0, +\infty)) \cap C^1((0, +\infty))$, and has bounded (first) derivative. Therefore, we can apply Lemma \ref{lem:lemma2_ter} for the choice $\psi=\mu$, to obtain that, as $n\rightarrow+\infty$
\begin{displaymath}
\sqrt{n} \left[ \mu(z_0) - \mu\left(\frac{Z_n}{n}\right)\right]\stackrel{w}{\longrightarrow} \mu'(z_0)Y\ ,
\end{displaymath}
where the random variable $Y$ is the same random variable as in Lemma \ref{lem:lem1_ter}.
Further, thanks to the continuity of $\sigma$ and the mapping theorem for the convergence in probability, the LLN \eqref{eq:lln_zn} entails that, as $n\rightarrow+\infty$
\begin{displaymath}
\sigma\left(\frac{Z_n}{n}\right)\stackrel{p}{\longrightarrow} \sigma(z_0).
\end{displaymath}
At this stage, Slutsky's theorem shows that, as $n\rightarrow+\infty$
\begin{displaymath}
\frac{\sqrt{n} \left[ \mu\left(\frac{Z_n}{n}\right) - \mathfrak{m}_{\alpha, \lambda}\right] + \mathfrak{s}_{\alpha, \lambda}x}{\sigma\left(\frac{Z_n}{n}\right)}\stackrel{w}{\longrightarrow}\frac{\mu'(z_0)Y + \mathfrak{s}_{\alpha, \lambda}x}{\sigma\left(z_0\right)}.
\end{displaymath}
Since the (cumulative distribution) function $\Phi$ is bounded and continuous, the proof is completed by applying the Portmanteau theorem and recalling Lemma \ref{lem:lem1_ter}. 
\end{proof}

The next proposition proves \eqref{rv_2} by means of an application of Proposition \ref{prop:clt_R}.
\begin{prp}\label{lem:lem4_ter}
For every $x \in \mathbb{R}$, as $n\rightarrow+\infty$, it holds
\begin{displaymath}
\left| \mathcal{I}^{(n)}_2(x) \right| \le \int_0^{+\infty} \left\|F_{W_n(z)} - \Phi \right\|_{\infty} \mu_{\frac{Z_n}{n}}(\mathrm{d} z)\rightarrow 0.
\end{displaymath}
\end{prp}
\begin{proof}
The inequality is straightforward, as it follows directly by taking the absolute value inside the integral in $\mathcal{I}^{(n)}_2$. Then, in order to prove convergence to zero, fix $\varepsilon>0$ and choose $\delta = \delta(\varepsilon)>0$ such that $\Phi(\delta) = 1 - \varepsilon/2$. Moreover, fix $\zeta_0$ and $\zeta_1$ as in Proposition \ref{prop:clt_R}.
Set
\begin{displaymath}
\bar{m} := \bar{m}(\varepsilon, \zeta_0, \zeta_1): = \min \left\{m \in \mathbb{N}\text{ : } z_0 - \frac{\delta \Sigma}{\sqrt{m}} >\zeta_0,\text{ and } z_0 - \frac{\delta \Sigma}{\sqrt{m}} <\zeta_1 \right\},
\end{displaymath}
which is well-defined since $z_0 \in (\zeta_0, \zeta_1)$. Set $\tilde{\zeta}_0 := z_0 - \frac{\delta \Sigma}{\sqrt{\bar{m}}}$ and $\tilde{\zeta}_1 := z_0 + \frac{\delta \Sigma}{\sqrt{\bar{m}}} $, and write
\begin{align}\label{par31}
&\int_0^{+\infty} \left\|F_{W_n(z)} - \Phi \right\|_{\infty} \mu_{\frac{Z_n}{n}}(\mathrm{d} z)\\
&\notag\quad = \int_{\tilde{\zeta}_0}^{\tilde{\zeta}_1} \left\|F_{W_n(z)} - \Phi \right\|_{\infty} \mu_{\frac{Z_n}{n}}(\mathrm{d} z) + \int_{\mathbb{R} \smallsetminus [\tilde{\zeta}_0, \tilde{\zeta}_1]} \left\|F_{W_n(z)} - \Phi \right\|_{\infty} \mu_{\frac{Z_n}{n}}(\mathrm{d} z). 
\end{align}
Now, consider the terms on the right-hand side of \eqref{par31} separately. With regard to the first term, inequality \eqref{eq:be_rn} entails that, for every $n \ge \bar{n} $, there hold
\begin{equation}\label{par32}
\int_{\tilde{\zeta}_0}^{\tilde{\zeta}_1}\left\|F_{W_n(z)} - \Phi \right\|_{\infty} \mu_{\frac{Z_n}{n}}(\mathrm{d} z) \le \frac{\log(n)}{n^\frac{1}{8}} \int_{\tilde{\zeta_0}}^{\tilde{\zeta}_1}C(z) \ \mu_{\frac{Z_n}{n}}(\mathrm{d} z) \le \frac{\log(n)}{n^\frac{1}{8} }\, \mathcal{M}_C,
\end{equation}
where $\mathcal{M}_C: = \max_{z \in [\zeta_0, \zeta_1]} C(z)$ is finite since the function $C$ is continuous. Concerning the second term on the right-hand side of \eqref{par31}, for any $n \ge \bar{m}$, there hold
\begin{align}\label{par33}
\int_{\mathbb{R} \smallsetminus [\tilde{\zeta}_0, \tilde{\zeta}_1]} \left\|F_{W_n(z)} - \Phi \right\|_{\infty} \mu_{\frac{Z_n}{n}}(\mathrm{d} z) 
& \le \int_{\mathbb{R} \smallsetminus [\tilde{\zeta}_0, \tilde{\zeta}_1]}\mu_{\frac{Z_n}{n}}(\mathrm{d} z) \\
&\notag = P\left[\frac{Z_n}{n} \notin \left[z_0 - \frac{\delta \Sigma}{\sqrt{\bar{m}}}, z_0 + \frac{\delta \Sigma}{\sqrt{\bar{m}}} \right] \right]\\
&\notag\le P\left[\frac{Z_n}{n} \notin \left[z_0 - \frac{\delta \Sigma}{\sqrt{n}}, z_0 + \frac{\delta \Sigma}{\sqrt{n}} \right] \right]\\
&\notag = P \left[\frac{Z_n- n z_0}{\sqrt{n} \Sigma} \notin [-\delta, \delta]\right].
\end{align}
Thus, combining identity \eqref{par31} with inequalities \eqref{par32}--\eqref{par33} yields
\begin{displaymath}
\int_0^{+\infty} \left\|F_{W_n(z)} - \Phi \right\|_{\infty} \mu_{\frac{Z_n}{n}}(\mathrm{d} z) \le \frac{\log(n)}{n^\frac{1}{8}}\, \mathcal{M}_C + P \left[\frac{Z_n- n z_0}{\sqrt{n} \Sigma} \notin [-\delta, \delta]\right]
\end{displaymath}
for any $n \ge \max(\bar{n}, \bar{m})$. Whence, from \eqref{eq:clt_zn}, 
\begin{align*}
0  & \le \limsup_{n \to + \infty} \int_0^{+\infty} \left\|F_{W_n(z)} - \Phi \right\|_{\infty} \mu_{\frac{Z_n}{n}}(\mathrm{d} z) \\
& \le \lim_{n \to + \infty} \left\{\frac{\log(n)}{n^\frac{1}{8}}\, \mathcal{M}_C + P \left[\frac{Z_n- n z_0}{\sqrt{n} \Sigma} \notin [-\delta, \delta]\right]\right\}\\
& = 0+ 2-2 \Phi(\delta) = \varepsilon,
\end{align*}
thanks to $\Phi(\delta) = 1 - \varepsilon/2$. The proof is completed by exploiting the arbitrariness of $\varepsilon$. 
\end{proof}


\section{Proof of Theorem \ref{thm_main} for $\alpha=0$}\label{sec3}

The proof of Theorem \ref{thm_main} is structured as follows: i) in Section \ref{sec31} we prove the asymptotic expansions \eqref{mom_m} and \eqref{mom_v}; ii) in Section \ref{sec32} we prove the LLN \eqref{lln}; iii) in Section \ref{sec33} we prove the CLT \eqref{clt}. Technical lemmas and propositions are deferred to Appendix \ref{app3}. 

\subsection{Asymptotic expansions \eqref{mom_m} and \eqref{mom_v}}\label{sec31}

To prove \eqref{mom_m} and \eqref{mom_v}, we make use of equations \eqref{app_probgen11} and \eqref{app_probgen21} for the derivatives (first and second derivatives) of the probability generating function $G_{K_{n}}$ of $K_{n}$ in the regime $\theta=\lambda n$, together with the asymptotic expansions of both the Digamma and Trigamma functions for large argument \citep[Equation \href{http://dlmf.nist.gov/5.11.E2} {(5.11.2)}]{nist}. This yields
\begin{align*}
 G_{K_n}'(1)&:=\left.\frac{\ddr }{\ddr s}G_{K_{n}}(s)\right|_{s=1}\\
 & = \lambda n \left[  \log((\lambda+1) n ) - \frac{1}{2 (\lambda +1) n}  -  \log(\lambda  n ) + \frac{1}{2 \lambda  n} + O\left(\frac{1}{n^2} \right) \right]	\\	
 & = n  \lambda \log \left(\frac{\lambda+1}{\lambda}\right) + \frac{1}{2(\lambda+1)}  + O\left(\frac{1}{n} \right) \\
 & =: n \, \mathcal{A} + \mathcal{B} + O\left(\frac{1}{n} \right) 
 \end{align*}
 and 
\begin{align*}
 G_{K_n}''(1)&:=\left.\frac{\ddr^{2} }{\ddr s^{2}}G_{K_{n}}(s)\right|_{s=1}\\
 & = (\lambda n)^2\,  \left\{ \frac{1}{(\lambda+1) n } - \frac{1}{\lambda n} + O\left(\frac{1}{n^2}\right) \right.\\
 &\quad\quad\quad\quad\left.   + \left[ \log \left(\frac{\lambda+1}{\lambda} \right) + \frac{1}{2\lambda (\lambda+1) n }+ O\left(\frac{1}{n^2}\right)\right]^2 \right\} \\
 & = n^2  \lambda^2 \log^2 \left(\frac{\lambda+1}{\lambda}\right) + n  \frac{\lambda+1}{\lambda} \, \left[ \log \left(\frac{\lambda+1}{\lambda}\right) -1 \right] + O(1)\\
  & =: n^2 \, \mathcal{C} + n\mathcal{D} + O(1).
\end{align*}
See Appendix \ref{app31} for details. By properties of the probability generating function,  
\begin{displaymath}
 \mathbb{E}[K_n] = G'_{K_n}(1) = n\,  \mathcal{A} + O(1)\\
\end{displaymath}
and 
\begin{align*}
    \operatorname{Var}\left(K_{n}\right) &= G''_{K_n}(1) + G'_{K_n} (1) - \left(G'_{K_n} (1)\right)^2\\
    &= n^2 \, \left[  \mathcal{C}- \mathcal{A}^2  \right] + n\, \left[ \mathcal{D} + \mathcal{A} - 2\mathcal{A} \mathcal{B} \right] + O(1).
\end{align*}
At this stage, it is enough to show that $\mathcal{A}= \mathfrak{m}_{0, \lambda} $, $\mathcal{C} = \mathfrak{m}_{0, \lambda}^2$ and $\mathcal{D} + \mathfrak{m}_{0, \lambda} - 2\mathfrak{m}_{0, \lambda} \mathcal{B} = \mathfrak{s}_{0, \lambda}^2$. The first two identities are immediate. The third identity can be verified via direct computation. Since $\mathcal{C}- \mathcal{A}^2 =0$, the proof of \eqref{mom_m} and \eqref{mom_v} is completed.

\subsection{LLN \eqref{lln}}\label{sec32}

The proof of the LLN \eqref{lln} is the same as the proof for $\alpha\in(0,1)$.

\subsection{CLT \eqref{clt}}\label{sec33}

To prove the CLT \eqref{clt}, let $\varphi_n$ be the characteristic function of the random variable 
\begin{displaymath}
\frac{K_n - n\mathfrak{m}_{0, \lambda}}{\sqrt{n \mathfrak{s}_{0, \lambda}^2} }.
\end{displaymath}
The proof of the CLT \eqref{clt} follows by the application of the following Berry-Esseen lemma.
\begin{lem}
\label{lem: BE lem_dir}
If $\xi \in \mathbb{R}$ satisfies \begin{equation}
    \label{eq: BE_xi_bound_dir}
    |\xi|\le \mathcal{C}_0\,  \mathfrak{s}_{0, \lambda} \, n^\delta
\end{equation} 
for a positive constant $\mathcal{C}_0$ and some $\delta \in (0, 1/6)$,  then there exists a constant $\tilde{c}$ such that
\begin{displaymath}
    \left|\varphi_{n} (\xi) - e^{- \frac{\xi^2}{2}} \right| \le \tilde{c}\,  e^{ - \frac{\xi^2}{2}}\, n^{3\delta - \frac{1}{2}}. 
\end{displaymath}
\end{lem}
\begin{proof}
We start by recalling a well-known approximation for the Gamma function \citep[Equation \href{https://dlmf.nist.gov/5.11.E10}{(5.11.10)} and Equation \href{https://dlmf.nist.gov/5.11.E11}{(5.11.11)} with $K=1$ therein]{nist}, which will be applied to $G_{K_{n}}$. For $w \in \mathbb{C}$ and $|\text{ph}(w)|<\pi$, as $|w| \to +\infty$
\begin{displaymath}
\Gamma(w) = e^{-w} \, w^w \, \left(\frac{2\pi}{w}\right)^{\frac{1}{2}} \left[1 + \mathfrak{R}(w) \right] =: \mathcal{P}(w) \, \left[1 + \mathfrak{R}(w) \right]
\end{displaymath} 
where, for $|\text{ph}(w)|\le\frac{\pi}{3}$,
\begin{equation}
\label{eq:gammabound}
\left| \mathfrak{R}(w) \right|  \le \frac{1}{2\pi^2 |w| } \left[1 + \min\left\{\sec(\text{ph}(w)), 2\right\}\right] \le \frac{3}{2\pi^2 |w| }. 
\end{equation}
Then,
\begin{align*}
G_{K_n}(s) & = \frac{\Gamma(n(s\lambda+1))}{\Gamma(ns\lambda)}  \frac{\Gamma(n\lambda )}{\Gamma(n(\lambda+1))}\\
& = \frac{\mathcal{P}\left( n(s\lambda+1)\right) \,  \left[1 + \mathfrak{R}(n(s\lambda+1)) \right] \,  \, \mathcal{P}\left( n\lambda \right) \,  \left[1 + \mathfrak{R}(n \lambda) \right]}{\mathcal{P}( n s\lambda) \, \left[1 + \mathfrak{R}(n s\lambda) \right] \,  \, \mathcal{P}\left( n(\lambda+1) \right) \,  \left[1 + \mathfrak{R}(n(\lambda+1)) \right]}\\
& = \left(\frac{ s(\lambda+1)}{s\lambda+1}\right)^\frac{1}{2}  \, \left[ \frac{(s\lambda+1)^{s\lambda+1} \, \lambda^\lambda}{(s\lambda)^{s\lambda} \, (\lambda+1)^{\lambda+1}}\right]^n  \frac{\left[1 + \mathfrak{R}(n (s \lambda+1) ) \right] \, \left[1 + \mathfrak{R}(n \lambda) \right]}{\left[1 + \mathfrak{R}(n s \lambda) \right]\, \left[1 + \mathfrak{R}(n( \lambda+1)) \right]} 
\end{align*}
for every $s \in \mathbb{C}$ such that $|\text{ph}(s)|\le \pi$. Based on the above expression for $G_{K_n}$, we set
\begin{displaymath}
 \mathfrak{R}_1^{(n)}(s) := \frac{\left[1 + \mathfrak{R}(n (s \lambda+1) ) \right] \, \left[1 + \mathfrak{R}(n \lambda) \right]}{\left[1 + \mathfrak{R}(n s \lambda) \right]\, \left[1 + \mathfrak{R}(n( \lambda+1)) \right]}, 
 \end{displaymath}
 \begin{displaymath}
 \mathfrak{R}_2^{(n)}(s) : = \left(\frac{ s(\lambda+1)}{s\lambda+1}\right)^\frac{1}{2} 
  \end{displaymath}
and 
 \begin{displaymath}
 f(s) : =  s\lambda  \log \left(s\lambda\right) -  (s\lambda + 1) \log\left(s\lambda+1\right)\ ,
   \end{displaymath}
so that $G_{K_n}(s)  =   \exp\{-n \left[f(s) - f(1)\right]\}   \mathfrak{R}_1^{(n)}(s)  \mathfrak{R}_2^{(n)}(s)$. Accordingly, we write 
\begin{align*}
 \varphi_n(\xi)  & = \exp\left\{-\sqrt{n} \, \frac{i\, \xi\,  \mathfrak{m}_{0, \lambda}}{\mathfrak{s}_{0, \lambda}} \right\} \, G_{K_n} \left(e^\frac{i \, \xi }{\sqrt{n} \mathfrak{s}_{0, \lambda}}\right)  \\
    & =  \exp\left\{-\sqrt{n} \, \frac{i\, \xi\,  \mathfrak{m}_{0, \lambda}}{\mathfrak{s}_{0, \lambda}} \right\} \,   \exp\left\{-n \left[f\left(e^\frac{i \, \xi }{\sqrt{n} \mathfrak{s}_{0, \lambda}}\right) - f(1)\right] \right\}  \\
    &\quad\times\mathfrak{R}_1^{(n)}\left(e^\frac{i \, \xi }{\sqrt{n} \mathfrak{s}_{0, \lambda}} \right)  \mathfrak{R}_2^{(n)}\left(e^\frac{i \, \xi }{\sqrt{n} \mathfrak{s}_{0, \lambda}} \right).
    \end{align*}
If $\xi\in\mathbb{R}$ satisfies \eqref{eq: BE_xi_bound_dir}, then \citet[Chapter IV, Lemma 5]{Pet(75)} guarantees that
\begin{displaymath}
    \left|e^\frac{i \, \xi }{\sqrt{n} \mathfrak{s}_{0, \lambda}}-1\right| \le \left| \frac{\xi }{\sqrt{n} \mathfrak{s}_{0, \lambda}}\right| \le \mathcal{C}_0 \, n^{\delta - \frac{1}{2}}.
\end{displaymath} 
Hence, by applying Taylor's formula to $f (e^\frac{i \, \xi }{\sqrt{n} \mathfrak{s}_{0, \lambda}} ) - f(1)$ and to $(e^\frac{i \, \xi }{\sqrt{n} \mathfrak{s}_{0, \lambda}} - 1)$, we obtain
\begin{align}\label{cart_parz}
    \varphi_{n}(\xi) &  = \exp\left\{ - \sqrt{n} \, \frac{i \, \xi }{\mathfrak{s}_{0, \lambda}} \left[  \mathfrak{m}_{0, \lambda}  + f'(1) \right] - \frac{ \, \xi^2 }{2 \, \mathfrak{s}_{0, \lambda}^2} \left[ f'(1) + f''(1)\right]\right\}\times \\
    &\notag\quad\times \mathfrak{R}_1^{(n)}\left( e^\frac{i \, \xi }{\sqrt{n} \mathfrak{s}_{0, \lambda}}\right) \, \mathfrak{R}_2^{(n)}\left(e^\frac{i \, \xi }{\sqrt{n} \mathfrak{s}_{0, \lambda}}\right) \, \mathfrak{R}_3^{(n)}\left( \xi\right),
\end{align}
where 
\begin{align*}
   \mathfrak{R}_3^{(n)}\left( \xi\right) & := \exp\left\{ n \, f'(1) \left[\left(e^\frac{i \, \xi }{\sqrt{n} \mathfrak{s}_{0, \lambda}} - 1\right) -\left(\frac{i \, \xi }{\sqrt{n} \mathfrak{s}_{0, \lambda}} - \frac{\xi^2 }{n \mathfrak{s}_{0, \lambda}^2}\right) \right] \right.\\
   &\quad + \frac{n}{2}\, f''(1) \left[ \left(e^\frac{i \, \xi }{\sqrt{n} \mathfrak{s}_{0, \lambda}} - 1\right)^2 + \frac{\xi^2}{n \mathfrak{s}_{0, \lambda}^2}\right]\\
   &\quad \left. + \frac{n}{2 } \left(e^\frac{i \, \xi }{\sqrt{n} \mathfrak{s}_{0, \lambda}} - 1\right)^3 \int_0^1 f'''\left(1 + t \left(e^\frac{i \, \xi }{\sqrt{n} \mathfrak{s}_{0, \lambda}} - 1\right)\right) \, (1-t)^2 \, \mathrm{d} t\right\}.
\end{align*}
By direct computation,
\begin{displaymath}
 f'(1) = - \lambda \log\left(\frac{\lambda+1}{\lambda}\right) = -\mathfrak{m}_{0, \lambda}
 \end{displaymath}
 and
 \begin{displaymath}
f'(1) + f''(1) = -\mathfrak{m}_{0, \lambda} + \frac{\lambda}{\lambda+1} = -\mathfrak{s}_{0, \lambda}^2.
\end{displaymath}
Accordingly, for every $\xi\in\mathbb{R} $ that satisfies \eqref{eq: BE_xi_bound_dir}, we can rewrite $ \varphi_{n}$ in \eqref{cart_parz} as
\begin{displaymath}
   \varphi_{n}(\xi) = \exp\left\{- \frac{ \, \xi^2 }{2 } \right\} \, \mathfrak{R}_1^{(n)}\left(e^\frac{i \, \xi }{\sqrt{n} \mathfrak{s}_{0, \lambda}}\right) 
   \, \mathfrak{R}_2^{(n)}\left(e^\frac{i \, \xi }{\sqrt{n} \mathfrak{s}_{0, \lambda}}\right) \, \mathfrak{R}_3^{(n)}\left(\xi \right).
\end{displaymath}
By combining \eqref{eq:gammabound} and \citet[Chapter IV, Lemma 5]{Pet(75)}, for any $\xi$ satisfying \eqref{eq: BE_xi_bound_dir}, one has that
\begin{displaymath}
\mathfrak{R}^{(n)}\left(\xi\right)  := \mathfrak{R}_1^{(n)}\left(e^\frac{i \, \xi }{\sqrt{n} \mathfrak{s}_{0, \lambda}}\right) \, \mathfrak{R}_2^{(n)}\left(e^\frac{i \, \xi }{\sqrt{n} \mathfrak{s}_{0, \lambda}}\right) \, \mathfrak{R}_3^{(n)}\left(\xi \right),
\end{displaymath}
is continuous and satisfies, as $n \rightarrow +\infty$
\begin{equation}
\label{eq: remainders_dir}
    \mathfrak{R}^{(n)}\left(\xi\right) = 1 + O\left(n^{3 \delta - \frac{1}{2}}\right)
\end{equation}
uniformly on compact sets. We refer to Appendix \ref{app: sec_BE_dir} for a detailed proof of \eqref{eq: remainders_dir}. Then, for every $n \in \N$, one can write
\begin{displaymath}
    n^{  -3\delta + \frac{1}{2}} \left|\mathfrak{R}^{(n)}\left(\xi\right)- 1\right| \le  S( \xi), 
\end{displaymath}
with $S(\xi) : = \sup_{n \in \mathbb{N}}  n^{-3\delta + \frac{1}{2}} \left|\mathfrak{R}^{(n)}\left( \xi\right)- 1\right|$, and Lemma \ref{lem: sup is continuous} guarantees that the function $S$ is continuous.
Thus, for every $\xi\in\mathbb{R}$ satisfying \eqref{eq: BE_xi_bound_dir}, there exists $\tilde{c}$ such that $\left|\mathfrak{R}^{(n)}\left( \xi\right)- 1\right| \le  \tilde{c}\,   n^{  3\delta - \frac{1}{2}}$, completing the proof.
\end{proof}


\section{Proof of the strong LLN \eqref{lln_strong}}\label{sec4}

To clarify how to interpret the sequence $\{K_{n}\}_{n\geq1}$ for $\alpha\in[0,1)$ and $\theta=\lambda n$, with $\lambda>0$, it is useful to recall the definition of  the two-parameter Poisson-Dirichlet distribution. Among possible definitions, a simple and intuitive one follows from the stick-breaking construction \citep{Per(92),Pit(97)}. For each $n \in \N$, let $\{V_j^{(n)}\}_{j \geq 1}$ be a sequence of independent random variables such that for $j\geq1$ 
\begin{displaymath}
V_j^{(n)} \sim \text{Beta}(1-\alpha, \lambda n + \alpha j),
\end{displaymath}
where $\text{Beta}(a,b)$ denotes the Beta distribution with parameter $a,b>0$. Now set
\begin{displaymath}
W_1^{(n)} := V_1^{(n)}
\end{displaymath}
and, for $j\geq2$
\begin{displaymath}
W_j^{(n)} := V_j^{(n)} \prod_{l=1}^{j-1}(1- V_l^{(n)})
\end{displaymath}
so that $W_j^{(n)}\in(0,1)$ for $j\geq1$ and $\sum_{j\geq1}W_j^{(n)}=1$ almost surely. For each $n\in\mathbb{N}$, the two-parameter Poisson-Dirichlet distribution with $\alpha\in[0,1)$ and $\theta=\lambda n$, with $\lambda>0$, 
can be also seen as the law of the decreasing random probabilities $(W_j^{(n)})_{j\geq1}$. See \citet[Chapter 4]{Pit(06)} and references therein for further details.

Let $\{\tilde{\mathfrak{p}}_n\}_{n \geq 1}$ be the sequence of random probability measures defined by $\tilde{\mathfrak{p}}_n = \sum_{j \geq 1} W_j^{(n)} \delta_j$, meaning that $\tilde{\mathfrak{p}}_n$ takes values in the space of all probability measures on $(\N, 2^{\N})$, $2^{\N}$ denoting the power set of $\N$; $\tilde{\mathfrak{p}}_n$ is typically referred to as the two-parameter Pitman-Yor process \citep{Pit(95),Pit(97)}. For each $n \in \N$, consider a sequence $\{X_i^{(n)}\}_{i \geq 1}$ of exchangeable $\N$-valued random variables having $\tilde{\mathfrak{p}}_n$ as directing random probability measure, that is, the $X_i^{(n)}$'s are conditionally i.i.d. given $\tilde{\mathfrak{p}}_n$. Due to the discreteness of $\tilde{\mathfrak{p}}_n$, for each $n \in \N$ the random variables $\{X_1^{(n)}, \dots, X_n^{(n)}\}$ induce a random partition of $\{1,\ldots,n\}$, and we denote by $K_n^{(n)}$ the number of blocks of such a random partition. Obviously, for each $n \in \N$, it holds that
\begin{equation} \label{KnKn}
K_n^{(n)} \stackrel{d}{=} K_n
\end{equation}
with $K_n$ being the number of blocks in the Ewens-Pitman model with $\alpha\in[0,1)$ and $\theta=\lambda n$, with $\lambda>0$; that is the distribution of $K_{n}$ follows by marginalizing \eqref{epsm} with $\theta=\lambda n$, for $\lambda>0$. However, the well-known Markov structure of $\{K_n\}_{n \geq 1}$ \citep[Proposition 9]{Pit(95)} does not extends to $\{K_n^{(n)}\}_{n \geq 1}$, whatever the $k$-dimensional laws $(k \geq 2)$ of the sequence $\{K_n^{(n)}\}_{n \geq 1}$.

The strong LLN \eqref{lln_strong} is thus concerned with the sequence $\{K_n^{(n)}\}_{n \geq 1}$, but the proof we present ensues only from \eqref{KnKn}, independently of the 
$k$-dimensional laws $(k \geq 2)$ of the sequence $\{K_n^{(n)}\}_{n \geq 1}$. For this reason, and for notational convenience, we drop the superscript from $K_n^{(n)}$. The way is now paved for the proof of the strong LLN \eqref{lln_strong}. As for the LLN, we consider the identity
\begin{displaymath}
\frac{K_n - n \mathfrak{m}_{\alpha, \lambda}}{n} = \frac{K_n - \mathbb{E} \left[ K_n\right] }{n} + \frac{\mathbb{E} \left[ K_n\right] - n \mathfrak{m}_{\alpha, \lambda}}{n}.
\end{displaymath}
The almost sure convergence to $0$ of the first term on the right-hand side follows from a standard application of the Borel--Cantelli lemma, after proving that 
\begin{displaymath}
\sum_{n = 0}^{+\infty} P \left( \left| K_n - \mathbb{E}[K_n]\right| > n\varepsilon  \right) < + \infty
\end{displaymath}
holds true for any $\varepsilon >0$. In particular, this can be proved by combining Markov's inequality (with exponent $p=4$) with the fact that, as $n \to + \infty$ 
\begin{equation} \label{mom4Kn}
\mathbb{E} \left[\left(K_n - \mathbb{E} \left[ K_n\right] \right)^4 \right] = O\left(n^2\right).
\end{equation}
See Appendix  \ref{app:mom4Kn} for details on \eqref{mom4Kn}. This completes the proof of the strong LLN \eqref{lln_strong}.


\section{Proof of the Berry-Esseen inequality \eqref{eq:be_dir}}

We start by combining Lemma \ref{lem: BE lem_dir} with  the well-known inequality \citep[Chapter V, Theorem 2]{Pet(75)}
\begin{align*}
    \left\|F_{n} - \Phi \right\|_\infty &\le \int_{| \xi| \le \mathcal{C}\,\sigma(z) \, n^\delta} \left|\frac{\varphi_{n} (\xi) - e^{- \frac{\xi^2}{2}}}{\xi} \right| \, \mathrm{d} \xi + \tilde{\mathcal{C}}{n^{-\delta}} \\
    &\le  \int_{-\frac{1}{n}}^{\frac{1}{n}} \left|\frac{\varphi_{n} (\xi) - e^{- \frac{\xi^2}{2}}}{\xi} \right| \, \mathrm{d} \xi + \frac{\tilde{c}  }{n^{ - 3\delta +\frac{1}{2}}} \int_{\frac{1}{n}}^{+\infty} \frac{e^{- \frac{\xi^2}{2}}}{\xi} \, \mathrm{d} \xi + \tilde{\mathcal{C}}{n^{-\delta}}\\
    & =: I_1 + I_2 + \tilde{\mathcal{C}}{n^{-\delta}}.
\end{align*}
We consider separately the terms $I_{1}$ and $I_{2}$. For $I_{1}$, we combine the triangle inequality, \citet[Chapter IV, Lemma 5]{Pet(75)} and  \citet[Section 8.4, Theorem 1, Equation (4)]{CT(97)} to write that, in a neighborhood of $0$, it holds
\begin{align*}
    &\left|\frac{\varphi_{n}(\xi) - e^{-\frac{\xi^2}{2}}}{\xi} \right|\\ 
     &\quad\le \left|\mathbb{E} \left[\frac{K_n - n\mathfrak{m}_{0, \lambda}}{\sqrt{n\mathfrak{s}_{0, \lambda}^2}}\right] \right|  + \frac{1}{2} \mathbb{E} \left[\left(\frac{K_n - n\mathfrak{m}_{0, \lambda}}{\sqrt{n\mathfrak{s}_{0, \lambda}^2}}\right)^2\right]  \left|\xi \right| +\frac{1}{2} \left|\xi \right| + \frac{1}{8} \left|\xi^4 \right|.
\end{align*}
From \eqref{mom_m} and \eqref{mom_v}, we obtain
\begin{displaymath}
  \left|\mathbb{E} \left[\frac{K_n - n\mathfrak{m}_{0, \lambda}}{\sqrt{n\mathfrak{s}_{0, \lambda}^2}}\right] \right|  = \frac{1}{\sqrt{n\mathfrak{s}_{0, \lambda}^2}}  \left|\mathbb{E} \left[K_n\right]  - n \mathfrak{m}_{0, \lambda}\right| \le \frac{S_1}{\sqrt{n}},  
\end{displaymath}
where $S_1 : = \frac{1}{\mathfrak{s}_{0, \lambda}} \, \sup_{n \in \mathbb{N}}\left|\mathbb{E} \left[K_n\right]  - n \mathfrak{m}_{0, \lambda} \right|$. An analogous argument allows to prove 
\begin{displaymath}
  \mathbb{E} \left[\left(\frac{K_n - n\mathfrak{m}_{0, \lambda}}{\sqrt{n\mathfrak{s}_{0, \lambda}^2}}\right)^2\right]  \le \frac{S_2}{n}
\end{displaymath}
where $S_2 : = \frac{1}{\mathfrak{s}_{0, \lambda}} \, \sup_{n \in \mathbb{N}}\left| \operatorname{Var}[K_n] - n \left(\mathcal{D}+\mathfrak{m}_{0, \lambda}- 2\mathfrak{m}_{0, \lambda}\mathcal{B}\right)\right|$. Accordingly, we write 
\begin{displaymath}
    I_1 \le \frac{2 S_1}{n^{3/2}} + \frac{2 S_2+ 1/2}{n^3} + \frac{1}{40 \, n^5} \le \mathcal{C} n^{- 3/2}.
\end{displaymath}
For $I_{2}$, argue as in the proof of the Berry-Esseen theorem for $R_n(z)$ to conclude that
\begin{displaymath}
    I_2 \le   \tilde{c}_1 \log(n)  \, n^{3\delta - \frac{1}{2}}.
 \end{displaymath}
In conclusion,
\begin{displaymath}
    \left\|F_n - \Phi \right\|_\infty \le    C_1 n^{-\frac{3}{2}} + C_2 \log(n)  \, n^{ - \frac{1}{2} + 3\delta}  +  \tilde{\mathcal{C}}{n^{-\delta}}
 \end{displaymath}
  for some positive constants $C_1$ and $C_2$. To conclude, it is easy to see that for every $\delta \in (0, 1/6)$, $\min\left(\delta,   -3\delta +\frac{1}{2} , \frac{3}{2}\right) \ge \frac{1}{8}$, which produces the rate in \eqref{eq:be_dir}.


\section{Discussion}\label{sec5}

Motivated by the genetic interpretation of the Ewens model, \citet{Fen(07)} first investigated the the large $n$ asymptotic behaviour of $K_{n}$ in the regime $\theta=\lambda n$, with $\lambda>0$, providing a large deviation principle \citep[Theorem 4.1 and Theorem 4.4]{Fen(07)}. In this paper, we considered the same regime under the Ewens-Pitman model of parameters $\alpha\in[0,1)$ and $\theta>0$, which includes the Ewens model as a special case of $\alpha=0$. For $\alpha\in[0,1)$, our main results in Theorem \ref{thm_main} provide a LLN and a CLT for $K_{n}$ in the regime $\theta=\lambda n$, as well as a Berry-Esseen refinement of the CLT for $\alpha=0$. For $\alpha\in(0,1)$ our results rely on the compound Poisson construction of $K_{n}$, which led to novel large $n$ approximation results for the negative-Binomial compound Poisson random partition and for the Mittag-Leffler distribution function. From \citet[Theorem 2]{Dol(21)}, an analogous compound Poisson construction holds for the number $M_{r,n}$ of blocks with frequency $1\leq r\leq n$ in Ewens-Pitman random partition; see also \citep{Cha(07),Dol(20)}. Therefore, we expect that the strategy developed to prove Theorem \ref{thm_main} can be applied to provide analogous LLNs and CLTs for $M_{r,n}$. More generally, \citet[Theorem 2]{Dol(21)} suggests that one may consider the problem of extending Theorem \ref{thm_main} to other statistics of the Ewens-Pitman random partition, such as $(M_{1,n},\ldots,M_{p,n})$ and $M_{p,n}/K_{n}$ for $p<n$. Instead, for $\alpha=0$ the extension of Theorem \ref{thm_main} to $M_{r,n}$, as well as to statistics thereof, may be developed by means of more direct Berry-Esseen type calculations, along the lines of Lemma \ref{lem: BE lem_dir}.

Theorem \ref{thm_main} paves the way to future research in Bayesian (nonparametric) statistics for ``species sampling" problems \citep{Lij(07),Fav(09),Bal(24)}. In such a context, one of the major problem is the estimation of the number of unseen species in a sample. Specifically, given $n\geq1$ observed individuals belonging to different species (or symbols), which are modeled as a random sample $(X_{1},\ldots,X_{n})$ from the two-parameter Poisson-Dirichlet distribution, the unseen-species problem calls for estimating
\begin{displaymath}
K_{m}^{(n)}=|\{X_{n+1},\ldots,X_{n+m}\}\setminus \{X_{1},\ldots,X_{n}\}|,
\end{displaymath}
namely the number of hitherto unseen (distinct) species that would be observed if $m\geq1$ additional samples $(X_{n+1},\ldots,X_{n+m})$ were collected from the same distribution. Then, the interest is in the conditional distribution of $K_{m}^{(n)}$ given the Ewens-Pitman random partition $(K_{n},\mathbf{N}_{n})$ induced by $(X_{1},\ldots,X_{n})$ (i.e., the posterior distribution), whose expectation provides a Bayesian estimator of $K_{m}^{(n)}$. For $\alpha\in(0,1)$, in analogy with the almost-sure fluctuation \eqref{as_limit}, the large $m$ asymptotic behaviour of the posterior distribution has been characterized, with applications to approximate credible intervals for the Bayesian estimator \citep{Fav(09)}. For $\alpha\in[0,1)$ one may consider the problem of investigating such a large $m$ asymptotic behaviour in the regime $\theta=\lambda m$, with $\lambda>0$, developing a posterior counterpart of Theorem \ref{thm_main} to introduce large $m$ Gaussian credible intervals for  the Bayesian estimator of $K_{m}^{(n)}$. Work through this direction is ongoing \citep{Con(24)}.


\section*{Appendices}
\addcontentsline{toc}{section}{Appendices}
\renewcommand{\thesubsection}{\Alph{subsection}}



\renewcommand{\theequation}{A.\arabic{equation}}

\setcounter{equation}{0}

\subsection{Technical results related to Theorem \ref{thm_main} for $\alpha\in(0,1)$}\label{app2}


\subsubsection{Asymptotic expansions \eqref{mom_m} and \eqref{mom_v}}\label{app21}

Based on the distribution of $K_{n}$ \citep[Equation 3.11]{Pit(06)}, a direct calculation leads to the following expressions:
\begin{displaymath}
\mathbb{E}\left[(K_n)_{\downarrow j}\right] = \left[\frac{\lambda n}{\alpha}\right]_{(j)} \sum_{i=0}^j (-1)^{j-i} \binom{j}{i} \frac{[\lambda n + i\alpha]_{(n)}}{[\lambda n]_{(n)}} \qquad (j \in \N)
\end{displaymath}
where $(x)_{\downarrow j} := \prod_{i=0}^{j-1} (x - i)$ denotes the falling factorial of $x$. Accordingly, we write
\begin{equation}\label{eq: mean_kn_gen}
\mathbb{E} \left[K_n\right]=\mathbb{E}\left[(K_n)_{\downarrow 1}\right]=\left(\frac{\theta}{\alpha}\right)\left[-1 + \frac{\Gamma(\theta+n+\alpha)}{\Gamma(\theta + \alpha)} \, \frac{\Gamma(\theta + n)}{\Gamma(\theta)}\right]
\end{equation}
and
\begin{align}\label{eq: var_kn_gen}
& \text{Var}(K_n) \\
\nonumber &\quad= \mathbb{E}\left[(K_n)_{\downarrow 2}\right] + \mathbb{E}\left[K_n\right] - \left(\mathbb{E}\left[K_n\right]\right)^2\\
 \nonumber &\quad=  \frac{\theta}{\alpha}\left(\frac{\theta}{\alpha} + 1\right) \left[1 - 2 \frac{\Gamma(\theta+n+\alpha)}{\Gamma(\theta+n)} \frac{\Gamma(\theta)}{\Gamma(\theta+\alpha)} + \frac{\Gamma(\theta+n+2\alpha)}{\Gamma(\theta+n)} \frac{\Gamma(\theta)}{\Gamma(\theta+2\alpha)} \right]\\
    \nonumber &\quad\quad+ \frac{\theta}{\alpha} \left[-1 + \frac{\Gamma(\theta+n+\alpha)}{\Gamma(\theta+n)} \frac{\Gamma(\theta)}{\Gamma(\theta+\alpha)} \right]\\
    \nonumber&\quad\quad-\left(\frac{\theta}{\alpha}\right)^2 \left[ 1+ \left(\frac{\Gamma(\theta+n+\alpha)}{\Gamma(\theta+n)} \frac{\Gamma(\theta)}{\Gamma(\theta+\alpha)}\right)^2 -2 \frac{\Gamma(\theta+n+\alpha)}{\Gamma(\theta+n)}\frac{\Gamma(\theta)}{\Gamma(\theta+\alpha)} \right]. 
\end{align}
See \citet[Theorem 2.1]{Ber(24)} for a proof of \eqref{eq: mean_kn_gen} and \eqref{eq: var_kn_gen} based on the sequential construction of the Ewens-Pitman model. If $\theta=\lambda n$, \eqref{eq: mean_kn_gen} becomes
\begin{align}\label{eq: mean_kn_gen_final}
 \mathbb{E}\left[K_n\right]& = \frac{\lambda n}{\alpha} \left[-1 + \frac{\Gamma \left(\alpha+ (\lambda +1) n\right) \ \Gamma(\lambda n)}{\Gamma(\alpha + \lambda n) \ \Gamma\left((\lambda +1) n\right)} \right].
 \end{align}
Similarly, if $\theta=\lambda n$, \eqref{eq: var_kn_gen} becomes
 \begin{align}\label{eq: var_kn_gen_final}
&\text{Var}(K_n) \\
&\notag\quad= \frac{\lambda n}{\alpha}\left(\frac{\lambda n}{\alpha} + 1\right) \left[1 - 2 \frac{\Gamma(n(\lambda+1)+\alpha)}{\Gamma(n(\lambda+1))} \frac{\Gamma(\lambda n)}{\Gamma(\lambda n +\alpha)} \right. +\frac{\Gamma(n(\lambda+1)+2\alpha)}{\Gamma(n(\lambda+1))} \left.  \frac{\Gamma(\lambda n)}{\Gamma(\lambda n +2\alpha)} \right] \\ 
&\notag\quad\quad+ \frac{\lambda n}{\alpha}  \left[-1 + \frac{\Gamma(n(\lambda+1)+\alpha)}{\Gamma(n(\lambda+1))} \frac{\Gamma(\lambda n)}{\Gamma(\lambda n +\alpha)} \right] \\
&\notag\quad\quad -\left(\frac{\lambda n}{\alpha}\right)^2 \left[ 1+ \left(\frac{\Gamma(n(\lambda+1)+\alpha)}{\Gamma(n(\lambda+1))} \frac{\Gamma(\lambda n)}{\Gamma(\lambda n +\alpha)}\right)^2 -2 \frac{\Gamma(n(\lambda+1)+\alpha)}{\Gamma(n(\lambda+1))} \frac{\Gamma(\lambda n)}{\Gamma(\lambda n +\alpha)} \right].
\end{align}
To conclude, an application of \citep[Equation 1]{TE(51)} to \eqref{eq: mean_kn_gen_final} and \eqref{eq: var_kn_gen_final} yields the final expressions given in Section \ref{sec21} after straightforward algebraic
manipulations.

 \subsubsection{Proof of equation \eqref{mom4Kn}}\label{app:mom4Kn}

By standard combinatorial relations, 
\begin{equation} \label{momento_4}
\mathbb{E} \left[\left(K_n - \mathbb{E} \left[ K_n\right] \right)^4 \right]  
 =\sum_{r = 0}^4 (-1)^{4-r} \binom{4}{r} \mathbb{E}[K_n]^{4-r} \sum_{j = 1}^r S(r, j) \mathbb{E} \left[(K_n)_{\downarrow j} \right]
\end{equation}
where $S(r, j)$ denotes the Stirling number of the second kind, and $(x)_{\downarrow j} := \prod_{i=0}^{j-1} (x - i)$ is the falling factorial of $x$. Now, rewrite the explicit expression of the falling factorial moments of $K_n$
already given in \ref{app21} as
\begin{displaymath}
\mathbb{E} \left[(K_n)_{\downarrow j}\right] = \frac{\Gamma(\theta/\alpha + j)}{\Gamma(\theta/\alpha)}
\sum_{i = 0}^r (-1)^{j-i} \binom{j}{i} \frac{\Gamma(\theta + i\alpha + n)}{\Gamma(\theta+ i\alpha)} \frac{\Gamma(\theta)}{\Gamma(\theta+n)}.
\end{displaymath}
Setting $\theta = \lambda n$, resort again to \citep[Equation 1]{TE(51)} to obtain
\begin{align*}
   \mathbb{E} \left[(K_n)_{\downarrow 1} \right] & = n \cdot \frac{\lambda}{\alpha} \left( L - 1 \right)  - L \,  \frac{\alpha-1}{2( \lambda+1)} + O \left(n^{-1}\right)\\
     \mathbb{E} \left[(K_n)_{\downarrow 2} \right] & = n^2 \cdot \frac{\lambda^2}{\alpha^2} \left(L - 1 \right)^2 \\
     &\quad  + n \cdot \frac{\lambda}{\alpha}  \left[ (L-1)^2 + L \, \frac{\alpha -1}{\lambda+1} - L^2 \, \frac{2\alpha-1}{\lambda+1}\right] + O(1)\\
      \mathbb{E} \left[(K_n)_{\downarrow 3} \right] & = n^3 \cdot \frac{\lambda^3}{\alpha^3} \left(L - 1 \right)^3 \\
      &\quad +   n^2 \cdot 3\,  \frac{\lambda^2}{\alpha^2} \left[ \left(L - 1 \right)^3  - L \, \frac{\alpha -1}{2(\lambda+1)} + L^2 \, \frac{2\alpha-1}{\lambda+1}  - L^3 \, \frac{3\alpha -1}{2(\lambda+1)}\right]\\
      &\quad + O\left(n\right)\\
       \mathbb{E} \left[(K_n)_{\downarrow 4} \right] & = n^4 \cdot \frac{\lambda^4}{\alpha^4} \left(L - 1 \right)^4 \\
      &\quad +   n^3 \cdot 2\,  \frac{\lambda^3}{\alpha^3} \left[ 3  \left(L - 1 \right)^4  + L \, \frac{\alpha -1}{\lambda+1} -3  L^2 \, \frac{2\alpha-1}{\lambda+1}  +3 L^3 \, \frac{3\alpha -1}{2(\lambda+1)} \right. \\
      & \quad \quad \quad \quad \quad \left.- L^4 \, \frac{4\alpha-1}{\lambda+1}\right]+ O\left(n^2\right)\
\end{align*}
where $L  := \left(\frac{\lambda+1}{\lambda} \right)^\alpha$. Since $\mathbb{E} \left[(K_n)_{\downarrow 1}\right] = \mathbb{E}[K_n]$, the above identities lead to
  \begin{align*}
   \mathbb{E}[K_n]^2 & =  n^2 \cdot \frac{\lambda^2}{\alpha^2} \left(L - 1 \right)^2  - n \cdot \frac{\lambda}{\alpha} \,  (L-1) L\,  \frac{\alpha -1}{\lambda+1}+ O \left(1 \right)\\
     \mathbb{E}[K_n]^3 & = n^3 \cdot \frac{\lambda^3}{\alpha^3} \left(L - 1 \right)^3 - n^2 \cdot 3 \frac{\lambda^2}{\alpha^2}\,  (L-1)^2 L\, \frac{\alpha -1}{2(\lambda+1)} +O(n)\\
          \mathbb{E}[K_n]^4 & = n^3 \cdot \frac{\lambda^3}{\alpha^3} \left(L - 1 \right)^3 \\
      &\quad +   n^2 \cdot 3\,  \frac{\lambda^2}{\alpha^2} \left[ \left(L - 1 \right)^3  - L \, \frac{\alpha -1}{2(\lambda+1)} + L^2 \, \frac{2\alpha-1}{\lambda+1}  - L^3 \, \frac{3\alpha -1}{2(\lambda+1)}\right]\\
      &\quad + O\left(n\right).
      \end{align*}
Therefore, \eqref{momento_4} reduces to
      \begin{align*}
     &  \mathbb{E} \left[\left(K_n - \mathbb{E} \left[ K_n\right] \right)^4 \right]     \\
      & \quad =  - 3  \mathbb{E}[K_n]^4  + 6  \mathbb{E}[K_n]^3 + 6 \mathbb{E}[K_n]^2  \mathbb{E} \left[(K_n)_{\downarrow 2} \right]  -12  \mathbb{E}[K_n]  \mathbb{E} \left[(K_n)_{\downarrow 2}\right]\\
     & \quad  \quad   - 4 \mathbb{E}[K_n]  \mathbb{E} \left[(K_n)_{\downarrow 3} \right] + 6   \mathbb{E} \left[(K_n)_{\downarrow 3} \right] +  \mathbb{E} \left[(K_n)_{\downarrow 4} \right]\\
     &  \quad   = n^4 \cdot \mathfrak{A}(\alpha, \lambda) + n^3 \cdot \mathfrak{B}(\alpha, \lambda) + O\left(n^2\right)
      \end{align*}
      with 
      \begin{displaymath}
      \mathfrak{A}(\alpha, \lambda) = \frac{\lambda^4}{\alpha^4} (L-1)^4 \left(-3+6-4+1\right)
      \end{displaymath}
      and
      \begin{align*}
       \mathfrak{B}(\alpha, \lambda) & =2 \, \frac{\lambda^3}{\alpha^3}  \left\{ 3 (L-1)^3 L \, \frac{\alpha-1}{\lambda+1} + 3 (L-1)^3 \right. \\ 
       & \left.\quad + 3 (L-1)^2 \left[ (L-1)^2 + L \, \frac{\alpha-1}{\lambda + 1} - L^2 \, \frac{2\alpha - 1}{\lambda+1}  - L(L-1) \, \frac{\alpha -1}{\lambda+1}\right] \right. \\
       & \left. \quad -6  (L-1)^3  + (L-1)^3 L \, \frac{\alpha-1}{\lambda+1} + 3 (L-1)^3  \right. \\
       & \left. \quad -3 (L-1) \, \left[ 2 (L-1)^3  -  L \, \frac{\alpha-1}{\lambda+1}  + 2L^2 \, \frac{2\alpha-1}{\lambda+1} - L^3 \, \frac{3\alpha-1}{\lambda+1}\right] \right. \\
       & \left. \quad + 3(L-1)^4 + L\, \frac{\alpha -1}{\lambda+1} - 3L^2 \ \frac{2\alpha-1}{\lambda+1} +  3L^3 \ \frac{3\alpha-1}{\lambda+1} - L^4  \ \frac{4\alpha-1}{\lambda+1}\right\}.
      \end{align*}
      Finally, straightforward algebraic computations show that
      \begin{displaymath}
       \mathfrak{A}(\alpha, \lambda) =  \mathfrak{B}(\alpha, \lambda) =0, 
       \end{displaymath} 
       proving equation \eqref{mom4Kn}.

 \subsubsection{Proof of Proposition \ref{prop:clt_Z}}\label{sec: Zn}

The espression of $\mathbb{E}[S_{\lambda n, \alpha}]$ is known form \eqref{eq: mean stable}. 
Furthermore, one has that 
\begin{displaymath}
    \mathbb{E}\left[G_{A, 1}^{k\alpha}\right] = \frac{\Gamma(A+ k\alpha)}{\Gamma\left(A\right) }.
\end{displaymath}
By independence, 
\begin{displaymath}
    \mathbb{E}\left[Z_n\right] = \mathbb{E}\left[S_{\lambda n, \alpha}\right]\mathbb{E}\left[G_{(\lambda+1)n, 1}^{\alpha}\right] = 
    \frac{\lambda n}{\alpha}\frac{\Gamma (\lambda n )}{ \Gamma \left(\lambda n + \alpha \right)}\frac{\Gamma((\lambda + 1)n + \alpha)}{\Gamma\left((\lambda + 1)n \right) }.
\end{displaymath}
Identity \eqref{mom_Zn1} follows by applying \citep[Equation 1]{TE(51)}  to the above expression. Analogously, 
\begin{displaymath}
    \mathbb{E}\left[Z_n^2\right] = \mathbb{E}\left[S_{\lambda n, \alpha}^2\right]  \mathbb{E}\left[G_{(\lambda+1)n, 1}^{2\alpha}\right] = 
    \frac{\lambda n}{\alpha} \left(\frac{\lambda n}{\alpha} + 1\right) \frac{\Gamma (\lambda n )}{ \Gamma \left(\lambda n + 2\alpha \right)} \frac{\Gamma((\lambda + 1)n + 2\alpha)}{\Gamma\left((\lambda + 1)n \right) }
\end{displaymath}
which, upon rearrangement of the terms, entails
\begin{align*}
    \operatorname{Var}\left[Z_n\right] & =  \mathbb{E}\left[Z_n^2\right] - \left(\mathbb{E}\left[Z_n\right]\right)^2 \\
    & =  \frac{\lambda^2 n^2}{\alpha^2}  \left[ \frac{\Gamma (\lambda n )}{ \Gamma \left(\lambda n + 2\alpha \right)}   \frac{\Gamma((\lambda + 1)n + 2\alpha)}{\Gamma\left((\lambda + 1)n \right) } -   \frac{\Gamma (\lambda n )^2}{ \Gamma \left(\lambda n + \alpha \right)^2}   \frac{\Gamma((\lambda + 1)n + \alpha)^2}{\Gamma\left((\lambda + 1)n \right)^2}\right]\\
    &\quad + \frac{\lambda n }{\alpha} \left[\frac{\Gamma (\lambda n )}{ \Gamma \left(\lambda n + 2\alpha \right)}   \frac{\Gamma((\lambda + 1)n + 2\alpha)}{\Gamma\left((\lambda + 1)n \right) } \right].
\end{align*}
Again \citep[Equation 1]{TE(51)}, simplified for the purpose of an approximation to the principal order, yields
\begin{align*}
    \operatorname{Var}\left[Z_n\right] & =  \frac{\lambda^2 n^2}{\alpha^2}  \left[ \left(\frac{\lambda+1}{\lambda}\right)^{2\alpha} \left(1 - \frac{\alpha (2\alpha-1)}{\lambda(\lambda+1)n} + O\left(\frac{1}{n^2}\right)\right) \right. \\
    &\quad \left.-  \left(\frac{\lambda+1}{\lambda}\right)^{2\alpha} \left(1 - \frac{\alpha (\alpha-1)}{\lambda(\lambda+1)n} + O\left(\frac{1}{n^2}\right)\right)\right]\\
    &\quad + \frac{\lambda n }{\alpha} \left[\left(\frac{\lambda+1}{\lambda}\right)^{2\alpha} + O \left(\frac{1}{n}\right) \right].
\end{align*}
A direct computation shows that the coefficient of $n^2$ vanishes, while the coefficient of $n$ rewrites as $\Sigma^2$, proving identity \eqref{mom_Zn2}.

\begin{proof}[Proof of the LLN \eqref{eq:lln_zn}] Fix $\varepsilon >0$ and combine Chebyshev's inequality with the (variance) asymptotic expansion \eqref{mom_Zn2} to get
\begin{align*}
    P\left[\left| \frac{Z_n - \mathbb{E}[Z_n]}{n} \right|> \varepsilon \right] &= P \left[\left| Z_n - \mathbb{E}[Z_n] \right| > n\varepsilon \right]\\
    & \le \frac{\operatorname{Var}(Z_n)}{n^2 \varepsilon^2}\\
    &  = O\left(\frac{1}{n}\right) \rightarrow 0
\end{align*}
as $n \rightarrow +\infty$. Since \eqref{mom_Zn1} implies that $n^{-1}\mathbb{E}[Z_n]\rightarrow z_0$ as $n \rightarrow+\infty$, the proof is concluded. 
\end{proof} 

\begin{proof}[Outline of the proof of the CLT \eqref{eq:clt_zn}] The proof consists of three steps.
\begin{enumerate}
    \item In Lemma \ref{lem:lem4_quater}, we provide an approximation of the distribution of $S_{\lambda n, \alpha}$ in terms of the distribution of the $(1-\alpha)$-th power of a Gamma random variable. 
    More precisely, denote by $d_{\text{TV}}$ the total variation distance between distributions on $(0,+\infty)$, namely
      \begin{displaymath}
        d_{\text{TV}}\left(\mu, \nu \right) : = \sup_{A \in \mathscr{B}((0,+\infty))} \left| \mu(A) - \nu(A)\right|.
    \end{displaymath}
    We will show that
    \begin{displaymath}
        d_{\text{TV}}\left(\mu_{S_{\lambda n, \alpha}}, \,  \mu_{G_{\rho n + \tau, B}
        ^{1-\alpha}}\right) =O\left(\frac{1}{n} \right),
    \end{displaymath}
    where $G_{\rho n + \tau, B}$ has $\operatorname{Gamma}\left(\rho n + \tau, B \right)$ distribution, for a suitable choice of parameters $\rho, \tau$ and $B$.
    \item In Lemma \ref{lem:lem_3_quater}, we will prove a general CLT for variables of the form
    \begin{displaymath}
        \left(\sum_{i = 0}^n X_i\right)^{1-\alpha}  \left(\sum_{i = 1}^n Y_i\right)^{\alpha},
    \end{displaymath}
where $\{X_i\}$ and $\{Y_i\}$ are two independent i.i.d. sequences of positive random variables. As an immediate consequence (Corollary \ref{cor:cor1_quater}), we derive a CLT for products of powers of independent Gamma random variables, which we then apply to the auxiliary variable $\hat{Z_n}: = G_{\rho n+ \tau, B}
        ^{1-\alpha}  G_{(\lambda+1)n, 1}^\alpha$.
    \item Thanks to a consequence of Lemma \ref{lem:lem4_quater} and to a further general result (Lemma  \ref{lem:lem5_quater}), we will show that it is possible to replace the approximating Gamma random variable with the random variable $S_{\lambda n, \alpha}$ in the CLT for $\hat{Z}_n$, obtaining the desired limit theorem for $Z_n$. 
\end{enumerate}
\end{proof}

\begin{paragraph}{\underline{Step 1} of the proof of the CLT \eqref{eq:clt_zn}}
\begin{lem}
    \label{lem:Zol}
   Let $E_\alpha : (0, +\infty) \to (0, +\infty)$ be defined by
\begin{displaymath}
    E_\alpha(y)  := \frac{1}{\sqrt{2\pi\alpha (1-\alpha)}} \, \left(\frac{\alpha}{y}\right)^{\frac{2-\alpha}{2(1-\alpha)}} \, \exp\left\{-(1-\alpha) \left(\frac{\alpha}{y}\right)^{\frac{\alpha}{1-\alpha}} \right\}.
\end{displaymath}
 Then, for any $y \in (0, +\infty)$, there hold
 \begin{align}
\left| \frac{f_\alpha(y)}{E_\alpha(y)} -1 \right| &\le C_\alpha y^{\frac{\alpha}{1-\alpha}} \label{eq:Zol_bound_1}\\
\nonumber \\
\left| \frac{f_\alpha(y)}{E_\alpha(y)} -1  - Q_\alpha y^{\frac{\alpha}{1-\alpha}}\right| &\le K_\alpha y^{\frac{2 \alpha}{1-\alpha}} \label{eq:Zol_bound_2}
\end{align}
for suitable positive constants $C_\alpha, \, Q_\alpha, \, K_\alpha$ depending only on $\alpha$.
\end{lem}

\begin{proof}
\citet[Theorem 2.5.2]{Zol(86)} shows that, as $y \to 0$
\begin{displaymath}
    \frac{f_\alpha(y)}{E_\alpha(y)} = 1 + O\left(y^\frac{\alpha}{1-\alpha}\right)
\end{displaymath}
Furthermore, it is known that, as $y \to +\infty$,
\begin{displaymath}
    f_\alpha(y) = O\left(y^{-(1+\alpha)}\right) \, \text{ and } \, 
    E_\alpha(y) = O\left(y^{ - \frac{2 - \alpha}{2(1-\alpha)}}\right).
\end{displaymath}
Whence, as $y \to +\infty$
\begin{displaymath}
    \frac{f_\alpha(y)}{E_\alpha(y)} = O\left(y^{\frac{\alpha(2\alpha -1)}{1-\alpha}}\right).
\end{displaymath}
Since $2\alpha-1<1$ for every $\alpha \in (0, 1)$, \eqref{eq:Zol_bound_1} is proven. To prove \eqref{eq:Zol_bound_2}, use again \citet[Theorem 2.5.2]{Zol(86)} to show that, as $y \to 0$
\begin{displaymath}
    \frac{f_\alpha(y)}{E_\alpha(y)} = 1 + Q_\alpha y^\frac{\alpha}{1-\alpha} +  O\left(y^{\frac{2\alpha}{1-\alpha}}\right)
\end{displaymath}
holds with a suitable constant $Q_\alpha$. The conclusion is obtained by considering the previous asymptotic behaviours as $y \to +\infty$.
\end{proof}

\begin{lem}
    \label{lem:lem4_quater}
    Set $\rho = \lambda \frac{1-\alpha}{\alpha}$, $\tau = \frac{1}{2}$, $B = (1-\alpha) \, \alpha^{\frac{\alpha}{1-\alpha}}$. Then, letting $\mu_{S_{\lambda n, \alpha}}$ and $\mu_{G_{\rho n + \tau, B}^{1-\alpha}}$
    stand for the distributions of $S_{\lambda n, \alpha}$ and $G_{\rho n + \tau, B}^{1-\alpha}$, respectively, as $n \to +\infty$ there holds
    \begin{displaymath}
        d_{\text{TV}}\left(\mu_{S_{\lambda n, \alpha}}, \,  \mu_{G_{\rho n + \tau, B}^{1-\alpha}}\right) =O\left(\frac{1}{n} \right).
    \end{displaymath}
   \end{lem}

\begin{proof} Since $\mu_{S_{\lambda n, \alpha}}$ and $\mu_{G_{\rho n + \tau, B}^{1-\alpha}}$ have densities
       \begin{align*}
f_{S_{\lambda n, \alpha}}(x) &= \frac{\Gamma(\lambda n)}{\Gamma \left(\frac{\lambda n}{\alpha}\right)} \,  x^{\frac{\lambda n -1}{\alpha} -1} \, f_\alpha \left(x^{-\frac{1}{\alpha}}\right)  \mathds{1}_{(0, +\infty)}(x) \\
f_{G_{\rho n + \tau, B}^{1-\alpha}}(x) &= \frac{\left[(1-\alpha) \alpha^\frac{\alpha}{1-\alpha}\right]^{\frac{1-\alpha}{\alpha}\lambda n + \frac{1}{2}}}{(1-\alpha) \Gamma \left(\frac{1-\alpha}{\alpha}\lambda n + \frac{1}{2}\right)}\,  \exp\left\{-(1-\alpha) \alpha^\frac{\alpha}{1-\alpha}\right\}\,  x^{\frac{\lambda n}{\alpha} + \frac{2\alpha -1}{2(1-\alpha)}} \mathds{1}_{(0, +\infty)}(x)
        \end{align*}       
respectively, the total variation distance at issue can be handled as follows:
\begin{align*}
 d_{\text{TV}}\left(\mu_{S_{\lambda n, \alpha}}, \,  \mu_{G_{\rho n + \tau, B}}\right) & = 
\frac{1}{2} \int_0^{+\infty} \left| f_{S_{\lambda n, \alpha}}(x) -  f_{G_{\rho n + \tau, B}^{1-\alpha}} (x) \right| \, \mathrm{d} x\\
\text{(bound \eqref{eq:Zol_bound_1})} &\le \frac{1}{2}  \int_0^{+\infty} \left|\frac{\Gamma(\lambda n) }{\Gamma\left(\frac{\lambda n}{\alpha }\right)}\, x^{\frac{\lambda n-1}{\alpha} - 1} \, E_\alpha\left(x^{-\frac{1}{\alpha}} \right) - f_{G_{\rho n + \tau, B}^{1-\alpha}}(x)  \right| \, \mathrm{d} x\\
    &\quad + \frac{1}{2} C_\alpha  \int_0^{+\infty} \frac{\Gamma(\lambda n) }{\Gamma\left(\frac{\lambda n}{\alpha }\right)}\, x^{\frac{\lambda n-1}{\alpha} - 1} \, E_\alpha\left(x^{-\frac{1}{\alpha}} \right) \, x^{- \frac{1}{1-\alpha}} \, \mathrm{d} x  \\
    & = : \frac{1}{2}  \mathcal{J}_1^{(n)} + \frac{C_\alpha}{2}   \mathcal{J}_2^{(n)}.
\end{align*}
Thus, te proof is concluded if we show that
\begin{displaymath}
   \mathcal{J}_1^{(n)}  = O\left(\frac{1}{n}\right) \quad \, \text{ and } \, \quad
     \mathcal{J}_2^{(n)} = O\left(\frac{1}{n}\right).
\end{displaymath}
Simple algebraic manipulations yield
\begin{align*}
    \mathcal{J}_1^{(n)} &= \left|\frac{\Gamma(\lambda n) }{\Gamma\left(\frac{\lambda n}{\alpha }\right)} \frac{\alpha^\frac{2-\alpha}{2(1-\alpha)}}{\sqrt{2\pi\alpha(1-\alpha)}} - \frac{B^{\frac{1-\alpha}{\alpha }\lambda n + \frac{1}{2}}}{(1-\alpha) \Gamma\left(\frac{1-\alpha}{\alpha} \lambda n + \frac{1}{2}\right)} \right| \times\\
    &\quad \times   \int_0^{+\infty} \exp {\left\{ -  B x^\frac{1}{1-\alpha}\right\}} \, x^{\frac{\lambda n}{\alpha} + \frac{2\alpha -1}{2(1-\alpha)}} \, \mathrm{d} x\\
    %
    & = \left|\frac{\Gamma(\lambda n) }{\Gamma\left(\frac{\lambda n}{\alpha }\right)} \frac{\alpha^\frac{2-\alpha}{2(1-\alpha)}}{\sqrt{2\pi\alpha(1-\alpha)}} - \frac{B^{\frac{1-\alpha}{\alpha } + \frac{1}{2}}}{(1-\alpha) \Gamma\left(\frac{1-\alpha}{\alpha} \lambda n + \frac{1}{2}\right)} \right| \times \\
    &\quad \times   \frac{(1-\alpha) \Gamma\left(\frac{1-\alpha}{\alpha} \lambda n + \frac{1}{2}\right)}{B^{\frac{1-\alpha}{\alpha } + \frac{1}{2}}}\\
    & = \left|\frac{\Gamma(\lambda n) \, \Gamma\left(\frac{1-\alpha}{\alpha} \lambda n + \frac{1}{2}\right)}{\Gamma\left(\frac{\lambda n}{\alpha }\right)} \, \left(\frac{1}{1-\alpha}\right)^{\frac{1-\alpha}{\alpha} \lambda n}  \, \left(\frac{1}{\alpha} \right)^{\lambda n - \frac{1}{2}} \left(2\pi\right)^{-\frac{1}{2}}- 1 \right|. 
\end{align*}
Moreover, for some suitable constants $c_\alpha$, $\tilde{C}_\alpha$ depending only on $\alpha$,
\begin{align*}
    \mathcal{J}_2^{(n)} &= \frac{\Gamma(\lambda n) }{\Gamma\left(\frac{\lambda n}{\alpha }\right)}  \int_0^{+\infty} c_\alpha \, \exp {\left\{ -  B x^\frac{1}{1-\alpha}\right\}} \, x^{\frac{\lambda n}{\alpha} + \frac{2\alpha -3}{2(1-\alpha)}} \, \mathrm{d} x\\
    %
    & = \frac{\Gamma(\lambda n) }{\Gamma\left(\frac{\lambda n}{\alpha }\right)}  \frac{c_{\alpha} (1-\alpha) \Gamma\left(\frac{1-\alpha}{\alpha} \lambda n - \frac{1}{2}\right)}{B^{\frac{1-\alpha}{\alpha }\lambda n - \frac{1}{2}}}\\
    & = \tilde{C}_\alpha \frac{\Gamma(\lambda n) }{\Gamma\left(\frac{\lambda n}{\alpha }\right)}  \frac{\Gamma\left(\frac{1-\alpha}{\alpha} \lambda n - \frac{1}{2}\right)}{\left[(1-\alpha) \alpha^\frac{\alpha}{1-\alpha}\right]^{\frac{1-\alpha}{\alpha }\lambda n} }.
\end{align*}
Set $\lambda n =: t$. Then, we study the asymptotic behavior of
\begin{displaymath}
   g_j(t) : = \frac{\Gamma(t) \, \Gamma\left(\frac{1-\alpha}{\alpha} t + \frac{j}{2}\right) }{\Gamma\left(\frac{t}{\alpha }\right)} 
\end{displaymath}
as $ t\to +\infty$, for $j \in \{-1, 1\}$. By combining Stirling's approximation to the first order, i.e.
\begin{displaymath}
    \Gamma(t) = \sqrt{2\pi (t-1)} \left(\frac{t-1}{e}\right)^{t-1} \left[ 1+ O\left(t^{-1}\right)\right] \quad (t \to +\infty)\ ,
\end{displaymath}
with the fact that
\begin{displaymath}
\left(1 + \frac{a}{t}\right)^t = e^a \left[ 1+ O\left(t^{-1} \right)\right]\quad (t \to +\infty)\ ,
\end{displaymath}
we get, after straightforward computations,
\begin{align*}
   g_j(t) & = \begin{cases}
   (2 \pi e)^{\frac{1}{2}} \left(\frac{1-\alpha}{\alpha}\right)^{\frac{1-\alpha}{\alpha} t }\alpha^{\frac{t}{\alpha} - \frac{1}{2}} \left[ 1- \frac{1}{2t} + O\left(\frac{1}{t^2}\right)\right]^t & \text{if } j = 1\\
   c_\alpha^\star (2 \pi)^{\frac{1}{2}} e^{\frac{3}{2}} \left(\frac{1-\alpha}{\alpha}\right)^{\frac{1-\alpha}{\alpha} t }\alpha^{\frac{t}{\alpha} - \frac{1}{2}} \frac{1}{t} \left[ 1- \frac{3}{2t} + O\left(\frac{1}{t^2}\right)\right]^t & \text{if } j = -1
   \end{cases} \\
   \\
  &  = \begin{cases}
   (2 \pi)^{\frac{1}{2}} \left(\frac{1-\alpha}{\alpha}\right)^{\frac{1-\alpha}{\alpha} t }\alpha^{\frac{t}{\alpha} - \frac{1}{2}} \left[ 1+ O\left(\frac{1}{t}\right)\right] & \text{if } j = 1\\
     O\left(\frac{1}{t}\right) & \text{if } j = -1
   \end{cases}
\end{align*}
where $c_\alpha^\star$ is another constant depending only on $\alpha$.
Substituting the first and second expression, respectively, in the expressions of $\mathcal{J}_1^{(n)} $ and $\mathcal{J}_2^{(n)} $, we obtain the desired results.
\end{proof}
\end{paragraph}

\begin{paragraph}{\underline{Step 2} of the proof of the CLT \eqref{eq:clt_zn}}

\begin{lem}
   \label{lem:lem_3_quater} 
   Let $X_0$ be a real random variable and $\{X_i\}_{i \ge 1}$ and $\{Y_i\}_{i \ge 1}$ be two sequences of i.i.d. positive random variables. We further assume:
   \begin{enumerate}
   \item[a)] $\mathbb{E}\left[X_1^4 \right] < +\infty$ and $\mathbb{E}\left[Y_1^4 \right] < +\infty$;
   \item[b)] $\mu_X : = \mathbb{E}\left[X_1\right]$, $\mu_Y : = \mathbb{E}\left[Y_1\right]$, $\sigma^2_X : = \operatorname{Var}\left(X_1\right)$ and $\sigma^2_Y : = \operatorname{Var}\left(Y_1\right)$;
   \item[c)] $X_0$, $\{X_i\}_{i \ge 1} $ and $\{Y_i\}_{i \ge 1} $ are independent.
   \end{enumerate}
  Then, the following CLT holds:
  \begin{displaymath}
      \frac{1}{\sqrt{n}} \left[ \left(\sum_{i = 0}^n X_i\right)^{1-\alpha}  \left(\sum_{i = 1}^n Y_i\right)^{\alpha} - nm \right]  \stackrel{\text{w}}{\longrightarrow} \mathcal{N}\left(0, V^2\right)
  \end{displaymath}
  as $n \to + \infty$, where
  \begin{displaymath}
      m := \mu_X^{1-\alpha} \mu_Y^\alpha
  \end{displaymath}
  and
  \begin{displaymath}
      V^2 := \mu_X^{2(1-\alpha)} \mu_Y^{2\alpha} \left[\frac{(1-\alpha)^2 \sigma_X^2}{\mu_X^2} + \frac{\alpha^2 \sigma_Y^2}{\mu_Y^2} \right].
  \end{displaymath}
\end{lem}
\begin{proof}
Start considering these trivial identities:
\begin{displaymath}
    \frac{1}{n} \sum_{ i = 0}^n X_i = \frac{X_0}{n} + \mu_X + \frac{\sigma_X}{n}\sum_{i = 1}^{n}\left(\frac{X_i - \mu_X}{\sigma_X}\right)
\end{displaymath}
and
\begin{displaymath}
    \frac{1}{n} \sum_{ i = 1}^n Y_i = \mu_Y + \frac{\sigma_Y}{n}\sum_{i = 1}^{n}\left(\frac{Y_i - \mu_Y}{\sigma_Y}\right).
\end{displaymath}
From \citep[Theorem 1.5]{Str(67)}, there exist two independent Brownian motions $\{B_t\}_{t \ge 0}$ and $\{B'_t\}_{t \ge 0}$ such that
\begin{displaymath}
   \sum_{i = 1}^{n}\left(\frac{X_i - \mu_X}{\sigma_X}\right) \stackrel{a.s.}{=} B_n + R_n \, , \hspace{20 pt} R_n: = \sum_{i = 1}^{n}\left(\frac{X_i - \mu_X}{\sigma_X}\right) - B_n
\end{displaymath}
\begin{displaymath}
   \sum_{i = 1}^{n}\left(\frac{Y_i - \mu_Y}{\sigma_Y}\right)  \stackrel{a.s.}{=} B'_n + R'_n \, , \hspace{20 pt} R'_n: = \sum_{i = 1}^{n}\left(\frac{Y_i - \mu_Y}{\sigma_Y}\right) - B'_n
\end{displaymath}
with  
\begin{displaymath}
    R_n = O \left( (n \log \log n)^{\frac{1}{4}}  (\log n)^{\frac{1}{2}}\right) = o \left(n^{\frac{1}{4} + \varepsilon}\right) \hspace{10 pt} \forall \varepsilon>0
\end{displaymath}
\begin{displaymath}
    R'_n = O \left( (n \log \log n)^{\frac{1}{4}}  (\log n)^{\frac{1}{2}}\right) = o \left(n^{\frac{1}{4} + \varepsilon}\right) \hspace{10 pt} \forall \varepsilon>0.
\end{displaymath}
Recalling that, from the law of iterated logarithm
\begin{displaymath}
    \frac{|B_n|}{n} = O\left(\sqrt{\frac{\log \log n}{n}}\right) \hspace{20pt} \text{and} \hspace{20pt}  \frac{|B'_n|}{n} = O\left(\sqrt{\frac{\log \log n}{n}}\right),
\end{displaymath}
one gets
\begin{align*}
    \left(\frac{1}{n} \sum_{ i = 0}^n X_i \right)^{1-\alpha} & \stackrel{a.s.}{=} \mu_X^{1-\alpha} \left[1 + \frac{X_0}{n\mu_X}  + \frac{\sigma_X}{n \mu_X }B_n  + \frac{\sigma_X}{n\mu_X }R_n\right]^{1-\alpha}\\
    &  = \mu_X^{1-\alpha} \left[1 +  \frac{(1-\alpha) \, \sigma_X}{n \mu_X }B_n  + o \left(n^{-\frac{3}{4} + \varepsilon}\right)\right]
\end{align*}
and 
\begin{align*}
    \left(\frac{1}{n} \sum_{ i = 1}^n Y_i\right)^\alpha & \stackrel{a.s.}{=}  \mu_Y^\alpha \left[1 + \frac{\sigma_Y}{n\mu_Y}B'_n  + \frac{\sigma_Y}{n\mu_Y}R'_n\right]^\alpha \\
    & =  \mu_Y^\alpha \left[1 + \frac{\alpha \, \sigma_Y}{n\mu_Y} B'_n  +  o \left(n^{-\frac{3}{4} + \varepsilon}\right) \right]
\end{align*}
for any $\varepsilon >0$. This implies
\begin{multline*}
     \sqrt{n}  \left[ \left(\frac{1}{n}\sum_{i = 0}^n X_i\right)^{1-\alpha}  \left(\frac{1}{n}\sum_{i = 1}^n Y_i\right)^{\alpha} - \mu_X^{1-\alpha} \mu_Y^\alpha \right] \\ \stackrel{a.s.}{=} \mu_X^{1-\alpha} \mu_Y^\alpha \left[\frac{(1-\alpha) \, \sigma_X}{ \mu_X }\frac{B_n}{\sqrt{n}}  + \frac{\alpha \, \sigma_Y}{\mu_Y} \frac{B'_n}{\sqrt{n}} + o \left(n^{-\frac{1}{2} + \varepsilon}\right)\right].
\end{multline*}
Observe that, by elementary properties of the two independent Brownian motions $\{B_t\}_{t \ge 0}$ and $\{B'_t\}_{t \ge 0}$ ,
\begin{displaymath}
  \mu_X^{1-\alpha} \mu_Y^\alpha\left[  \frac{(1-\alpha) \, \sigma_X}{ \mu_X }\frac{B_n}{\sqrt{n}}  + \frac{\alpha \, \sigma_Y}{\mu_Y} \frac{B'_n}{\sqrt{n}} \right]  \stackrel{d}{=} \mathcal{N}\left(0, V^2\right).
\end{displaymath}
Finally, apply Slutsky's theorem to conclude the proof.
\end{proof}

\begin{cor}
\label{cor:cor1_quater}
Let  $G_{\rho n + \tau,  B} \sim \operatorname{Gamma}(\rho n + \tau,  B)$ and $G_{(\lambda+1)n, 1} \sim \operatorname{Gamma}((\lambda+1)n, 1)$ be two independent random variables. Then, as $n \to + \infty$, 
 \begin{displaymath}
       \frac{1}{\sqrt{n} }\left[ G_{\rho n + \tau,  B}
        ^{1-\alpha}  G_{(\lambda+1)n, 1}^\alpha- nm \right] \stackrel{\text{w}}{\longrightarrow} \mathcal{N}(0, V^2),
   \end{displaymath}
where 
\begin{displaymath}
    m := \left(\frac{\rho}{B}\right)^{1-\alpha} (\lambda+1)^\alpha
\end{displaymath}
and
\begin{displaymath}
    V^2 := \left(\frac{\rho}{B}\right)^{2(1-\alpha)} (\lambda+1)^{2\alpha} \left[ \frac{(1-\alpha)^2}{\rho} + \frac{\alpha^2}{\lambda+1}\right].
\end{displaymath}
\end{cor}

\begin{proof}
    Introduce a random variable $X_0$ and two sequences of i.i.d. random variables $\{X_i\}_{i \ge 1}$ and $\{Y_i\}_{i \ge 1}$ such that: 
   \begin{itemize}
   \item[i)] $X_1 \sim \operatorname{Gamma}(\rho, B)$, $Y_1 \sim \operatorname{Gamma}(\lambda +1, 1)$ and  $X_0 \sim \operatorname{Gamma}(\tau, B)$;
   \item[ii)] $X_0$, $\{X_i\}_{i \ge 1} $ and $\{Y_i\}_{i \ge 1} $ are independent.
   \end{itemize}
   Since $G_{\rho n + \tau,  B} \stackrel{d}{=} \sum_{i = 0}^n X_i$ and $G_{(\lambda+1)n, 1}  \stackrel{d}{=} \sum_{i = 1}^n Y_i$, the statement immediately follows from Lemma \ref{lem:lem_3_quater}. The expressions of $m$ and $V^2$ can be checked via direct computation.
\end{proof}
\end{paragraph}

\begin{paragraph}{\underline{Step 3} of the proof of the CLT \eqref{eq:clt_zn}}
The next statement makes use of the so-called Ky-Fan distance between random variables, namely
    \begin{displaymath}
        d_{\text{KF}}\left(X, Y \right) := \inf \left\{ \varepsilon >0 \, : \, P \left(\left|X-Y \right|> \varepsilon\right) \le \varepsilon \right\}.
    \end{displaymath}

\begin{lem}
    \label{lem:lem5_quater}
    On a probability space $(\Omega, \mathcal{F}, P)$, consider a real random variable $\eta$, and two sequences of real random variables $\{\eta_n\}_{n \ge 1}$, $\{(\xi_n, \,  \xi'_n)\}_{n \ge 1}$ such that:
    \begin{itemize}
        \item[i)] $\eta_n \stackrel{w}{\longrightarrow} \eta$, as $n \to +\infty$;
    \item[ii)] 
    \begin{equation}
    \label{KF}
        d_{\text{KF}}(\xi_n, \xi_n') = o\left(\frac{1}{\sqrt{n}}\right);
    \end{equation}
    \item[iii)]  there exist $c \in \mathbb{R}$ and a real random variable $L$ such that, as $n \to +\infty$
     \begin{equation}
     \label{item: iii_lem5_quater}  
    \sqrt{n} \left(\xi_n' \eta_n - c\right) \stackrel{w}{\longrightarrow} L.
     \end{equation}
    \end{itemize}
    Then, as $n \to +\infty$, 
    \begin{displaymath}
        \sqrt{n} \left(\xi_n \eta_n - c\right) \stackrel{\text{w}}{\longrightarrow} L,
    \end{displaymath}
with the same $c$ and $L$ as in \eqref{item: iii_lem5_quater}.
\end{lem}
\begin{proof}
Write
\begin{displaymath}
    \sqrt{n} \left(\xi_n \eta_n - c\right) = \sqrt{n} \left(\xi'_n \eta_n - c\right) + \sqrt{n} \eta_n \left(\xi_n - \xi'_n\right).
\end{displaymath}
By means of Slutsky's theorem, the proof is concluded if we show
\begin{displaymath}
    \sqrt{n}  \left(\xi_n - \xi'_n\right) \stackrel{p}{\longrightarrow}0 
\end{displaymath}
 as $n \to +\infty$. Since the Ky-Fan distance metrizes convergence in probability, it is enough to note that
 \begin{displaymath}
    d_{\text{KF}}\left(\sqrt{n} \left(\xi_n - \xi'_n\right), \mathbf{0}\right)  = d_{\text{KF}}\left(\sqrt{n}\xi_n,  \sqrt{n} \xi'_n\right) \le \sqrt{n} d_{\text{KF}}\left(\xi_n, \xi'_n\right) \stackrel{\eqref{KF}}{\longrightarrow} 0
\end{displaymath}
 where $\mathbf{0}$ denotes the degenerate random variable equal to 0 a.s.. 
\end{proof}

\begin{proof}[Proof of the CLT \eqref{eq:clt_zn}]
Set $\rho = \lambda \frac{1-\alpha}{\alpha}$, $\tau = \frac{1}{2}$, $B = (1-\alpha) \, \alpha^{\frac{\alpha}{1-\alpha}}$ as in Lemma \ref{lem:lem4_quater} and introduce an auxiliary random variable $G_{\rho n + \tau, B} \sim \operatorname{Gamma}\left( \rho n + \tau, B \right)$, independent of $G_{(\lambda+1)n, 1} $. Then, the assumptions of Corollary \ref{cor:cor1_quater} are met, with $m$ and $V^2$ as follows:
\begin{displaymath}
    m = \left(\frac{\rho}{B}\right)^{1-\alpha} (\lambda+1)^\alpha = \frac{\lambda}{\alpha}  \left(\frac{\lambda+1}{\lambda}\right)^\alpha = z_0
\end{displaymath}
and
\begin{align*}
    V^2 & = \left(\frac{\rho}{B}\right)^{2(1-\alpha)} (\lambda+1)^{2\alpha} \left[ \frac{(1-\alpha)^2}{\rho} + \frac{\alpha^2}{\lambda+1}\right] \\
    & = \frac{\lambda^2}{\alpha^2} \left(\frac{\lambda+1}{\lambda}\right)^{2\alpha} \left[\frac{(1-\alpha)\alpha}{\lambda} + \frac{\alpha^2}{\lambda+1}\right]  = \Sigma^2.
\end{align*}
Corollary \ref{cor:cor1_quater} then entails that,  as $n \to + \infty$
 \begin{displaymath}
       \frac{1}{\sqrt{n} }\left[ G_{\rho n + \tau,  B}
        ^{1-\alpha}  G_{(\lambda+1)n, 1}^\alpha- n z_0 \right] \stackrel{w}{\longrightarrow} \mathcal{N}(0, \Sigma^2).
   \end{displaymath}
Now, let $d_{\text{Prok}}$ denote the Prokhorov distance between distributions, defined as follows. 
Given $B \subseteq \mathbb{R}$ and $\varepsilon>0$, let $B^\varepsilon : =\{x \in \mathbb{R} \, : \, \inf_{y \in B} |x-y|^2 \le \varepsilon \}$; then
\begin{displaymath}
        d_{\text{Prok}}\left(\mu, \nu \right) :=  \inf \left\{\varepsilon >0 \, : \, \mu(B) \le \nu\left(B^\varepsilon\right) + \varepsilon  , \,  \forall B \in \mathscr{B}(\mathbb{R})\right\}.
    \end{displaymath}
Thanks to a well-known bound between probability metrics---see \citep[Inequality (4.13)]{Hub(81)}---Lemma \ref{lem:lem4_quater} implies that
\begin{displaymath}
        d_{\text{Prok}}\left(\mu_{S_{\lambda n, \alpha}}, \,  \mu_{G_{\rho n + \tau, B}
        ^{1-\alpha}}\right) \le d_{\text{TV}}\left(\mu_{S_{\lambda n, \alpha}}, \,  \mu_{G_{\rho n + \tau, B}
        ^{1-\alpha}}\right) =O\left(\frac{1}{n} \right).
    \end{displaymath}
Now, a theorem of Strassen---see  \citep[Corollary 11.6.4]{Dud(02)}---allows, in turn, to find a suitable coupling between the random variables  $S_{\lambda n, \alpha}$ and $G_{\rho n + \tau, B}^{1-\alpha}$ such that
\begin{displaymath}
         d_{\text{KF}}\left(S_{\lambda n, \alpha}, \,  G_{\rho n + \tau, B}^{1-\alpha}\right)  =O\left(\frac{1}{n} \right).
    \end{displaymath}
To conclude, it is enough to invoke Lemma \ref{lem:lem5_quater} with the choices $\eta_n =  G_{(\lambda+1)n, 1}^\alpha$, $\xi'_n = G_{\rho n + \tau,  B}
        ^{1-\alpha}$, $\xi_n = S_{\lambda n, \alpha}$, $c = z_0$ and $L = \mathcal{N}\left(0, \Sigma^2\right)$.
\end{proof}   
\end{paragraph}   


\subsubsection{On a property of uniform integrability of $Z_n$} 
We now prove a property of uniform integrability of $Z_n$ which, in combination with convergence in distribution, guarantees convergence of the moments. 
Hence, from the CLT we will deduce convergence of the moments of $(Z_n - n \, z_0)/(\sqrt{n \Sigma^2})$ to those of the standard Gaussian.
\begin{lem} Under the assumption of Proposition \ref{prop:clt_Z}, it holds
    \label{lem:UI_Zn}
    \begin{displaymath}
        \sup_{n \in \mathbb{N}} \mathbb{E}\left[\left(\frac{Z_n - nz_0}{\sqrt{n}}\right)^4\right] < + \infty.
    \end{displaymath}
\end{lem}
\begin{proof}
    We show that as $n \to +\infty$
    \begin{displaymath}
        \mathbb{E}\left[\left(Z_n - nz_0\right)^4\right]  = O\left(n^2\right).
    \end{displaymath}
In fact, 
\begin{align*}
    \mathbb{E}\left[\left(Z_n - nz_0\right)^4\right] & = \sum_{k = 0}^4 \binom{4}{k}\,  \mathbb{E} [Z_n^k]\,  (nz_0)^{4-k}\\
    & = \sum_{k = 0}^4 \binom{4}{k} \, \frac{\Gamma(\lambda n)}{\Gamma\left(\frac{\lambda n}{\alpha}\right)} \, \frac{\Gamma \left(k+ \frac{\lambda n}{\alpha}\right)}{\Gamma(k \alpha + \lambda n)} \, \frac{\Gamma \left((\lambda +1 ) n + k\alpha\right)}{\Gamma\left((\lambda +1 )n \right)} \times \,\\
    &\quad \times (-1)^{4-k} \, n^{4-k}\, \left(\frac{\lambda n}{\alpha}\right)^{4-k}\, \left(\frac{\lambda +1}{\lambda}\right)^{4-k}.
\end{align*}
First-order Stirling's approximation, combined with some simple computations, yields
\begin{align*}
     \mathbb{E}\left[\left(Z_n - nz_0\right)^4\right]  & = n^4 \, \left(\frac{\lambda}{\alpha}\right)^{4}\, \, \left(\frac{\lambda +1}{\lambda}\right)^{4}\sum_{k = 0}^4 \binom{4}{k} \,  (-1)^{4-k}\times  \\
    &\quad\times \left[1- \frac{k\alpha (k\alpha -1)}{2\lambda n} + \frac{k(k-1)\alpha}{2\lambda n} + \frac{k\alpha (k\alpha -1)}{2 (\lambda+1) n} + O\left(\frac{1}{n^2}\right) \right] \\
    & = c_0 \, n^4 \sum_{k = 0}^4 \binom{4}{k} \,  (-1)^{4-k} +  n^3 \sum_{k = 0}^4 \binom{4}{k} (-1)^{4-k} \left(c_1 \, k  + c_2 \, k^2\right) \\
    &\quad+ O(n^2)
\end{align*}
for some constants $c_0, \, c_1 \, , c_2$ depending only on $\alpha $ and $\lambda$. Now, it can be easily seen that 
\begin{displaymath}
     \sum_{k = 0}^4 \binom{4}{k} \,  (-1)^{4-k}\, k^i = 0
\end{displaymath}
for $i = 0, 1, 2$, which concludes the proof.
\end{proof}

By combining Lemma \ref{lem:UI_Zn} and Proposition \ref{prop:clt_Z} we obtain the following statement.
\begin{prp}
\label{lem:lemma6_quater(formula momenti UI)}
    Let $\Psi: (0, +\infty) \to \mathbb{R}$ be such that $\Psi \in C^3(0, +\infty)$ with bounded derivatives. Then
    \begin{displaymath}
        \mathbb{E} \left[ \Psi\left(\frac{Z_n}{n}\right)\right] = \, \Psi(z_0) + \frac{1}{2n} \left(\frac{\lambda+1}{\lambda}\right)^\alpha \frac{1-\alpha}{\lambda +1} \Psi'(z_0) + \frac{1}{2n} \Sigma^2\,  \Psi''(z_0) + o\left(\frac{1}{n}\right).
    \end{displaymath}
\end{prp}
\begin{proof}
    Let $e_n: = \mathbb{E} \left[ \Psi\left(\frac{Z_n}{n}\right)\right]$. By means of \citep[Equation 1]{TE(51)}, we write
    \begin{align*}
        e_n &= \frac{1}{n} \left[n\,  \frac{\lambda }{\alpha} \left(\frac{\lambda +1}{\lambda}\right)^\alpha \left(1+\frac{\alpha(1-\alpha)}{2n \lambda (\lambda +1)}\right) \right] + o\left(\frac{1}{n}\right) \\
        &= z_0 + \frac{1}{2n} \left(\frac{\lambda +1}{\lambda} \right)^\alpha \frac{1-\alpha}{\lambda +1} + o\left(\frac{1}{n} \right).
    \end{align*}
Taylor's formula with Bernstein's integral remainder \citep[Appendix B]{Dud(02)} shows that
\begin{align*}
\Psi \left(\frac{Z_n}{n}\right) - \Psi(e_n) &= \Psi'(e_n) \left(\frac{Z_n}{n} -e_n\right) + \frac{1}{2} \Psi''(e_n) \left(\frac{Z_n}{n} -e_n\right)^2 \\
&\quad + \frac{1}{2}\left(\frac{Z_n}{n} -e_n\right)^3 \int_0^1 \Psi''' \left(e_n + s \left(\frac{Z_n}{n} -e_n\right)\right) \, (1-s)^2\, \mathrm{d}s.
\end{align*}
Taking expectation on both sides we obtain
\begin{align*}
\mathbb{E}\left[\Psi \left(\frac{Z_n}{n}\right)\right]  &= \Psi(e_n) + \frac{1}{2} \Psi''(e_n) \, \operatorname{Var} \left(\frac{Z_n}{n} \right) \\
&\quad + \frac{1}{2} \mathbb{E}\left[\left(\frac{Z_n}{n} -e_n\right)^3  \int_0^1 \Psi''' \left(e_n + s \left(\frac{Z_n}{n} -e_n\right)\right) \, (1-s)^2\, \mathrm{d}s \right].
\end{align*}
By applying again Taylor's formula to $\Psi'$ and $\Psi''$, the last term becomes
\begin{align*}
\mathbb{E}\left[\Psi \left(\frac{Z_n}{n}\right)\right]  &= \Psi(z_0) + \frac{1}{2n} \, \left(\frac{\lambda+1}{\lambda}\right)^\alpha\frac{1-\alpha}{\lambda +1 }\, \Psi'(z_0)  + \frac{1}{2}\, \Psi''(z_0) \, \frac{n \Sigma^2 + O(1)}{n^2} \\
&\quad + \frac{1}{2} \mathbb{E}\left[\left(\frac{Z_n}{n} -e_n\right)^3  \int_0^1 \Psi''' \left(e_n + s \left(\frac{Z_n}{n} -e_n\right)\right) \, (1-s)^2\, \mathrm{d}s \right].
\end{align*}
Since $\Psi'''$ is bounded, to reach the conclusion it is enough to prove that
\begin{displaymath}
    \lim_{n \to +\infty }   n \, \mathbb{E}\left[\left| \frac{Z_n}{n} - e_n\right|^3\right] = \lim_{n \to +\infty }   n \, \mathbb{E}\left[\left| \frac{Z_n}{n} - z_0\right|^3\right] =0.
\end{displaymath}
Now, by combining Lemma \ref{lem:UI_Zn} with Proposition \ref{prop:clt_Z}, we get
\begin{displaymath}
\lim_{n \to +\infty} \mathbb{E}\left[\left| \frac{Z_n - nz_0}{\sqrt{n}}  \right|^3\right] = \int_\mathbb{R} |x|^3 \, \mathrm{d} \Phi(x)
\end{displaymath}
in view of the uniform integrability. Since $n \, \left| \frac{Z_n}{n} - z_0\right|^3 = \frac{1}{\sqrt{n}} \left| \frac{Z_n - nz_0}{\sqrt{n}}\right|^3$, we finally obtain that
\begin{displaymath}
   \lim_{n \to +\infty} \mathbb{E}\left[ \frac{1}{\sqrt{n}} \left| \frac{Z_n - nz_0}{\sqrt{n}}  \right|^3\right]
    =  \lim_{n \to +\infty} \frac{1}{\sqrt{n}} \times \lim_{n \to +\infty} \mathbb{E}\left[\left| \frac{Z_n - nz_0}{\sqrt{n}}  \right|^3\right] = 0.
   \end{displaymath}
This concludes the proof.
\end{proof}


\subsubsection{Proof of Proposition \ref{prop:clt_R}}\label{sec:BE_Rn}

\begin{proof}[Proof of the Berry-Esseen inequality \eqref{eq:be_rn}] Denote by $G_{R_n(z)}$ the probability generating function of $R_n(z)$. Then, the proof of Proposition \ref{prop:clt_R} is based on the study of the large $n$ asymptotic behavior of $G_{R_n(z)}$ and its derivatives. This study is based on the following steps.
\begin{enumerate}
\item State four preparatory results of technical nature, that are Lemmata \ref{lem:Jed}, \ref{lem:xi_Paris}, \ref{lem:paris}, and \ref{lem: sup is continuous}. 
\item Show how the functions $\mu(z)$ and $\sigma(z)$ in Proposition \ref{prop:clt_R} are related to the moments of $R_n(z)$.
\item Prove a Berry-Esseen inequality for $R_n(z)$; 
\item Conclude the proof of the Berry-Esseen inequality \eqref{eq:be_rn}.
\end{enumerate}
\end{proof}

\begin{paragraph}{\underline{Step 1} of the proof of the Berry-Esseen inequality \eqref{eq:be_rn}} 
The first lemma allows to write a quantity of interest in terms of the so-called Kr\"atzel function 
(see \citep[formula (1.1)]{Par(21)}), namely
\begin{displaymath}
    F_{p, \nu}(w): = \int_0^{+\infty} t^{\nu - 1} \, \exp \left\{t^p - \frac{w}{t} \right\}  \, \mathrm{d} t
\end{displaymath}
for any $w \in \mathbb{C}^+ := \{ w \in \mathbb{C} \, : \, Re(w)>0\}$. A standard check shows that $F_{p, \nu}$ is holomorphic on the whole of $\mathbb{C}^+$. 

\begin{lem}\label{lem:Jed}
\label{lem: prelim J}
   For $\varepsilon\geq 0, \, \delta\geq 0 , \,  y \in \mathbb{C}^+$ define
   \begin{equation}
   \label{eq: def_prel_J}
       J_{\varepsilon, \delta}^{(n)} (y) : = \int_0^{+\infty} x^{n+\varepsilon -\frac{2-\alpha}{2(1-\alpha)} + \frac{\delta\alpha}{1-\alpha}} \, \exp\left\{-x (ny)^\frac{1}{\alpha} - (1-\alpha) \left(\frac{\alpha}{x}\right)^\frac{\alpha}{1-\alpha}\right\} \, \mathrm{d}x.
   \end{equation}
Then, it holds that
   \begin{displaymath}
       J_{\varepsilon, \delta}^{(n)} (y) =  \mathcal{C}(\alpha, n, y) \, \left(\frac{1}{ny}\right)^{\frac{\varepsilon}{\alpha} + \frac{\delta}{1-\alpha}} \, \mathcal{F}^{(n)}_{\varrho(\varepsilon, \delta)}(y),
   \end{displaymath}
   where:
   \begin{align} 
    \label{Ccal_alphany}
       \mathcal{C}(\alpha, n, y) &:= \frac{1-\alpha}{\alpha} \, \left(\frac{1}{ny}\right)^{\frac{n}{\alpha} - \frac{2-\alpha}{2\alpha(1-\alpha)} + \frac{1}{\alpha}} \\
  \label{varrho_epsdel}
       \varrho(\varepsilon, \delta) &:= \frac{1}{2} - \delta - \varepsilon \frac{1-\alpha}{\alpha} \\
  \mathcal{F}^{(n)}_{\varrho}(y) &:= F_{\frac{1-\alpha}{\alpha}, \, n \frac{1-\alpha}{\alpha} - \varrho} \left((1-\alpha) \left(\alpha^\alpha ny\right)^\frac{1}{1- \alpha}\right). \nonumber
   \end{align}
\end{lem}

\begin{proof}
First, fix $y \in (0, +\infty)$ and prove the result by direct computation, using the substitution $t = \left[ x (ny)^{\frac{1}{\alpha}}\right]^\frac{\alpha}{1-\alpha}$ in the integral on the right-hand side of \eqref{eq: def_prel_J}. 
Notice that both sides of the identity at issue are analytic function of $y$, as $y$ varies in $\mathbb{C}^+$. To be precise, we are considering the principal branch of each power of $y$, which is holomorphic on 
$\mathbb{C}\setminus (-\infty, 0]$. Finally, since both sides of \eqref{eq: def_prel_J} coincide on $(0, +\infty)$, then they must coincide on the whole of $\mathbb{C}^+$ by uniqueness of the analytic continuation. 
\end{proof}

The next lemma is concerned with equation \eqref{eq:tauParis}. To simplify its statement, we rewrite \eqref{eq:tauParis} in a simpler form. For fixed $v \in (0,+\infty)$, let $\xi(v)$ denote the only real, positive solution to the equation
\begin{equation}
\label{tilde_tau}
    \xi(v)^{\frac{1}{\alpha}} = v\xi(v) +1.
\end{equation}
Indeed, the regularity of the mapping $v \mapsto \xi(v)$, contained in the next statement, is an important issue in what follows. 

\begin{lem}\label{lem:xi_Paris}
The mapping $v \mapsto \xi(v)$ admits an analytical continuation to the whole of $\mathbb{C}^+$. In particular, $v \mapsto \xi(v)$  is analytic on $(0,+\infty)$.
\end{lem}

\begin{proof}
First, the existence of the mapping $v \mapsto \xi(v)$ from $[0, +\infty)$ into $[1, +\infty)$ is checked by basic calculus tools. Then, a direct application of \citep[Equation (11.5)]{Bel(19)}), with the same $\alpha$ as in \eqref{tilde_tau},
$\beta = \alpha$, $x = \xi(v)^{\frac{1}{\alpha}}$, and $y = v$, shows that the function
 \begin{displaymath}
 \xi_1(v) := \alpha \sum_{n=0}^{+\infty} \frac{\Gamma(\alpha n + \alpha)}{\Gamma\big((\alpha-1)n + \alpha + 1\big)} \frac{v^n}{n!}
 \end{displaymath}
provides a solution to \eqref{tilde_tau}. More precisely, the above series converges inside the disc $|v| < \alpha^{-\alpha} (1-\alpha)^{\alpha - 1}$, and meets $\xi_1(v) = 1 + \alpha v + O(v^2)$ as $v \to 0$. Whence, $\xi(v) =  \xi_1(v)$
for any $v \in [0, \alpha^{-\alpha} (1-\alpha)^{\alpha - 1})$. To get closer to the goal, it is crucial to notice that $\xi_1$ is a special case of a so-called H-function (or Fox function), whose theory is exposed, e.g., in 
\citep[Chapters 1--2]{KiSa(04)}) and \citep[Chapter 2]{PaKa(01)}). Exploiting this, it is worth considering
\begin{displaymath}
\xi_2(v) := \frac{\alpha}{2\pi i} \int_{\gamma_2} \frac{\Gamma(\alpha s + \alpha) \Gamma(-s)}{\Gamma\big((\alpha-1)s + \alpha + 1\big)} (-v)^s \mathrm{d}s
\end{displaymath}
where $\gamma_2$ is a path defined as follows: it stars at $-i\infty$; proceeds upward along the imaginary axis; before reaching the real axis, encircles the origin clock-wise, remaining on $\mathbb{C}^- := \{ w \in \mathbb{C} \, : \, Re(w)<0\}$ and leaving all the poles of the form $-1- m\alpha^{-1}$, $m \in \N_0$, to its left; proceeds again along the imaginary axis towards $+i\infty$. See Fig 1 below. 

\begin{figure}[h]
\begin{minipage}{0.4\textwidth}
\centering
\includegraphics [width= 1\textwidth, page = 1]{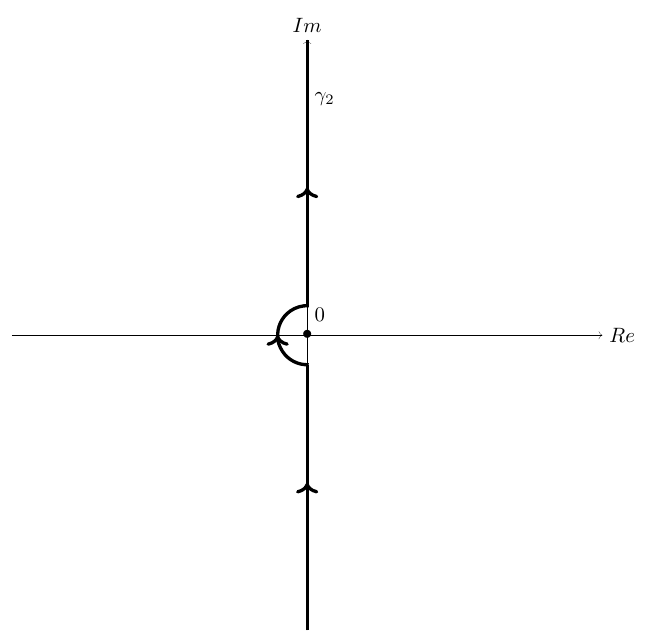}
\end{minipage}
\begin{minipage}{0.4\textwidth}
\centering
\includegraphics [width= 1\textwidth, page = 2]{figs.pdf}
\end{minipage}
\caption{The paths $\gamma_2$ (left) and $\gamma_{2, N}$(right)}
\end{figure}

At this stage, an application of \citep[Theorem 1.1, Formula (1.2.20)]{KiSa(04)} with $a^{\ast} = 1 + \alpha - (1-\alpha) = 2\alpha$ therein shows that $\xi_2$ is well-defined and holomorphic in the domain 
$\{v \in \mathbb{C}\ :\ \ |\text{Arg}(v)| < \pi \alpha\}$, where $\text{Arg}$ stands for the principal argument. Therefore, recalling that 
\begin{displaymath}
Res\left(s \mapsto \Gamma(-s); n\right) = \frac{(-1)^{n+1}}{n!} \qquad (n \in \N_0), 
\end{displaymath}
the residue theorem entails that 
\begin{displaymath}
\frac{1}{2\pi i} \oint_{\gamma_{2,N}} \frac{\Gamma(\alpha s + \alpha) \Gamma(-s)}{\Gamma\big((\alpha-1)s + \alpha + 1\big)} (-v)^s \mathrm{d}s = 
\sum_{n=0}^{N} \frac{\Gamma(\alpha n + \alpha)}{\Gamma\big((\alpha-1)n + \alpha + 1\big)} \frac{v^n}{n!}
\end{displaymath}
holds for any $v \in \mathbb{C}$ and $N \in \N$, where $\gamma_{2,N}$ is the closed path that starts at $-i(N + 1/2)$, proceeds exactly as $\gamma_2$ up to $+i(N + 1/2)$, and comes back to $-i(N + 1/2)$ along the semicircle
of radius $(N + 1/2)$, centered at the origin, routed clock-wise. See Fig 1. Exploiting the estimates displayed in \citep[Section 1.2]{KiSa(04)}, we let $N$ go to infinity to conclude that $\xi_1(v) = \xi_2(v)$ on 
$\{ w \in \mathbb{C} \, : \, |v| < \alpha^{-\alpha} (1-\alpha)^{\alpha - 1}, |\text{Arg}(v)| < \pi \alpha\}$.  Whence, $\xi(v) =  \xi_2(v)$ for any $v \in (0, \alpha^{-\alpha} (1-\alpha)^{\alpha - 1})$. In order to include the half-line
$(\alpha^{-\alpha} (1-\alpha)^{\alpha - 1}, +\infty)$, we invoke again \citep[Equation (11.5)]{Bel(19)}), with $-\alpha(1-\alpha)^{-1}$ in the place of $\alpha$ therein, $\beta = \alpha(1-\alpha)^{-1}$, 
$x = v^{-1} \xi(v)^{\frac{1-\alpha}{\alpha}}$, and $y = v^{\frac{1}{1-\alpha}}$. We get that the function
 \begin{displaymath}
 \xi_3(v) := \frac{\alpha}{1 - \alpha} v^{\frac{\alpha}{1-\alpha}} \sum_{n=0}^{+\infty} \frac{\Gamma\left(\frac{n -\alpha}{1- \alpha}\right)}{\Gamma\left(\frac{\alpha n + 1 -2\alpha}{1- \alpha}\right)} 
 \frac{(-1)^{n+1}}{n!}\left(v^{\frac{-1}{1-\alpha}}\right)^n
 \end{displaymath}
provides a solution to \eqref{tilde_tau}. More precisely, the above series converges if $|v| > \alpha^{-\alpha} (1-\alpha)^{\alpha - 1}$, and represents a holomorphic function on the region
$\{ w \in \mathbb{C} \, : \, |v| > \alpha^{-\alpha} (1-\alpha)^{\alpha - 1}, |\text{Arg}(v)| < \pi\}$, provided that the powers of $v$ are intended as their the principal branches. Since 
$\xi_3(v) = v^{\frac{\alpha}{1-\alpha}} + O(v^{-1})$ as $v \to +\infty$, we deduce that $\xi(v) =  \xi_3(v)$ for any $v \in (\alpha^{-\alpha} (1-\alpha)^{\alpha - 1}, +\infty)$. Now, we assume temporarily that $\alpha$ is not 
of the form $q(q+1)^{-1}$ for some $q \in \N$. This way, none of the values of the form $\alpha - m(1-\alpha)$, $m \in \N$, can coincide with the origin.
Resorting once again to the theory of H-functions, we consider
\begin{displaymath}
\xi_4(v) := -\frac{\alpha}{1 - \alpha} v^{\frac{\alpha}{1-\alpha}} 
\frac{1}{2\pi i} \int_{\gamma_4} \frac{\Gamma\left(\frac{s -\alpha}{1- \alpha}\right) \Gamma(-s)}{\Gamma\left(\frac{\alpha s + 1 -2\alpha}{1- \alpha}\right)} \left(v^{\frac{-1}{1-\alpha}}\right)^s \mathrm{d}s
\end{displaymath}
where $\gamma_4$ is a path of the following form: it starts at $-i\infty$; proceeds upward along the imaginary axis before reaching the origin; turns right to cross upward the real axis between $\alpha$ and 1; avoids all the poles of the form 
$\alpha - m(1-\alpha)$, $m \in \N_0$, which lie at the right of the origin, by turning around them counter-clock-wise; crosses again the real axis downward between the origin itself and the  
least positive pole of the form $\alpha - m(1-\alpha)$; encircles the origin clock-wise, ranging in the fourth, third and second quadrant, respectively, by leaving all the negative poles of the form $\alpha - m(1-\alpha)$ to its left; proceeds again along the imaginary axis towards $+i\infty$. See the Fig 2 below. 

\begin{figure}[h]
\begin{minipage}{0.4\textwidth}
\centering
\includegraphics [width= 1\textwidth, page = 3]{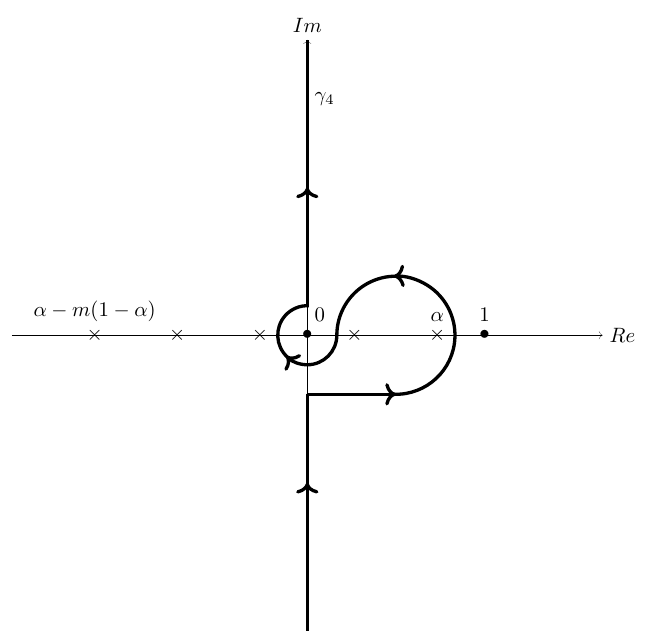}
\end{minipage}
\begin{minipage}{0.4\textwidth}
\centering
\includegraphics [width= 1\textwidth, page = 4]{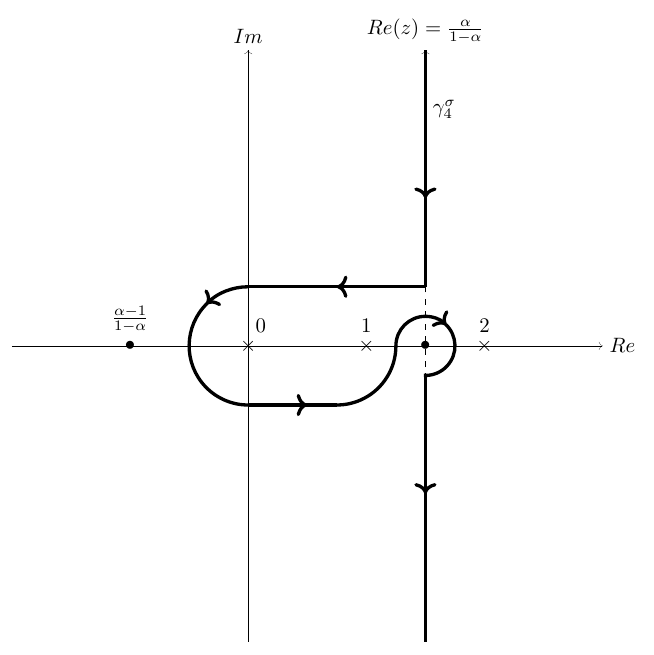}
\end{minipage}
\caption{The path $\gamma_4$ (left) and its transformation $\gamma_{4}^\sigma$ trough the map $s \mapsto \sigma(s) : = \frac{\alpha-s}{1-\alpha}$ (right)}
\end{figure}

Again, an application of \citep[Theorem 1.1, Formula (1.2.20)]{KiSa(04)} with $a^{\ast} = 1 - \frac{\alpha}{1-\alpha} + \frac{1}{1-\alpha} = 2$ therein shows that $\xi_4$ is well-defined and holomorphic in 
$\{v \in \mathbb{C}\ :\ \ |\text{Arg}(v)| < \pi(1- \alpha)\}$.
A similar argument, as the one above based of the residue theorem, leads us to conclude that $\xi_3(v) = \xi_4(v)$ on $\{ w \in \mathbb{C} \, : \, |v| > \alpha^{-\alpha} (1-\alpha)^{\alpha - 1}, |\text{Arg}(v)| < \pi (1- \alpha)\}$.  
Whence, $\xi(v) =  \xi_4(v)$ for any $v \in (\alpha^{-\alpha} (1-\alpha)^{\alpha - 1}, +\infty)$. It remains to show that $\xi_2(v) =  \xi_4(v)$ on $\{v \in \mathbb{C}\ :\ \ |\text{Arg}(v)| < \pi \alpha_{\ast}\}$, where
$\alpha_{\ast} := \min\{\alpha, 1-\alpha\}$. To this aim, we change the variable in the expression of $\xi_4$ by setting $\sigma = \frac{\alpha - s}{1- \alpha}$, to obtain
\begin{displaymath}
\xi_4(v) = \frac{\alpha}{2\pi i} \int_{\gamma_4^{\sigma}} \frac{\Gamma\big(1 - (\sigma(\alpha-1) + \alpha + 1)\big) \Gamma(-\sigma)}{\Gamma\big(1 - (\alpha + \alpha\sigma)\big)} v^{\sigma} \mathrm{d}\sigma
\end{displaymath}
where $\gamma_4^{\sigma}$ denotes the transformation of $\gamma_4$. See again Fig 2. Now, an application of the reflection formula of the gamma function yields
\begin{displaymath}
\xi_4(v) = \frac{\alpha}{2\pi i} \int_{-\gamma_4^{\sigma}} \frac{\Gamma(\alpha \sigma + \alpha) \Gamma(-\sigma)}{\Gamma\big((\alpha-1)\sigma + \alpha + 1\big)} 
\frac{\sin\big(\pi(\alpha \sigma + \alpha)\big)}{\sin\big(\pi((\alpha-1)\sigma + \alpha)\big)} v^{\sigma} \mathrm{d}\sigma
\end{displaymath}
where the minus in front of $\gamma_4^{\sigma}$ means a change of orientation.
Finally, by resorting to a well-known Cauchy theorem, we can deform the path $-\gamma_4^{\sigma}$ into a new path that consists of the conjunction of $\gamma_2$ with a loop the encircles the point $\frac{\alpha}{1-\alpha}$ 
counter-clock-wise, without changing the value of the integral. See Fig 3 below.

\begin{figure}[h]
\centering
\includegraphics [width= 0.4\textwidth, page = 5]{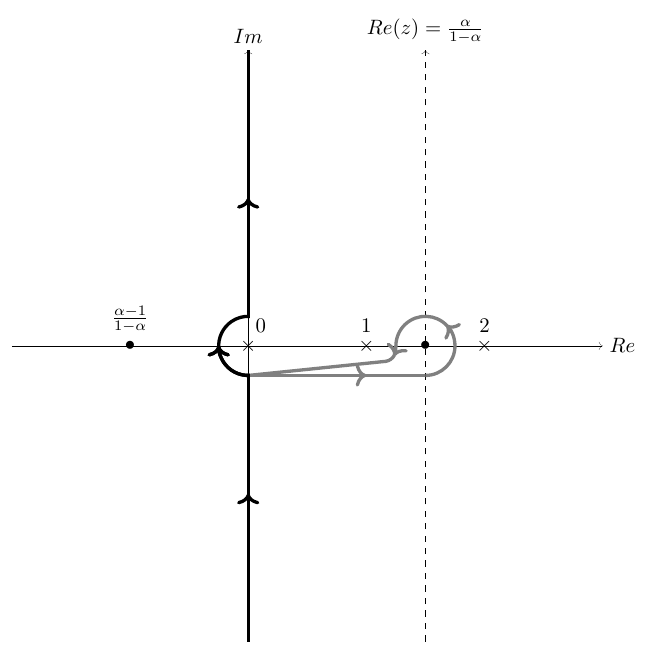}
\caption{Modification of the path $-\gamma_4^\sigma$, equal to the conjunction of $\gamma_2$ (black in the figure) and a loop around the point $\frac{\alpha}{1-\alpha}$ (gray in the figure).}
\end{figure}

Thus, setting
\begin{displaymath}
g(\sigma; v) :=  \frac{\Gamma(\alpha \sigma + \alpha) \Gamma(-\sigma)}{\Gamma\big((\alpha-1)\sigma + \alpha + 1\big)} \frac{\sin\big(\pi(\alpha \sigma + \alpha)\big)}{\sin\big(\pi((\alpha-1)\sigma + \alpha)\big)} v^{\sigma}
\end{displaymath}
we have
\begin{displaymath}
\xi_4(v) = \frac{\alpha}{2\pi i} \int_{\gamma_2} g(\sigma; v) \mathrm{d}\sigma + \alpha\ Res\left(g(\cdot ; v); \frac{\alpha}{1-\alpha}\right).
\end{displaymath}
Repeating the above argument based on the approximation of $\gamma_2$ with $\gamma_{2,N}$, we get
\begin{displaymath}
\frac{1}{2\pi i} \oint_{\gamma_{2,N}} g(\sigma; v) \mathrm{d}\sigma = - \sum_{n=0}^N Res\left(g(\cdot ; v); n\right) - \sum_{n=0}^{N_\alpha} Res\left(g(\cdot ; v); \frac{n+\alpha}{1-\alpha}\right)
\end{displaymath}
with $N_\alpha := [(N+1/2)(1-\alpha)-\alpha]$. At this stage, it is enough to notice that 
\begin{displaymath}
Res\left(g(\cdot ; v); \frac{n+\alpha}{1-\alpha}\right) = 0
\end{displaymath}
for any $n \in \N$. Therefore, taking the limit as $N \to +\infty$, we can conclude that $\xi_4(v) = \xi_1(v)$ on $\{ w \in \mathbb{C} \, : \, |v| < \alpha^{-\alpha} (1-\alpha)^{\alpha - 1}, |\text{Arg}(v)| < \pi \alpha_{\ast}\}$ and, consequently,
that $\xi_4(v) = \xi_2(v)$ on $\{ w \in \mathbb{C} \, : \, |\text{Arg}(v)| < \pi \alpha_{\ast}\}$. 
Finally, after proving the last identity, we can remove the constraint that $\alpha \neq q(q+1)^{-1}$ for some $q \in \N$ thanks to a straightforward approximation argument with respect to $\alpha$, to conclude that 
$\xi_4(v) = \xi_2(v)$ on $\{ w \in \mathbb{C} \, : \, |\text{Arg}(v)| < \pi \alpha_{\ast}\}$ whichever $\alpha \in (0,1)$ may be. At this stage, the remark that $\mathbb{C}^+ \subseteq  
\left\{ w \in \mathbb{C} \, : \, |\text{Arg}(v)| < \pi \max\{\alpha, 1-\alpha\} \right\}$ leads to identify the desired analytical continuation of $v \mapsto \xi(v)$ with either $v \mapsto \xi_2(v)$ or $v \mapsto \xi_4(v)$ according to
whether $\alpha \geq 1/2$ or $\alpha < 1/2$, respectively. This concludes the proof.
\end{proof}

Now, we investigate the large $n$
asymptotic behaviour of the function $\mathcal{F}^{(n)}_{\varrho}$ that appears in Lemma \ref{lem:Jed}. Our Lemma \ref{lem:paris} below can be regarded as a re-adaptation of \citep[formula (5.3)]{Par(21)}. 
We just premise the definition of a useful notation.
\begin{defi}
    \label{def: O_uc}
    Let $\mathcal{D}$ be some subset of $\mathbb{R}^d$ (or $\mathbb{C}$).
    Consider two sequences of functions $\{f_n\}_{n \in \mathbb{N}}$, $\{g_n\}_{n \in \mathbb{N}}$, with $f_n,\,  g_n: \mathcal{D} \to \mathbb{C}$ for every $n$. If, for some $\beta \in \mathbb{R}$, there exists a function $h : \mathcal{D} \to \mathbb{C}$ such that 
    \begin{displaymath}
        n^{-\beta} \left[f_n - g_n \right] \rightarrow h
    \end{displaymath}
    uniformly on compact sets, we write
    \begin{displaymath}
        f_n = g_n + O_{u.c.}(n^\beta).
    \end{displaymath}
\end{defi}

Let now recall some previous notation. For $y \in \mathbb{C}^+$, set $\tau(y) := \xi(1/(\alpha y))$, with the same function $\xi$ as in Lemma \ref{lem:xi_Paris}, and $D(y)$ be as in \eqref{eq:defD}, with $y$ in place of $z$. 
Moreover, set
    \begin{equation} \label{gParis}
    g(y ) := \frac{y}{\tau(y)} - \frac{1-\alpha}{\alpha} \log \tau(y)
    \end{equation}
where $\log$ is intended as the principal branch of the logarithm.

\begin{lem} 
    \label{lem:paris}
    With the above notation, for $y \in \mathbb{C}^+$ it holds that
     \begin{equation} \label{Paris_readapted}
     \mathcal{F}^{(n)}_{\varrho}(y) =  \mathcal{P}^{(n)}_{\varrho}(y) \, \left[1 + \frac{1}{n} \mathcal{R}_\varrho (y) + r_\varrho^{(n)}(y) \right]
     \end{equation}
    where:
    \begin{enumerate}
    \item[i)] the principal part $\mathcal{P}^{(n)}_{\varrho}$ is defined as
    \begin{displaymath}
     \mathcal{P}^{(n)}_{\varrho}(y) := \left(\alpha n y \right)^{n- \varrho \frac{\alpha}{1-\alpha}} \left(\frac{1}{\tau(y)}\right)^\varrho \, \exp\left\{-n \left[1+g(y)\right]\right\} \, \sqrt{\frac{2\pi \alpha }{n(1-\alpha) D(y) }};
   \end{displaymath}
    \item[ii)] the first order remainder $\mathcal{R}_\varrho$ is given by
   \begin{displaymath}
       \mathcal{R}_{\varrho}(y) = \varrho^2 \, c_1(y) + \varrho \, c_2(y) + c_3(y)
   \end{displaymath}
   with 
   \begin{displaymath}
    c_1(y) := \frac{\alpha}{2(1-\alpha)D(y)} \hspace{20pt} \text{and} \hspace{20pt}  c_2(y) := \frac{1-2\alpha}{2(1-\alpha) D(y)}+\frac{1}{2D^2(y)}
\end{displaymath}
   and another analytic function $c_3$ independent of $\varrho$;
    \item[iii)] the second order remainder $r_\varrho^{(n)}$ satisfies, as $n \to +\infty$ 
   \begin{displaymath}
       r_\varrho^{(n)}(y)  = O_{u.c.} \left(\frac{1}{n^2}\right).
   \end{displaymath}
   \end{enumerate}
\end{lem}
\begin{proof}
Fix $y \in (0,+\infty)$. Following \citep[Section 5]{Par(21)}, define
\begin{displaymath}
    X(y) : = (1-\alpha) \, ny
\end{displaymath}
and
\begin{displaymath}
    a_{n, \varrho}(y) : = \frac{1}{\alpha y} - \frac{\varrho + 1}{(1-\alpha) ny}.
\end{displaymath}
Notice that $a_{n, \varrho}(y) X(y) = n \frac{1-\alpha}{\alpha} - \varrho -1$. Set 
\begin{displaymath}
    \tau_n(y) := \xi(a_{n, \varrho}(y))\ ,
\end{displaymath}
where $\xi$ denotes the solution to \eqref{tilde_tau}. 
Now, a strict application of  \citep[Formula (5.3)]{Par(21)} yields 
\begin{align}
\label{Paris_readapted2}
\mathcal{F}_\varrho^{(n)}(y) &= \sqrt{\frac{2 \pi}{(1-\alpha) n y}} \, \left[\frac{1}{\alpha \tau_n(y)}  + \frac{1-\alpha}{\alpha} \, \frac{1}{\alpha y} - \frac{\varrho +1}{\alpha n y}\right]^{-\frac{1}{2}} \, 
(\alpha n y)^{n - \frac{\varrho \alpha}{1-\alpha}} \ \times \, \\ 
&\quad\times \left(\frac{1}{\tau_n(y)}\right)^{\varrho} \, \exp\left\{ -n \left[ 1 + \frac{y}{\tau_n(y)} - \frac{1-\alpha}{\alpha} \, \log(\tau_n(y))\right] + \frac{\alpha (\varrho + 1)}{1-\alpha}\right\}\ \times \nonumber \\
&\quad\times \left[1 + \frac{\mathscr{R}_{1}(y)}{n} +O_{u.c.}\left(\frac{1}{n^2}\right)\right] \nonumber
\end{align}
where $\mathscr{R}_{1}$ is a suitable analytical function which does not depend on $\varrho$. At this stage, to prove \eqref{Paris_readapted}, we need further investigations on the quantity $\tau_n(y)$,
obtained by exploiting the regularity of $v \mapsto \xi(x)$ state in Lemma \ref{lem:xi_Paris}.
Accordingly, set $a(y) := (\alpha y)^{-1}$, and notice that $\tau(y) = \xi(a(y))$. Moreover, notice that
\begin{displaymath}
a_{n, \varrho}(y) = a(y) - \left(\frac{\varrho + 1}{(1-\alpha) y}\right) \frac{1}{n},
\end{displaymath}
which shows that $a_{n, \varrho}(y)$ can be regarded as a small increment (as $n \to +\infty$) of $a(y)$. Thus, a Taylor expansion of the function $\xi$ around $a(y)$ yields
\begin{displaymath}
 \xi(a_{n, \varrho}(y)) = \xi(a(y)) \left[1 + \frac{\mathscr{A}_\varrho(  y)}{n} + \frac{\mathscr{B}_\varrho( y)}{n^2} + O_{u.c.}\left(\frac{1}{n^3}\right)\right],
\end{displaymath}
where
\begin{displaymath}
    \mathscr{A}_\varrho( y) := - \frac{\alpha (\varrho +1)}{(1-\alpha) D(y) }
\end{displaymath}
and
\begin{displaymath}
    \mathscr{B}_\varrho( y) := - \mathscr{A}^2_\varrho( y) \, \frac{D(y) \, (1-3\alpha) + 1-\alpha }{2 \alpha \, D(y) }
\end{displaymath}
with $D(y)$ as in \eqref{eq:defD}. Three consequences of the above Taylor expansion are:
\begin{itemize}
\item[i)]
\begin{displaymath}
    \left(\frac{1}{\tau_n(y)}\right)^{\varrho} =  \left(\frac{1}{\tau(y)}\right)^{\varrho} \, \left[1 + \frac{1}{n} \, \mathscr{R}_{\varrho, 2}(y) + O_{u.c.} \left(\frac{1}{n^2}\right)\right]
\end{displaymath}
\item[ii)]
\begin{align*}
    &\left[\frac{1}{\alpha \tau_n(y)}  + \frac{1-\alpha}{\alpha} \, \frac{1}{\alpha y} - \frac{\varrho +1}{\alpha n y}\right]^{-\frac{1}{2}} \\
    &\quad =   \left[\frac{1}{\alpha \tau(y)}  + \frac{1-\alpha}{\alpha} \, \frac{1}{\alpha y} \right]^{-\frac{1}{2}} \, \left[1 + \frac{1}{n} \, \mathscr{R}_{\varrho, 3}(y) + O_{u.c.} \left(\frac{1}{n^2}\right)\right]
\end{align*}
\item[iii)]
\begin{align*}
   &\exp\left\{ -n \left[ 1 + \frac{y}{\tau_n(y)} - \frac{1-\alpha}{\alpha} \, \log(\tau_n(y))\right] + \frac{\alpha (\varrho + 1)}{1-\alpha}\right\}  \\
    &\quad = \exp\left\{ -n \left[ 1 + \frac{y}{\tau(y)} - \frac{1-\alpha}{\alpha} \, \log(\tau(y))\right] \right\} \, \left[1 + \frac{1}{n} \, \mathscr{R}_{\varrho, 4}(y)  + O_{u.c.} \left(\frac{1}{n^2}\right)\right],
\end{align*}
\end{itemize}
with
\begin{align*}
    \mathscr{R}_{\varrho, 2}(y) &:= \varrho \, \mathscr{A}_{\varrho}(y) \\
    \mathscr{R}_{\varrho, 3}(y) &:= \frac{1}{2} \, \left(\frac{\mathscr{A}_\varrho(y)}{\alpha \tau(y)}  + \frac{\varrho+1}{\alpha y}\right) \, \frac{\alpha y}{D(y)} \\
    \mathscr{R}_{\varrho, 4}(y) &:= \frac{y}{\tau(y)} \left[\mathscr{A}_\varrho(y)^2 - \mathscr{B}_\varrho(y) \right].
\end{align*}
At this stage, identity \eqref{Paris_readapted} follows by just plugging the above expansions i)-ii)-iii) into \eqref{Paris_readapted2} and patiently rearranging the terms. Finally, after establishing \eqref{Paris_readapted} 
for $y \in (0,+\infty)$, we conclude that it must hold also for $y \in \mathbb{C}^+$ by the uniqueness of the analytic continuation. 
\end{proof}

The following lemma states a basic calculus result, that 
guarantees the existence of a continuous function bounding from above any uniformly convergent sequence of functions on a compact set. 
Being unable to find a precise statement on some calculus textbook, we include here
both statement and proof, to maintain the paper self-contained.
\begin{lem}
    \label{lem: sup is continuous}
    Let $\mathbb{K} \subseteq \mathbb{R}^d$ (or $\mathbb{C}$) be a compact set. Consider a sequence of continuous functions $\{f_n\}_{n \in \mathbb{N}}$, $f_n : \mathbb{K} \to \mathbb{R}$ such that, as $n \to +\infty$, $f_n \rightarrow f$ uniformly for some (continuous) function $f:\mathbb{K} \to \mathbb{R}$. For each $x \in \mathbb{K}$, define $S(x) : = \sup_{n \in \mathbb{N}} \left|f_n(x) \right|$. Then $S:\mathbb{K} \to \mathbb{R}$ is continuous.
\end{lem}

To prove Lemma \ref{lem: sup is continuous}, we first need the following lemma. 
\begin{lem}\label{lem_aux_sup}
For $i=1, 2$, let $(\mathbb{M}_i, d_{\mathbb{M}_i})$ be metric spaces, endowed with the metric topology. Let $\mathbb{K}_i \subseteq \mathbb{M}_i$ be compact subsets, and let $F: \mathbb{K}_1 \times \mathbb{K}_2 \to \mathbb{R}$ be a continuous function on the product space. Then, $G: \mathbb{K}_1 \to \mathbb{R}$ defined by $G(x)  = \sup_{y \in \mathbb{K}_2} F(x, y)$ is continuous. 
\end{lem} 

\begin{proof}[Proof of lemma \ref{lem_aux_sup}]
Consider a convergent sequence $\{x_n\}_{n \in \N} \subseteq \mathbb{K}_1$ with $x_n \to \bar{x} \in \mathbb{K}_1$. We prove that 
\begin{displaymath}
G(\bar{x}) \le \liminf_{n \to +\infty} G(x_n) \le \limsup_{n \to +\infty} G(x_n) \le G(\bar{x}).
\end{displaymath}
Since $\mathbb{K}_2$ is compact and the real-valued function $y \mapsto F(\bar{x}, y)$ on $\mathbb{K}_2$ is continuous, there exists $y^\star \in \mathbb{K}_2$ such that
\begin{displaymath}
G(\bar{x}) = \sup_{y \in \mathbb{K}_2} F(\bar{x}, y) = F(\bar{x}, y^\star).
\end{displaymath}
By definition, $G(x_n) \ge F(x_n, y^\star)$ for every $n \in \N$. Furthermore, the real-valued function $x \mapsto F(x, y^\star)$ on $\mathbb{K}_1$ is continuous. Whence,
\begin{displaymath}
\liminf_{n \to +\infty} G(x_n) \ge \liminf_{n \to +\infty} F(x_n, y^\star) =  F(\bar{x}, y^\star) = G(\bar{x}).
\end{displaymath}
We prove the other inequality by contradiction. Assume there exists a subsequence $\{x_{n_j}\}_{j \in \N}$ of $\{x_n\}$ such that
\begin{displaymath}
G(x_{n_j}) \rightarrow L > G(\bar{x}).
\end{displaymath}
For each $j \in \N$, since the real-valued function $y \mapsto F(x_{n_j}, y)$ on $\mathbb{K}_2$ is continuous, there exists $y_j \in \mathbb{K}_2$ such that 
\begin{displaymath}
G(x_{n_j}) = \sup_{y \in \mathbb{K}_2} F(x_{n_j}, y) = F(x_{n_j}, y_j).
\end{displaymath}
Now, consider the sequence $\{y_j\}_{j \in \N} \subseteq \mathbb{K}_2$. By sequential compactness, there exists a subsequence $\{y_{j_k}\}_{k \in \N}$ converging to some point $\bar{y} \in \mathbb{K}_2$. Hence, by continuity of $F$, 
\begin{displaymath}
G(x_{n_{j_k}}) = F(x_{n_{j_k}}, y_{j_k}) \rightarrow F(\bar{x}, \bar{y}) \le G(\bar{x})
\end{displaymath}
which produces the contradiction. 
\end{proof}

\begin{proof}[Proof of lemma \ref{lem: sup is continuous}]
Let $\N^* = \N \cup \{\infty\}$  denote the Alexandroff compactification of $\N$ with the discrete topology. Further, define $F: \mathbb{K} \times \N^* \to [0, +\infty)$, 
\begin{displaymath}
F(x, n) = \begin{cases} | f_n(x) |& \text{ if } n \in \N \\[0.4cm]
|f(x)|& \text{ if } n = \infty.
\end{cases}
\end{displaymath}
If $F$ is continuous, then Lemma \ref{lem_aux_sup} entails that the function $G:\mathbb{K} \to \mathbb[0, +\infty)$, defined by
$$ G(x) := \sup_{n \in \N^*} F(x, n)=  \max\{ \sup_{n \in \N}  | f_n(x) |,  | f(x)|\} = \max\{ S(x),  |f(x)|\} $$
is continuous. Note that, for every $x \in \mathbb{K}$, $S(x)$ is either equal to the limit of $|f_n(x)|$ or it is a maximum, so that $S(x) \ge |f(x)|$. This implies that $S=G$, and that $S$ is continuous. 
It remains to prove that $F$ is continuous. Consider a convergent sequence $\{(x_k, n_k)\}_{k \in \N} \subseteq \mathbb{K} \times \N^*$, with $(x_k, n_k) \to (\bar{x}, \bar{n})$. We treat two cases separately:
\begin{itemize}
\item[i)] if $\bar{n} \in \N$, then the sequence $\{n_k\}$ is eventually constant with $n_k = \bar{n}$. Hence, by continuity of $f_{\bar{n}}$, 
\begin{displaymath}
F(x_k, n_k) = f_{n_k}(x_k) = f_{\bar{n}}(x_k) \rightarrow f_{\bar{n}}(\bar{x}) = F(\bar{x}, \bar{n});
\end{displaymath}
\item[ii)] if $\bar{n} = \infty$, then by uniform convergence of $\{f_n\}_{n \geq 1}$,
\begin{displaymath} 
F(x_k, n_k) = f_{n_k}(x_k) \rightarrow f(\bar{x})= F(\bar{x}, \infty).
\end{displaymath}
\end{itemize}
\end{proof}

\end{paragraph}

\begin{paragraph}{\underline{Step 2} of the proof of the Berry-Esseen inequality \eqref{eq:be_rn}} The following proposition connects the quantities $\mu(z)$ and $\sigma^2(z)$ that appear in Proposition \ref{prop:clt_R}
to the mean and variance of $R_n(z)$, respectively.
\begin{prp} As $z$ varies in $(0,+\infty)$ and $n \to \infty$, one has
    \label{prop: mu_sigma_Rn}
    \begin{displaymath}
        \mathbb{E}[R_n(z)] = n \mu(z) + O_{u.c.}(1)
    \end{displaymath}
    and
    \begin{displaymath}
        \operatorname{Var}(R_n(z)) = n \sigma^2(z) + O_{u.c.}(1).
    \end{displaymath}
\end{prp}
\begin{proof}
First, exploit \citet[equation (12)]{Dol(20)} to obtain the expression of the probability generating function of $R_n(z)$: 
\begin{align}
\label{eq: Grn}
G_{R_n(z)}(s) &= \frac{\sum_{k = 1}^ n  \mathscr{C}(n, k;  \alpha) \, (nsz)^k}{\sum_{k = 1}^ n  \mathscr{C}(n, k;  \alpha) \, (nz)^k}\\
   \nonumber & = e^{nz (s-1) }\, s^\frac{n}{\alpha} \frac{\int_{0}^{+\infty} x^n \, \exp\left\{-x (nsz)^\frac{1}{\alpha} \right\} \, f_\alpha (x) \, \mathrm{d}{x}}{\int_{0}^{+\infty} x^n \, \exp\left\{-x (nz)^\frac{1}{\alpha} \right\} \, f_\alpha (x) \, \mathrm{d}{x}}\\
 \nonumber  & = e^{nz (s-1) }\, s^\frac{n}{\alpha} \, \frac{I_n(zs)}{I_n(z)}\ ,
\end{align}
where $s$ can be considered as a real variable and, for any $y > 0$, 
\begin{displaymath}
I_n(y) := \int_{0}^{+\infty} x^n \, \exp\left\{-x (ny)^\frac{1}{\alpha} \right\} \, f_\alpha (x) \, \mathrm{d}{x}\ .
 \end{displaymath}
Whence,
\begin{displaymath}
    I_n'(y) = -\frac{n^\frac{1}{\alpha}}{\alpha} y^\frac{1-\alpha}{\alpha} \int_0^{+\infty} x^{n+1} \, \exp\left\{-x(ny)^\frac{1}{\alpha}\right\} \, f_\alpha(x) \, \mathrm{d}x =: -\frac{n^\frac{1}{\alpha}}{\alpha} \, y^\frac{1-\alpha}{\alpha}\,  \mathcal{J}_n^{(1)}(y)
\end{displaymath}
and
\begin{align*}
    I_n''(y) &= \frac{1-\alpha}{\alpha} \, \frac{1}{y} \,  I_n'(y) + \frac{n^{2\alpha}}{\alpha^2}\, y^\frac{2(1-\alpha)}{\alpha}  \int_0^{+\infty} x^{n+2} \, \exp\left\{-x(ny)^\frac{1}{\alpha}\right\} \, f_\alpha(x) \, \mathrm{d}x \\
    & =: \frac{1-\alpha}{\alpha} \, \frac{1}{y}\,  I_n'(y) + \frac{n^{2\alpha}}{\alpha^2}\, y^\frac{2(1-\alpha)}{\alpha} \,  \mathcal{J}_n^{(2)}(y).
\end{align*}
Differentiation and evaluation at $s=1$ of \eqref{eq: Grn} then yield
    \begin{equation}
    \label{eq:Gnprime}
        G'_{R_n(z)}(1) = nz + \frac{n}{\alpha } + z\,  \frac{I_n'(z)}{I_n(z)} = nz + \frac{n}{\alpha }   -  \frac{(nz)^\frac{1}{\alpha}}{\alpha} \,   \frac{\mathcal{J}_n^{(1)}(z)}{I_n(z)}
    \end{equation}
    and
    \begin{align}
    \label{eq:Gnsecond}
         G''_{R_n(z)}(1) &= \left(nz + \frac{n}{\alpha }\right)^2 - \frac{n}{\alpha} + 2 \left(nz + \frac{n}{\alpha }\right)\, z\,  \frac{I_n'(z)}{I_n(z)} + z^2\,  \frac{I_n''(z)}{I_n(z)} \\
         & = \left(nz + \frac{n}{\alpha }\right)^2 - \frac{n}{\alpha} + \left[2n \left(z + \frac{1}{\alpha }\right) + \frac{1-\alpha}{\alpha} \right] \times \nonumber\\
         &\quad\times \left[G'_{R_n(z)}(1) - nz - \frac{n}{\alpha}\right]  + \frac{(nz)^{2\alpha}}{\alpha^2}\,    \frac{\mathcal{J}_n^{(2)}(z)}{I_n(z)}. \nonumber
    \end{align}
Thanks to the approximation results given in Lemma \ref{lem:Zol}, we can utilize the function $E_{\alpha}$ therein to get the identities 
\begin{displaymath}
    \frac{(nz)^\frac{1}{\alpha}}{\alpha} \,   \frac{\mathcal{J}_n^{(1)}(z)}{I_n(z)} =   \frac{(nz)^\frac{1}{\alpha}}{\alpha}   \frac{ J_{1, 0}^{(n)}(z) + Q_\alpha J_{1, 1}^{(n)}(z) +  R^{(n)}_{1}(z) /\chi_\alpha}{J_{0, 0}^{(n)}(z) + Q_\alpha J_{0, 1}^{(n)}(z) +  R^{(n)}_{0}(z) /\chi_\alpha}
\end{displaymath}
and
\begin{displaymath}
    \frac{(nz)^\frac{2}{\alpha}}{\alpha^2} \,   \frac{\mathcal{J}_n^{(2)}(z)}{I_n(z)} =   \frac{(nz)^\frac{1}{\alpha}}{\alpha}   \frac{ J_{2, 0}^{(n)}(z) + Q_\alpha J_{2, 1}^{(n)}(z) +  R^{(n)}_{2}(z) /\chi_\alpha}{J_{0, 0}^{(n)}(z) + Q_\alpha J_{0, 1}^{(n)}(z) +  R^{(n)}_{0}(z) /\chi_\alpha},
\end{displaymath}
where $Q_\alpha$ is the same as in \eqref{eq:Zol_bound_2}, $\chi_\alpha$ is a positive constant depending only on $\alpha$, and the $J_{\varepsilon, \delta}^{(n)}$'s are as in \eqref{eq: def_prel_J}. Moreover, for $j = 0, 1, 2$, define
\begin{displaymath}
    R^{(n)}_{j}(z)  := \int_0^{+\infty} x^{n+j} \, \exp\left\{-x(nz)^\frac{1}{\alpha}\right\} \, E_\alpha(x) \left[ -1 - Q_\alpha x^\frac{\alpha}{1-\alpha}  + \frac{f_\alpha(x)}{E_\alpha(x)}\right] \, \mathrm{d}x.
\end{displaymath}
These quantities represent remainders terms that satisfy
\begin{equation}
\label{eq: Zol_resto}
    \left|\frac{R^{(n)}_{j}(z) }{\chi_\alpha}\right| \le \mathcal{K}_\alpha J_{j, 2}^{(n)}( z)
\end{equation}
for some positive constant $\mathcal{K}_\alpha$. Now, to make some simplifications, set
\begin{displaymath}
    \Psi(\alpha, n, z): = \left(\alpha n y \right)^{n} \, \sqrt{\frac{2\pi \alpha }{n(1-\alpha) D(z) }}
\end{displaymath}
with the same $D(z)$ as in \eqref{eq:defD}. Then, for $j = 0, 1, 2$, define
\begin{displaymath}
    \mathfrak{R}^{(n)}_{j}(z)  := \frac{R^{(n)}_{j}(z)  \, \alpha^{\frac{\alpha}{2(1-\alpha)}-j} (nz)^\frac{\alpha}{2(1-\alpha)} \sqrt{\tau(z)}}{\chi_\alpha\,  \mathcal{C}(\alpha, n, z) \, \Psi(\alpha, n, z)}
\end{displaymath}
where $\tau(z)$ denotes the unique real, positive solution to equation \eqref{eq:tauParis}, while $\mathcal{C}(\alpha, n, z)$ is the same as in \eqref{Ccal_alphany}. 
Combining Lemmata \ref{lem: prelim J} and \ref{lem:paris} with \eqref{eq:tauParis} allows to write, after some straightforward algebraic manipulations, the identities
\begin{displaymath}
\frac{(nz)^\frac{1}{\alpha}}{\alpha} \,   \frac{\mathcal{J}_n^{(1)}(z)}{I_n(z)} = \frac{N_1(z,n)}{1+\varepsilon(z,n)} \qquad \text{and}\qquad \frac{(nz)^\frac{2}{\alpha}}{\alpha^2} \,   \frac{\mathcal{J}_n^{(2)}(z)}{I_n(z)} = \frac{N_2(z,n)}{1+\varepsilon(z,n)},
\end{displaymath}
where:
\begin{align*}
\varepsilon(z,n) & := \frac{1}{n} R_{\varrho(0, 0)} + r^{(n)}_{\varrho(0, 0)} +\frac{1}{n} \,  Q_\alpha \alpha^{\frac{\alpha}{1-\alpha} } \, \frac{\tau(z)}{z}  \left[ 1 + \frac{1}{n} R_{\varrho(0, 1)} + r^{(n)}_{\varrho(0, 1)} \right]+ 
      \mathfrak{R}^{(n)}_{0} \\
N_1(z,n) & :=  n \left(\frac{1}{\alpha} + \frac{z}{\tau(z)}\right) \left[ 1 + \frac{1}{n} R_{\varrho(1, 0)} + r^{(n)}_{\varrho(1, 0)} \right] \\
     &\quad+\  Q_\alpha \alpha^\frac{\alpha}{1-\alpha} \frac{\tau(z)}{z} \left(\frac{1}{\alpha} + \frac{z}{\tau(z)}\right) \left[ 1 + \frac{1}{n} R_{\varrho(1, 1)} + r^{(n)}_{\varrho(1, 1)} \right]+ \mathfrak{R}^{(n)}_{1}\\
N_2(z,n) & := n^2 \left(\frac{1}{\alpha} + \frac{z}{\tau(z)}\right)^2 \left[ 1 + \frac{1}{n} R_{\varrho(2, 0)} + r^{(n)}_{\varrho(2, 0)} \right] \\
&\quad+ n \, Q_\alpha \alpha^\frac{\alpha}{1-\alpha} \frac{\tau(z)}{z} \left(\frac{1}{\alpha} + \frac{z}{\tau(z)}\right)^2 \left[ 1 + \frac{1}{n} R_{\varrho(2, 1)} + r^{(n)}_{\varrho(2, 1)} \right]+ \mathfrak{R}^{(n)}_{2}.
     \end{align*}    
Now, consider the straightforward identity 
\begin{displaymath}
    \frac{1}{1+\varepsilon(z,n)} = 1-\varepsilon(z,n) + \frac{\varepsilon(z,n)^2}{1+\varepsilon(z,n)}.
\end{displaymath}
The bound \eqref{eq: Zol_resto} and Lemma \ref{lem:paris} guarantee that
\begin{displaymath}
    \frac{\varepsilon^2(z,n)}{1+\varepsilon(z,n)} = O_{u.c.}\left(\frac{1}{n^2}\right).
\end{displaymath}
Thus, rearranging of the terms and highlighting only the first orders in the asymptotical expansion in $n$, we can write
\begin{align*}
    & \frac{(nz)^\frac{1}{\alpha}}{\alpha}   \,   \frac{\mathcal{J}_n^{(1)}(z)}{I_n(z)} \\
     & = N_1(z,n) - \varepsilon(z,n) N_1(z,n) + N_1(z,n) \frac{\varepsilon^2(z,n)}{1+\varepsilon(z,n)} \\ 
     & = n \left[\frac{1}{\alpha} + \frac{z}{\tau(z)} \right] + \left[\left(\frac{1}{\alpha} + \frac{z}{\tau(z)}\right)\left(R_{\varrho(1, 0)} - R_{\varrho(0, 0)}\right) \right] + O_{u.c.}\left(\frac{1}{n}\right)\\
     & = n \left[\frac{1}{\alpha} + \frac{z}{\tau(z)} \right] + \left[ \frac{1-\alpha}{\alpha }\left(\frac{1}{\alpha} + \frac{z}{\tau(z)}\right) \left(\frac{1-2\alpha}{\alpha} c_1(z) - c_2(z)\right) \right] + O_{u.c.}\left(\frac{1}{n}\right)
\end{align*}
and
\begin{align*}
     & \frac{(nz)^\frac{2}{\alpha}}{\alpha^2}   \,   \frac{\mathcal{J}_n^{(2)}(z)}{I_n(z)} \\
     & =  N_2(z,n) - \varepsilon(z,n) N_2(z,n) + N_2(z,n) \frac{\varepsilon(z,n)^2}{1+\varepsilon(z,n)} \\
    & = n^2 \left[\left(  \frac{1}{\alpha} + \frac{z}{\tau(z)} \right)^2\right]  +  n \left[ \left(  \frac{1}{\alpha} + \frac{z}{\tau(z)} \right)^2  \left(R_{\varrho(2, 0)} - R_{\varrho(0, 0)}\right) \right] + O_{u.c.}(1)\\
     & = n^2 \left[\left(  \frac{1}{\alpha} + \frac{z}{\tau(z)} \right)^2\right]  +  n \left[2 \frac{1-\alpha}{\alpha } \left(  \frac{1}{\alpha} + \frac{z}{\tau(z)} \right)^2  \left(\frac{2-3\alpha}{\alpha} c_1(z) - c_2(z)\right) \right] + O_{u.c.}(1).
\end{align*}
Finally, substitute the above expansions in the expressions \eqref{eq:Gnprime} and \eqref{eq:Gnsecond} $G'_{R_n(z)}(1)$ and $G''_{R_n(z)}(1)$, respectively, to obtain
\begin{displaymath}
    G'_{R_n(z)}(1) = n\,  \mathcal{A}(z) + \mathcal{B}(z) + O_{u.c.}\left(\frac{1}{n}\right)
\end{displaymath}
and
\begin{displaymath}
    G''_{R_n(z)}(1) = n^2 \, \mathcal{C}(z) + n\, \mathcal{D}(z) + O_{u.c.}(1)
\end{displaymath}
where:
\begin{align*}
\mathcal{A}(z) &:=  z - \frac{z}{\tau(z)} \\
\mathcal{B}(z) &:= -  \frac{1-\alpha}{\alpha }\left(\frac{1}{\alpha} + \frac{z}{\tau(z)}\right) \left(\frac{1-2\alpha}{\alpha} c_1(z) - c_2(z)\right) \\
\mathcal{C} (z) &: = \left(z + \frac{1}{\alpha}\right)^2 - 2\left(z + \frac{1}{\alpha}\right) \left(\frac{z}{\tau(z)} + \frac{1}{\alpha}\right) + \left( \frac{1}{\alpha} + \frac{z}{\tau(z)}\right)^2  = \left(z - \frac{z}{\tau(z)}\right)^2\\
\mathcal{D}(z)&: = -\frac{1}{\alpha} + 2\left(z + \frac{1}{\alpha} \right)\mathcal{B}(z) - \frac{1-\alpha}{\alpha}\left(\frac{z}{\tau(z)}+ \frac{1}{\alpha} \right)  \\
  &\quad +   \frac{2(1-\alpha)}{\alpha}  \left(\frac{1}{\alpha} + \frac{z}{\tau(z)} \right)^2  \left(\frac{2-3\alpha}{\alpha} c_1(z) - c_2(z)\right)
\end{align*}
with the same $c_1$ and $c_2$ as in Lemma \ref{lem:paris}. At this stage, notice that, with reference to the quantities $\mu(z)$ and $\sigma^2(z)$ that appear in Proposition \ref{prop:clt_R}, one has
that $\mathcal{A}(z) = \mu(z)$, $\mathcal{C}(z) = \left(\mu(z)\right)^2$ and  $\mathcal{D}(z) + \mu(z) - 2\mu(z) \mathcal{B}(z) = \sigma^2(z)$. 
Indeed, the first two identities are evident while the third can be verified via direct computation. Combining the above identities, one gets 
\begin{displaymath}
 \mathbb{E}[R_n(z)] = G'_{R_n(z)}(1) = n\,  \mathcal{A}(z) + O_{u.c.}(1)
\end{displaymath}
and
\begin{align*}
    \operatorname{Var}\left(R_n(z)\right) &= G''_{R_n(z)}(1) + G'_{R_n(z)} (1) - \left(G'_{R_n(z)} (1)\right)^2\\
    &= n^2 \, \left[  \mathcal{C}(z) - \mathcal{A}^2(z)  \right] + n\, \left[ \mathcal{D}(z) + \mathcal{A}(z) - 2\mathcal{A}(z) \mathcal{B}(z) \right] + O_{u.c.}(1)
\end{align*}
and the proof is concluded. 
\end{proof}
\end{paragraph}

\begin{paragraph}{\underline{Step 3} of the proof of the Berry-Esseen theorem \eqref{eq:be_rn}} We prove the following Berry-Esseen lemma for $R_n(z)$.
\begin{lem}
\label{lem: BE lem}
Fix $\zeta_0 $ and $\zeta_1$ such that $0< \zeta_0 < z_0 <\zeta_1 <+\infty$, with the same $z_0$ as in Proposition \ref{prop:clt_Z}. 
If $\xi \in \mathbb{R}$ satisfies 
\begin{equation}
    \label{eq: BE_xi_bound}
    |\xi|\le \mathcal{C}\,  \sigma(z) \, n^\delta
\end{equation} 
for every $z \in [\zeta_0, \zeta_1]$, for some positive constant $\mathcal{C}$, and some $\delta \in (0, 1/6)$, then there exists anoter constant $\tilde{c}$, depending on $\zeta_0$, $\zeta_1$, $\alpha$, $\lambda$ and
$\mathcal C$ such that
\begin{displaymath}
    \left|\varphi_{W_n(z)} (\xi) - e^{- \frac{\xi^2}{2}} \right| \le \tilde{c}\,  e^{ - \frac{\xi^2}{2}}\, n^{ 3\delta - \frac{1}{2}}.
\end{displaymath}
\end{lem}

\begin{proof}
Combining \eqref{eq: Grn} with Lemma \ref{lem:Zol}, we get
\begin{displaymath}
       \frac{I_n(zs)}{I_n(z)} =   \frac{ J_{0, 0}^{(n)}(zs)  +  R_n(zs) /\chi_\alpha}{J_{0, 0}^{(n)}(z) +  R_n(z) /\chi_\alpha} =  \frac{ J_{0, 0}^{(n)}(zs)}{J_{0, 0}^{(n)}(z) } \, \frac{1 + R_n(zs) /\left(\chi_\alpha \,  J_{0, 0}^{(n)}(zs)\right)}{1 + R_n(z) /\left(\chi_\alpha \,  J_{0, 0}^{(n)}(z)\right)},
\end{displaymath}
for $z \in (0,+\infty)$ and $s \in \mathbb{C}^+$, where: $\chi_\alpha$ is a positive constant depending only on $\alpha$, the $J_{\varepsilon, \delta}^{(n)}$'s are as in \eqref{eq: def_prel_J}, and, for any $y \in \mathbb{C}^+$
\begin{displaymath}
    R_n(y)  := \int_0^{+\infty} x^{n} \, \exp\left\{-x(ny)^\frac{1}{\alpha}\right\} \, E_\alpha(x) \left[ -1  + \frac{f_\alpha(x)}{E_\alpha(x)}\right] \, \mathrm{d}x 
\end{displaymath}
satisfies
\begin{equation}
\label{bound_R_appendix}
    \left|R_n(y)  \right| \le C_\alpha \, J_{0, 1}^{(n)}\left( Re(y) \right)
\end{equation}
with $C_\alpha $ as in \eqref{eq:Zol_bound_1}. Then, combining Lemma \ref{lem: prelim J} and Lemma \ref{lem:paris}, we have
\begin{align*}
     \frac{I_n(zs)}{I_n(z)} 
     & =s^{-\frac{n}{\alpha}+ \frac{2-\alpha}{2\alpha(1-\alpha)}- \frac{1}{\alpha}} \frac{\mathcal{F}^{(n)}_\frac{1}{2} (zs)}{\mathcal{F}^{(n)}_\frac{1}{2} (z)} \, \,\mathfrak{R}_1^{(n)}(z, s) \\
     & = s^{-\frac{n}{\alpha}+ \frac{2-\alpha}{2\alpha(1-\alpha)}- \frac{1}{\alpha}} \frac{\mathcal{P}^{(n)}_\frac{1}{2}(zs)}{\mathcal{P}^{(n)}_\frac{1}{2}(z)} \,\mathfrak{R}_1^{(n)}(z, s) \, \mathfrak{R}_2^{(n)}(z, s),
\end{align*}
where
\begin{align*}
\mathfrak{R}_1^{(n)}(z, s) &:=  \frac{1 + R_n(zs) /\left(\chi_\alpha \,  J_{0, 0}^{(n)}(zs)\right)}{1 + R_n(z) /\left(\chi_\alpha \,  J_{0, 0}^{(n)}(z)\right)} \\
\\
 \mathfrak{R}_2^{(n)}(z, s) &:=  \frac{ \mathcal{P}^{(n)}_\frac{1}{2}(z)}{\mathcal{F}^{(n)}_\frac{1}{2}(z) } \frac{ \mathcal{F}^{(n)}_\frac{1}{2}(zs)}{\mathcal{P}^{(n)}_\frac{1}{2}(zs) }.
\end{align*}
Now, combine the above identities with the expression of $\mathcal{P}^{(n)}_\frac{1}{2}$ in Lemma \ref{lem:paris}, with $g$ as in \eqref{gParis}, to obtain
\begin{align*}
     G_{R_n(z)}(s) & =  e^{nz(s-1)}\, s^{\frac{1}{2(1-\alpha)}} \, \frac{\mathcal{P}^{(n)}_\frac{1}{2}(zs)}{\mathcal{P}^{(n)}_\frac{1}{2}(z)} \,  \,\mathfrak{R}_1^{(n)}(z, s) \, \mathfrak{R}_2^{(n)}(z, s)\\
& =  \exp\left\{ -n \, \left[\left(g(zs) - zs - \log(s) \right) - \left( g(z) - z \right)\right] \right\} \, \mathfrak{R}_1^{(n)}(z, s) \, \mathfrak{R}_2^{(n)}(z, s) \, \mathfrak{R}_3^{(n)}(z, s),
\end{align*}
with
\begin{align*}
  \mathfrak{R}_3^{(n)}(z, s) & :=   s^\frac{1}{2} \,  \sqrt{\frac{\alpha z + \tau(z) (1-\alpha) }{\alpha z s + \tau(z s) (1-\alpha)}}.
\end{align*}
Now, let $\varphi_{W_n(z)}$ denote the characteristic function of the variable $W_n(z)$. That is,
\begin{align*}
    & \varphi_{W_n(z)}(\xi) \\
    & = \exp\left\{-\sqrt{n} \, \frac{i\, \xi\,  \mu(z)}{\sigma(z)} \right\} \, G_{R_n(z)} \left(e^\frac{i \, \xi }{\sqrt{n} \sigma(z)}\right) \\
    & = \exp\left\{-\sqrt{n} \, \frac{i\, \xi\,  \mu(z)}{\sigma(z)} \right\} \, \exp\left\{ -n \, \left[g\left(ze^\frac{i \, \xi }{\sqrt{n} \sigma(z)}\right) - g(z) - z \left(e^\frac{i \, \xi }{\sqrt{n} \sigma(z)} -1\right) - 
    \frac{i \, \xi }{\sqrt{n} \sigma(z)} \right] \right\} \times \\
    &\hspace{13pt}\times \mathfrak{R}_1^{(n)}\left(z, e^\frac{i \, \xi }{\sqrt{n} \sigma(z)}\right) \, \mathfrak{R}_2^{(n)}\left(z, e^\frac{i \, \xi }{\sqrt{n} \sigma(z)}\right) \, \mathfrak{R}_3^{(n)}\left(z, e^\frac{i \, \xi }{\sqrt{n} \sigma(z)}\right).
\end{align*}
If $\xi$ satisfies \eqref{eq: BE_xi_bound}, then \citet[Chapter IV, Lemma 5]{Pet(75)} guarantees that
\begin{displaymath}
    \left|e^\frac{i \, \xi }{\sqrt{n} \sigma(z)}-1\right| \le \left| \frac{\xi }{\sqrt{n} \sigma(z)}\right| \le \mathcal{C} \, n^{\delta - \frac{1}{2}}.
\end{displaymath} 
We apply Taylor's formula to $g \left(e^\frac{i \, \xi }{\sqrt{n} \sigma(z)} \right) - g(z)$ and then to $\left(e^\frac{i \, \xi }{\sqrt{n} \sigma(z)} - 1\right)$ to obtain
\begin{align*}
    &\varphi_{W_n(z)}(\xi)\\
    &\quad= \exp\left\{\sqrt{n} \, \frac{i \, \xi }{\sigma(z)} \left[ -\mu(z) + z +1 -z\, g'(z) \right] - \frac{ \, \xi^2 }{2 \, \sigma^2(z)} \left[ z + zg'(z) - z^2 g''(z)\right]\right\} \times\\
    &\quad\quad\times \mathfrak{R}_1^{(n)}\left(z, e^\frac{i \, \xi }{\sqrt{n} \sigma(z)}\right) \, \mathfrak{R}_2^{(n)}\left(z, e^\frac{i \, \xi }{\sqrt{n} \sigma(z)}\right) \, \mathfrak{R}_3^{(n)}\left(z, e^\frac{i \, 
    \xi }{\sqrt{n} \sigma(z)}\right)\, \mathfrak{R}_4^{(n)}\left(z, \xi\right),
\end{align*}
where 
\begin{align*}
   &\mathfrak{R}_4^{(n)}\left(z, \xi\right)\\
    &\quad := \exp\left\{ nz \left[ 1-g'(z) \right] \left[\left(e^\frac{i \, \xi }{\sqrt{n} \sigma(z)} - 1\right) -\left(\frac{i \, \xi }{\sqrt{n} \sigma(z)} - \frac{\xi^2 }{2 n \sigma^2(z)}\right) \right] \right.\\
   &\quad\quad\quad\quad\quad + \frac{1}{2}\, nz^2  g''(z) \left[ \left(e^\frac{i \, \xi }{\sqrt{n} \sigma(z)} - 1\right)^2 + \frac{\xi^2}{n \sigma^2}\right]\\
   &\quad\quad\quad\quad\quad \left. + \frac{1}{2 } n z^3 \, \left(e^\frac{i \, \xi }{\sqrt{n} \sigma(z)} - 1\right)^3 \int_0^1 g'''\left(z \left[1 + t \left(e^\frac{i \, \xi }{\sqrt{n} \sigma(z)} - 1\right)\right]\right) \, (1-t)^2 \, \mathrm{d} t\right\}.
\end{align*}
Recalling the definitions of $\mu(z), \, \sigma^2(z)$ and $g(z)$, we conclude that
\begin{align*}
   &\varphi_{W_n(z)}(\xi)\\
    &\quad= \exp\left\{- \frac{ \, \xi^2 }{2 } \right\}\mathfrak{R}_1^{(n)}\left(z, e^\frac{i \, \xi }{\sqrt{n} \sigma(z)}\right) \, \mathfrak{R}_2^{(n)}\left(z, e^\frac{i \, \xi }{\sqrt{n} \sigma(z)}\right) \, \mathfrak{R}_3^{(n)}\left(z, e^\frac{i \, \xi }{\sqrt{n} \sigma(z)}\right)\, \mathfrak{R}_4^{(n)}\left(z, \xi\right).
\end{align*}
By combining Lemma \ref{lem:Zol}, Lemma \ref{lem:paris} and \citet[Chapter IV, Lemma 5]{Pet(75)} we prove that the function 
$$
 \mathfrak{R}^{(n)}\left(z, \xi\right)  :=  \mathfrak{R}_1^{(n)}\left(z, e^\frac{i \, \xi }{\sqrt{n} \sigma(z)}\right) \, \mathfrak{R}_2^{(n)}\left(z, e^\frac{i \, \xi }{\sqrt{n} \sigma(z)}\right) \, \mathfrak{R}_3^{(n)}\left(z, e^\frac{i \, \xi }{\sqrt{n} \sigma(z)}\right)\, \mathfrak{R}_4^{(n)}\left(z, \xi \right),
$$
defined for $z \in [\zeta_0, \zeta_1]$ and $\xi$ satisfying \eqref{eq: BE_xi_bound}, is continuous and satisfies
\begin{equation}
\label{eq: remainders}
    \mathfrak{R}^{(n)}\left(z, \xi\right) = 1 + O_{u.c.}\left(n^{ - \frac{1}{2} + 3 \delta}\right).
\end{equation}
We refer to Appendix \ref{app: sec_BE} for a detailed proof of \eqref{eq: remainders}. Lemma \ref{lem: sup is continuous} then guarantees that for every $z \in [\zeta_0, \zeta_1]$ and $\xi$ satisfying \eqref{eq: BE_xi_bound}
\begin{displaymath}
    n^{ \frac{1}{2} - 3 \delta} \left|\mathfrak{R}^{(n)}\left(z, \xi\right)- 1\right| \le  S(z, \xi),
\end{displaymath}
where $S(z, \xi) : = \sup_{n \in \mathbb{N}} \left|\mathfrak{R}^{(n)}\left(z, \xi\right)- 1\right|$ is continuous, hence bounded on compact sets. This allows to conclude that there exists a constant $\tilde{c}$, 
depending on $\zeta_0$, $\zeta_1$, $\alpha$, $\lambda$ and $\mathcal C$ for which $\left|\mathfrak{R}^{(n)}\left(z, \xi\right)- 1\right| \le  \tilde{c}\,   n^{ - \frac{1}{2} + 3 \delta}$. This concludes the proof. 
\end{proof}
\end{paragraph}

\begin{paragraph}{\underline{Step 4} of the proof of the Berry-Esseen theorem \eqref{eq:be_rn}}
\begin{proof}[Proof of the Berry-Esseen Theorem \eqref{eq:be_rn}]
Fix $\zeta_0 $ and $\zeta_1$ such that $0< \zeta_0 < z_0 <\zeta_1 <+\infty$, with the same $z_0$ as in Proposition \ref{prop:clt_Z}. 
Then, combine Lemma \ref{lem: BE lem} with  the well-known inequality \citep[Chapter V, Theorem 2]{Pet(75)} to obtain that 
\begin{align*}
    \left\|F_{W_n(z)} - \Phi \right\|_\infty &\le \int_{| \xi| \le \mathcal{C}\,\sigma(z) \, n^\delta} \left|\frac{\varphi_{W_n(z)} (\xi) - e^{- \frac{\xi^2}{2}}}{\xi} \right| \, \mathrm{d} \xi + \tilde{\mathcal{C}}{n^{-\delta}} \\
    &\le  \int_{-\frac{1}{n}}^{\frac{1}{n}} \left|\frac{\varphi_{W_n(z)} (\xi) - e^{- \frac{\xi^2}{2}}}{\xi} \right| \, \mathrm{d} \xi + \frac{2\tilde{c}  }{n^{ - 3\delta +\frac{1}{2}}} \int_{\frac{1}{n}}^{+\infty} \frac{e^{- \frac{\xi^2}{2}}}{\xi} \, \mathrm{d} \xi + \tilde{\mathcal{C}}{n^{-\delta}}\\
    & =: I_1 + I_2 + \tilde{\mathcal{C}} n^{-\delta}
\end{align*}
hold for any $z \in [\zeta_0, \zeta_1]$, with $\tilde{\mathcal{C}} = \max_{z \in [\zeta_0, \zeta_1]} [\mathcal{C}\sigma(z)]^{-1}$.

To bound $I_1$, combine the triangle inequality, \citet[Chapter IV, Lemma 5]{Pet(75)}, \citet[Section 8.4, Theorem 1, Equation (4)]{CT(97)}, and the elementary inequality $e^{-x} \geq 1-x$, to write 
\begin{align*}
    \left|\frac{\varphi_{W_n(z)}(\xi) - e^{-\frac{\xi^2}{2}}}{\xi} \right| &\le \left|\frac{\varphi_{W_n(z)}(\xi) - 1}{\xi} \right| + \left|\frac{1 - e^{-\frac{\xi^2}{2}}}{\xi} \right|\\
    & \le \left|\mathbb{E} \left[W_n(z)\right] \right|  + \frac{1}{2} \mathbb{E} \left[W_n(z)^2\right]  \left|\xi \right| +\frac{1}{2} \left|\xi \right| .
\end{align*}
The combination of Proposition \ref{prop: mu_sigma_Rn} with Lemma \ref{lem: sup is continuous} entails the existence of some positive constant $M_1$ for which, for any $z \in [\zeta_0, \zeta_1]$, 
\begin{displaymath}
  \left|\mathbb{E} \left[W_n(z)\right] \right|  \leq \frac{\operatorname{Var}(R_n(z))}{n \sigma^2(z)} \le M_1\ .
\end{displaymath}
Analogously, there exists a positive constant $M_1$ for which, for any $z \in [\zeta_0, \zeta_1]$, 
\begin{displaymath}
\mathbb{E} \left[W_n^2(z)\right] \le M_2\ . 
\end{displaymath}
Whence, 
\begin{displaymath}
    I_1 \le \frac{2 M_1}{n} + \frac{M_2 + 1}{2n^2}.
\end{displaymath}
To bound $I_2$, notice that
\begin{displaymath}
    I_2 =  \frac{\tilde{c}}{n^{ - 3\delta + \frac{1}{2}}}\, \Gamma\left(0, \frac{1}{2n^2}\right),
\end{displaymath}
where $\Gamma(a, x):= \int_x^{+\infty} t^{a-1} e^{-t} \mathrm{d}t$ denotes the incomplete Gamma function. Now, recall that, as $n \to +\infty$
\begin{displaymath}
   \Gamma\left(0, \frac{1}{2n^2}\right) \sim 2\log(n). 
\end{displaymath}
See, e.g., \citet[ Equation \href{https://dlmf.nist.gov/8.4.E4}{(8.4.4)} and Equation \href{https://dlmf.nist.gov/6.6.E4}{(6.6.2)}]{nist}. Whence, there exists some positive constant $\tilde{c}_1$ for which
\begin{displaymath}
    I_2 \le   \tilde{c}_1 \log(n)  \, n^{ - \frac{1}{2} + 3\delta}.
 \end{displaymath}
To conclude, for some positive constants $C_1, C_2$, we can write 
\begin{displaymath}
    \left\|F_{W_n(z)} - \Phi \right\|_\infty \le    C_1 n^{-1} + C_2 \log(n)  \, n^{ - \frac{1}{2} + 3\delta}  +  \tilde{\mathcal{C}}{n^{-\delta}}.
 \end{displaymath}
It is easy to see that for every $\delta \in (0, 1/6)$, $\min\left(\delta, -3\delta +\frac{1}{2}\right) \ge \frac{1}{8}$, which produces \eqref{eq:be_rn} by choosing $\delta = 1/8$. 
\end{proof}
\end{paragraph}


\subsubsection{Proof of Equation \eqref{eq: remainders}}\label{app: sec_BE}
Fix $\zeta_0 $ and $\zeta_1$ such that $0< \zeta_0 < z_0 <\zeta_1 <+\infty$, with the same $z_0$ as in Proposition \ref{prop:clt_Z}, and $z \in [\zeta_0, \zeta_1]$. Set $\tau(zs) := \xi(1/(\alpha zs))$, with the same function $\xi$ as in Lemma \ref{lem:xi_Paris}. Then, there exists some positive  $\tilde{\delta}$, depending only on $\zeta_0, \zeta_1$, for which the map $s \mapsto \tau(zs)$ is holomorphic in the disc $|s-1| \leq \tilde{\delta}$.
Preliminarily, note that if $s = e^{\frac{i\xi}{\sqrt{n} \sigma(z)}}$ with $\xi$ satisfying \eqref{eq: BE_xi_bound},  \citet[Chapter IV, Lemma 5]{Pet(75)} entails 
\begin{equation} \label{smenouno}
\left|s-1 \right| = \left|e^{\frac{i\xi}{  \sqrt{n} \sigma(z)}}-1 \right|  \le \left| \frac{\xi }{ \sqrt{n} \sigma(z)}\right| \le \mathcal{C} \, n^{\delta - \frac{1}{2}},
\end{equation}
so that there exists $\bar{n} \in \N$ such that for every $n \ge \bar{n}$, 
\begin{displaymath}
\left| e^{\frac{i\xi}{  \sqrt{n} \sigma(z)}} - 1\right| \le \mathcal{C} n^{\delta-\frac{1}{2}} \le \tilde{\delta}.
\end{displaymath}

Concerning the behaviour of $\mathfrak{R}^{(n)}_1$, write 
\begin{displaymath}
\mathfrak{R}^{(n)}_1(z, s) = \frac{1+ \mathfrak{H}(s) }{1+ \mathfrak{H}(1) } = 1+ \left[\mathfrak{H}(s) -  \mathfrak{H}(1)\right] \left[ 1- \mathfrak{H}(1)\right] + \frac{ \left[\mathfrak{H}(s) -  \mathfrak{H}(1)\right]  \, \mathfrak{H}^2(1)}{1+ \mathfrak{H}(1)}, 
\end{displaymath}
 with
\begin{displaymath}
\mathfrak{H}(s) := \frac{R_n(zs)}{\chi_\alpha \, J_{0, 0}^{(n)}(zs)}.
\end{displaymath}
The combination of \eqref{bound_R_appendix}, Lemma \ref{lem: prelim J} and Lemma \ref{lem:paris} yields, after straightforward computations,
\begin{align*}
 | \mathfrak{H}(1)| &\le \frac{C_\alpha \, J_{0, 1}^{(n)} (z) }{\chi_\alpha \, J_{0, 0}^{(n)} (z) } \\
& = \frac{1}{n} \, \frac{C_\alpha \, \alpha^{\frac{\alpha}{1-\alpha}}}{\chi_\alpha} \, \frac{\tau(z)}{z}   \frac{1+ \frac{1}{n} \mathcal{R}_{-\frac{1}{2}}(z) + 
r_{-\frac{1}{2}}^{(n)} (z)}{1+ \frac{1}{n} \mathcal{R}_{\frac{1}{2}}(z) + r_{\frac{1}{2}}^{(n)} (z)} \\
& =  \frac{1}{n} \phi(z) \left[1 +r_n \right],
\end{align*}
where $\phi(z) := C_\alpha \, \alpha^{\frac{\alpha}{1-\alpha}}\chi_{\alpha}^{-1} \, [\tau(z)/z]$ is continuous and strictly positive on $[\zeta_0, \zeta_1]$. Moreover, from Lemma \ref{lem:paris}, 
\begin{displaymath}
n r_n := n \left[\frac{1+ \frac{1}{n} \mathcal{R}_{-\frac{1}{2}}(z) + r_{-\frac{1}{2}}^{(n)} (z)}{1+ \frac{1}{n} \mathcal{R}_{\frac{1}{2}}(z) + r_{\frac{1}{2}}^{(n)} (z)} - 1\right]
\end{displaymath}
converges uniformly to some continuous function on $[\zeta_0, \zeta_1]$. Then, we can conclude that
\begin{displaymath}
\mathfrak{H}(1) = O_{u.c.}\left(\frac{1}{n}\right)
\end{displaymath}
and, by the same argument, 
\begin{displaymath}
\frac{ \mathfrak{H}^2(1)}{1+ \mathfrak{H}(1)} = O_{u.c.}\left(\frac{1}{n^2}\right).
\end{displaymath}
Now, by the fundamental theorem of calculus, 
\begin{displaymath} 
\mathfrak{H}(s) - \mathfrak{H}(1) = (s-1)\, \int_0^1 \mathfrak{H}'(1 + t(s-1))\,   \mathrm{d} t.
\end{displaymath}
It can be easily seen that $s \mapsto \mathfrak{H}(s)$ is holomorphic in $\{w \in \mathbb{C} \, : \, Re(w) >0\}$. The combination of this fact with the bound \eqref{smenouno} entails that
there exists a continuous function $G:[\zeta_0, \zeta_1] \to (0, +\infty)$, not depending on $s$ and $n$, for which
\begin{displaymath} 
|\mathfrak{H}(s) - \mathfrak{H}(1) | \le G(z) \, \mathcal{C} \, n^{\delta - \frac{1}{2}}.
\end{displaymath}
Whence,
\begin{displaymath} 
\mathfrak{R}^{(n)}_1(z, s) = 1 + O_{u.c.}\left(n^{\delta - \frac{1}{2}}\right).
\end{displaymath}

Concerning the behaviour of $\mathfrak{R}^{(n)}_2$, note that $\mathfrak{R}_2^{(n)}(z, s)$ is the product of two terms. As to the former, apply Lemma \ref{lem:paris} to obtain 
\begin{displaymath} 
\frac{ \mathcal{P}^{(n)}_\frac{1}{2}(z)}{\mathcal{F}^{(n)}_\frac{1}{2}(z) } = \frac{1}{1+ \frac{1}{n} \mathcal{R}_{\frac{1}{2}}(z) + r^{(n)}_{\frac{1}{2}}(z)} =: 
\frac{1}{1+\varepsilon(z)} =  1-\varepsilon(z) + \frac{\varepsilon(z)^2}{1+\varepsilon(z)}.
\end{displaymath}
Whence,
\begin{displaymath} 
\left| \frac{ \mathcal{P}^{(n)}_\frac{1}{2}(z)}{\mathcal{F}^{(n)}_\frac{1}{2}(z) }  - 1\right| 
\le \left|  \frac{1}{n} \mathcal{R}_{\frac{1}{2}}(z) + r^{(n)}(z) \right| + \left| \frac{\varepsilon(z)^2}{1+\varepsilon(z)}\right|.  
\end{displaymath}
For the latter, Lemma \ref{lem:paris} again shows that
\begin{displaymath} 
\left| \frac{ \mathcal{F}^{(n)}_\frac{1}{2}(zs)}{\mathcal{P}^{(n)}_\frac{1}{2}(zs) }  - 1  \right| = \left| \frac{1}{n} \mathcal{R}_{\frac{1}{2}}(zs)+ r^{(n)}(zs)\right|.
\end{displaymath}
If $|s - 1| < \tilde{\delta}$, we conclude that
\begin{displaymath} 
\frac{ \mathcal{P}^{(n)}_\frac{1}{2}(z)}{\mathcal{F}^{(n)}_\frac{1}{2}(z) } = 1  + O_{u.c.}\left(\frac{1}{n}\right) \hspace{10pt} \text{and} 
\hspace{10pt} \frac{ \mathcal{F}^{(n)}_\frac{1}{2}(zs)}{\mathcal{P}^{(n)}_\frac{1}{2}(zs) } = 1+ O_{u.c.}\left(\frac{1}{n}\right).
\end{displaymath}
Whence,
\begin{displaymath} 
\mathfrak{R}_2^{(n)} \left(z, s\right) =  \frac{ \mathcal{P}^{(n)}_\frac{1}{2}(z)}{\mathcal{F}^{(n)}_\frac{1}{2}(z) } \frac{ \mathcal{F}^{(n)}_\frac{1}{2}(zs)}{\mathcal{P}^{(n)}_\frac{1}{2}(zs) } = 1+ O_{u.c.}\left(\frac{1}{n}\right).
\end{displaymath}

Concerning the behaviour of $\mathfrak{R}^{(n)}_3$, rewrite $\mathfrak{R}_3^{(n)}(z, s)$ as 
\begin{displaymath} 
\mathfrak{R}_3^{(n)}(z, s) = s^{\frac{1}{2}} \, \left( 1+ \frac{\alpha z (s-1) + (1-\alpha) \left[\tau(zs) - \tau(z)\right]}{\alpha z + (1-\alpha) \tau(z)}\right)^{-\frac{1}{2}} =: s^{\frac{1}{2}} \, \left( 1+ H_z(s)\right)^{-\frac{1}{2}}.
\end{displaymath}
Now, if $z \in [\zeta_0, \zeta_1]$ and $|s - 1| < \tilde{\delta}$, there exists some positive constant $M_3$ for which
\begin{displaymath} 
 \left| \left( 1+ H_z(s)\right)^{-\frac{1}{2}} -1 \right| \leq M_3\left| H_z(s)\right|\ .
\end{displaymath}
Moreover, Lemma \ref{lem:xi_Paris} entails
\begin{displaymath} 
| \tau(zs) - \tau(z) | \le z \int_1^s \left| \frac{\mathrm{d}}{\mathrm{d}s} \tau(z s) \right| \le C_1(z) \, | s-1| 
\end{displaymath}
for some continuous function $C_1$ of $z$. Whence, 
\begin{displaymath} 
\left| H_z(s)\right| \le \left[ \frac{\alpha +C_1(z)}{\alpha z + (1-\alpha) \tau(z)} \right] \, | s-1|  = : C_2(z) \, |s-1| 
\end{displaymath}
where the function $C_2$ is continuous. Thus, if $s$ fulfills \eqref{smenouno}, we get
\begin{align*}
\left|H_z\left(s\right)\right| \le  \,C_2(z) \, \mathcal{C} \, n^{\delta-\frac{1}{2}}
\end{align*}
and, consequently,
\begin{displaymath}
 \mathfrak{R}_3^{(n)}\left(z, s\right) = s^\frac{1}{2}   \, \left(  1+ H_z\left(s\right)\right)^{-\frac{1}{2}} = 1+O_{u.c.}\left(n^{\delta-\frac{1}{2}}\right).
\end{displaymath}

Concerning the behaviour of $\mathfrak{R}^{(n)}_4$, elementary properties of the exponential combined with \citet[Chapter IV, Lemma 5]{Pet(75)} yield
\begin{align*}
  \mathfrak{R}_4^{(n)}\left(z, \xi\right) & := \exp\left\{ nz \left[ 1-g'(z) \right] \frac{C\,  \xi^3 }{6 \, n^\frac{3}{2} \sigma(z)^3} + \frac{1}{2}\, nz^2  g''(z) \, \frac{ C \, \xi^3 }{n^\frac{3}{2} \sigma(z)^3} \right.\\
   & \left. + \frac{1}{2 } n z^3 \, \frac{ C^3 \, \xi^3 }{n^\frac{3}{2} \sigma(z)^3}  \int_0^1 g'''\left(z \left[1 + t \left(e^\frac{i \, \xi }{\sqrt{n} \sigma(z)} - 1\right)\right]\right) \, (1-t)^2 \, \mathrm{d} t\right\}
\end{align*}
for a suitable $C \in \mathbb{C}$ with $|C|<1$. Note that, if $|s-1| < \tilde{\delta}$, it also holds that $|1+t(s-1)| < \tilde{\delta}$ for all $t \in [0, 1]$. 
Moreover, the function $t \mapsto g'''(z[1 + t(s-1)])$ is bounded on $[0, 1]$, if $z \in [\zeta_0, \zeta_1]$. 
Therefore, we get
\begin{displaymath}
\left|\frac 16 z \left[ 1-g'(z) \right] + \frac 12 z^2  g''(z) + \frac 12 z^3\int_0^1 g'''(z[1 + t(s-1)])(1-t)^2 \, \mathrm{d} t \right| \le     G(z)
\end{displaymath}
for some continuous function $G:[\zeta_0, \zeta_1] \to (0,+\infty)$ not depending on $\xi$ and $n$.  To conclude, apply again \citet[Chapter IV, Lemma 5]{Pet(75)} to show that, for any $\xi$ satisfying \eqref{eq: BE_xi_bound}, 
it holds
\begin{displaymath}
\left|  \mathfrak{R}_4^{(n)}\left(z, \xi\right)-1 \right| \le G(z) \frac{ |\xi|^3 }{n^\frac{1}{2} \sigma(z)^3}  \le G(z) \, \mathcal{C}^3 \, n^{3\delta - \frac{1}{2}}.
\end{displaymath}
Whence,
\begin{displaymath} 
\mathfrak{R}_4^{(n)}\left(z, \xi \right)  = 1+O_{u.c.}\left(n^{3\delta-\frac{1}{2}}\right).
\end{displaymath}

Altogether, we have proved that for $z \in [\zeta_0, \zeta_1]$,  $s = e^{\frac{i\xi}{\sqrt{n} \sigma(z)}}$ with $\xi$ satisfying \eqref{eq: BE_xi_bound} 
\begin{align*}
\mathfrak{R}_1^{(n)}\left(z, s\right) &= 1+ O_{u.c.}\left(n^{\delta-\frac{1}{2}}\right) \\
\mathfrak{R}_2^{(n)}\left(z, s\right) &= 1+ O_{u.c.}\left(n^{-1}\right) \\
\mathfrak{R}_3^{(n)}\left(z, s\right) &= 1+O_{u.c.}\left(n^{\delta-\frac{1}{2}}\right) \\
\mathfrak{R}_4^{(n)}\left(z, \xi \right) &= 1+O_{u.c.}\left(n^{3\delta-\frac{1}{2}}\right).
\end{align*}
Recalling that $\mathfrak{R}^{(n)}(z, \xi)$ is defined as the product of these four terms, and since $0> 3\delta-\frac{1}{2} > \delta - \frac{1}{2} > -1$ for $\delta \in \left(0, \frac{1}{6}\right)$, we conclude that $\mathfrak{R}^{(n)}(z, \xi) = 1+O_{u.c.}\left(n^{3\delta-\frac{1}{2}}\right)$.


\renewcommand{\theequation}{B.\arabic{equation}}

\setcounter{equation}{0}

\subsection{Technical results on Theorem \ref{thm_main} for $\alpha=0$}\label{app3}


\subsubsection{Asymptotic expansions \eqref{mom_m} and \eqref{mom_v}}\label{app31}

For $n\in\mathbb{N}$ let $(B_{1},\ldots,B_{n})$ be independent random variables such that $B_{i}$ is distributed according to a Bernoulli distribution with parameter $\theta/(\theta+i+1)$, for $i=1,\ldots,n$. Then, it holds
\begin{equation}\label{sum_rep0}
K_{n}=\sum_{i=1}^{n}B_{i}.
\end{equation}
\citep[Chapter 3]{Pit(06)}. Let $G_{K_n}$ be the probability generating function of $K_{n}$. Based on representation \eqref{sum_rep0}, it is straightforward to show that for $s>0$
\begin{displaymath}
G_{K_{n}}(s)=\frac{[s\theta]_{(n)}}{[\theta]_{(n)}}=\frac{\Gamma(s\theta+n)\Gamma(\theta)}{\Gamma(s\theta)\Gamma(\theta+n)},
\end{displaymath}
so that
\begin{equation}\label{app_probgen1}
G^{'}_{K_{n}}(1):=\left.\frac{\ddr}{\ddr s}G_{K_{n}}(s)\right|_{s=1}=\theta \left[ \psi(\theta+n) - \psi(\theta) \right] 
\end{equation}
and
\begin{equation}\label{app_probgen2}
G^{''}_{K_{n}}(1):=\left.\frac{\ddr^{2}}{\ddr s^{2}}G_{K_{n}}(s)\right|_{s=1}= \theta^2 \left\{ \psi^{(1)}(\theta+n) - \psi^{(1)} (\theta) + \left[ \psi(\theta+n) - \psi(\theta) \right]^2 \right\}
\end{equation}
where $\psi$ denotes the Digamma function, defined for $z \in \mathbb{C}$ as $\psi(z) := \frac{\mathrm{d}}{\mathrm{d}z} \log\Gamma(z)$, and $\psi^{(1)}$ denotes its derivative, i.e. the Trigamma function. If $\theta=\lambda n$ in \eqref{app_probgen1} then
\begin{equation}\label{app_probgen11}
G^{'}_{K_{n}}(1)=\lambda n \left[ \psi((\lambda+1)n) - \psi(\lambda n) \right].
\end{equation}
If $\theta=\lambda n$ in \eqref{app_probgen2} then
\begin{equation}\label{app_probgen21}
G^{''}_{K_{n}}(1)=(\lambda n)^2 \left\{ \psi^{(1)}((\lambda+1)n) - \psi^{(1)} (\lambda n) + \left[ \psi((\lambda+1)n) - \psi(\lambda n) \right]^2 \right\}.
\end{equation}
The formulae in Section \ref{sec31} follow from an application of \citet[ Equation \href{https://dlmf.nist.gov/5.11.E2}{(5.11.2)}]{nist} to \eqref{app_probgen11} and \eqref{app_probgen21}, respectively.


\subsubsection{Proof of equation \eqref{mom4Kn}}\label{app:mom4Kn_dir}

Arguing as in Appendix \ref{app:mom4Kn}, one gets
 \begin{align*}
     &  \mathbb{E} \left[\left(K_n - \mathbb{E} \left[ K_n\right] \right)^4 \right]     \\
      & \quad =  - 3  \mathbb{E}[K_n]^4  + 6  \mathbb{E}[K_n]^3 + 6 \mathbb{E}[K_n]^2  \mathbb{E} \left[(K_n)_{\downarrow 2} \right]  -12  \mathbb{E}[K_n]  \mathbb{E} \left[(K_n)_{\downarrow 2}\right]\\
     & \quad  \quad   - 4 \mathbb{E}[K_n]  \mathbb{E} \left[(K_n)_{\downarrow 3} \right] + 6   \mathbb{E} \left[(K_n)_{\downarrow 3} \right] +  \mathbb{E} \left[(K_n)_{\downarrow 4} \right].
      \end{align*}
Now,
 \begin{align*}
   \mathbb{E} \left[(K_n)_{\downarrow 1} \right] & = n \cdot \lambda\,  \Phi_0(n, \lambda)\\
   \mathbb{E} \left[(K_n)_{\downarrow 2} \right] & = n^2 \cdot \lambda^2 \, \left[  \Phi_0^2(n, \lambda) +  \Phi_1(n, \lambda)\right] \\
   \mathbb{E} \left[(K_n)_{\downarrow 3} \right] & = n^3 \cdot \lambda^3 \, \left[\Phi_0^3(n, \lambda) + 3 \Phi_0(n, \lambda) \Phi_1(n, \lambda) + \Phi_2(n, \lambda) \right] \\       
   \mathbb{E} \left[(K_n)_{\downarrow 4} \right] & = n^4 \cdot \lambda^4 \, \left[\Phi_0^4(n, \lambda) + 6 \Phi_0^2(n, \lambda) \Phi_1(n, \lambda) + 4 \Phi_0(n, \lambda) \Phi_2(n, \lambda) \right.\\
      & \quad \quad \quad \quad \left.+ 3 \Phi_1^2(n, \lambda) + \Phi_3(n, \lambda) \right]
 \end{align*}
where
\begin{displaymath}
\Phi_i(n, \lambda) := \psi^{(i)} \left(n (\lambda+1)\right)  -\psi^{(i)}\left(n \lambda\right)
\end{displaymath}
and $\psi^{(i)}$ denotes the polygamma function \citep[Section \href{http://dlmf.nist.gov/5.15} {(5.15)}]{nist}. Making use of the asymptotic expansions of the polygamma functions for large argument 
\citep[Equation \href{http://dlmf.nist.gov/5.11.E2} {(5.11.2)}]{nist}, one can write
 \begin{align*}
 \Phi_0(n, \lambda) & = L + \frac{1}{2 \lambda (\lambda+1) n} + O \left(n^{-2}\right)\\
 \Phi_1(n, \lambda) & = - \frac{1}{\lambda(\lambda+1) n} + O \left(n^{-2}\right)\\
  \Phi_2(n, \lambda) & =  O \left(n^{-2}\right)\\
 \Phi_3(n, \lambda) & = O \left(n^{-3}\right)
  \end{align*}
 where $L  = \log \left(\frac{\lambda+1}{\lambda} \right)$. Whence,
  \begin{align*}
   \mathbb{E} \left[(K_n)_{\downarrow 1} \right] & = n \cdot \lambda L + \frac{1}{2(\lambda+1)} + O\left(n^{-1}\right)\\
     \mathbb{E} \left[(K_n)_{\downarrow 2} \right] & = n^2 \cdot \lambda^2 L^2 + n\cdot  \frac{\lambda}{\lambda+1} (L-1) +O(1)\\
    \mathbb{E} \left[(K_n)_{\downarrow 3} \right] & = n^3 \cdot \lambda^3 L^3 + n^2 \cdot 3 \frac{\lambda^2}{\lambda+1} \left(\frac{1}{2} L ^2 -L\right) + O(n)\\      
       \mathbb{E} \left[(K_n)_{\downarrow 4} \right] & = n^4 \cdot \lambda^4L^4 + n^3 \cdot 2\frac{\lambda^3}{\lambda+1} (L^3 - 3L^2) + O\left(n^2\right).
 \end{align*}
 Since $\mathbb{E} \left[(K_n)_{\downarrow 1} \right] = \mathbb{E}[K_n]$, this also gives
 \begin{align*}
        \mathbb{E} \left[K_n\right]^2 & = n^2 \cdot \lambda^2 L^2 + n\cdot  L\frac{\lambda}{\lambda+1}+O(1)\\
         \mathbb{E} \left[K_n\right]^3 & = n^3 \cdot \lambda^3 L^3 + n^2\cdot  3 L^2\frac{\lambda^2}{2(\lambda+1)}+O(n)\\
          \mathbb{E} \left[K_n\right]^4 & = n^4 \cdot \lambda^4 L^4 + n^3\cdot  2L^3\frac{\lambda^3}{\lambda+1}+O\left(n^2\right).
 \end{align*}   
It follows that
\begin{displaymath}
\mathbb{E} \left[\left(K_n - \mathbb{E} \left[ K_n\right] \right)^4 \right] =  n^4 \cdot \mathfrak{A}( \lambda) + n^3 \cdot \mathfrak{B}( \lambda) + O\left(n^2\right)
\end{displaymath}
 with 
      \begin{displaymath}
      \mathfrak{A}(\lambda) = \lambda^4 L^4 \left(-3+6-4+1\right)
      \end{displaymath}
      and
      \begin{align*}
       \mathfrak{B}( \lambda) & = 2 \,\lambda^3  L \left\{ -3\frac{L^2}{\lambda+1} + 3 L^2 + 3 L \left[ \frac{2}{\lambda+1}L - \frac{1}{\lambda+1} \right] -6  L^2 \right. \\
       &\left. \quad - 2 \left[ L^2 \frac{1}{2(\lambda+1)} + \frac{3}{\lambda+1} \left(\frac{L^2}{2} -L\right)\right]  +3L^3 + \frac{1}{\lambda+1} (L^2 -3L) \right\}
      \end{align*}
      Straightforward algebraic computations show that
      \begin{displaymath}
       \mathfrak{A}( \lambda) =  \mathfrak{B}(\lambda) =0, 
       \end{displaymath} 
       proving equation \eqref{mom4Kn}.

\subsubsection{Proof of equation \eqref{eq: remainders_dir}}\label{app: sec_BE_dir}

Preliminary note that if $s = e^{\frac{i\xi}{\sqrt{n} \mathfrak{s}_{0,\lambda}(z)}}$ with $\xi$ satisfying \eqref{eq: BE_xi_bound_dir}, then  \citet[Chapter IV, Lemma 5]{Pet(75)} entails 
$$
\left|s-1 \right| = \left|e^{\frac{i\xi}{  \sqrt{n} \mathfrak{s}_{0,\lambda}(z)}}-1 \right|  \le \left| \frac{\xi }{ \sqrt{n} \mathfrak{s}_{0,\lambda}(z)}\right| \le \mathcal{C} \, n^{\delta - \frac{1}{2}}. 
$$
In particular, there exists $\bar{n} \in \N$ such that, for every $n \ge \bar{n}$, 
$$ 
\left| \frac{\xi }{ \sqrt{n} \mathfrak{s}_{0,\lambda}(z)}\right| \le \frac{\pi}{3}.  
$$
From now on, we assume $n \ge \bar{n}$. Concerning the behaviour of $\mathfrak{R}^{(n)}_1$, rewrite $\mathfrak{R}^{(n)}_1(s)$ as
$$\mathfrak{R}_1^{(n)}(s) := 1 + \mathfrak{R}(n (s \lambda+1) ) + \mathfrak{R}(n \lambda) -  \mathfrak{R}(n s \lambda) -  \mathfrak{R}(n( \lambda+1))  + \mathfrak{E}(s),$$
where 
$$
\mathfrak{E}(s) = \mathfrak{R}(n (s \lambda+1) )   \mathfrak{R}(n \lambda) - \left[ \mathfrak{R}(n (s \lambda+1) ) + \mathfrak{R}(n \lambda) \right] \, \varepsilon(s) + \frac{\varepsilon(s)^2}{1+\varepsilon(s)}
$$
with 
$$
\varepsilon(s) : =  \mathfrak{R}(n s \lambda) + \mathfrak{R}(n( \lambda+1)) +  \mathfrak{R}(n s \lambda)    \mathfrak{R}(n( \lambda+1)). 
$$
From \eqref{eq:gammabound}, it holds
$$ 
\left| \mathfrak{R}(w) \right| \le  \frac{3}{2 \pi^2 \, \left| w \right|}   \le  \frac{3}{2 \pi^2 \, n\lambda} 
$$
for every $w \in \{n( s \lambda+1), ns\lambda, n(\lambda+1), n\lambda\}$, since $|w| \ge n\lambda$. This implies 
$$
|\mathfrak{R}_1^{(n)}(s) - 1 | \le  \frac{1}{n} \, \frac{6}{ \pi^2 \lambda}  + |\mathfrak{E}(s)| \hspace{10pt} \text{and } \hspace{10pt}  |\mathfrak{E}(s)| = O_{u.c.} \left(\frac{1}{n^2}\right). 
$$
Thus, 
$$  
\mathfrak{R}_1^{(n)}\left(s\right) = 1+ O_{u.c.}\left(n^{-1}\right).
$$
Concerning the behaviour of $\mathfrak{R}^{(n)}_2$, rewrite $\mathfrak{R}^{(n)}_2$ as 
$$ 
\mathfrak{R}_2^{(n)}(s)  = \left(1 + \frac{ s-1}{s\lambda+1}\right)^\frac{1}{2}.   
$$
Then, 
\begin{align*}
\left|  \mathfrak{R}_2^{(n)}(s) -1 \right|  & = \left| \left(1 + \frac{ s-1}{s\lambda+1}\right)^\frac{1}{2} - 1\right|\\
&=    \left|\frac{ s-1}{s\lambda+1} \right|  \sum_{k = 0}^{+\infty} \binom{ \frac{1}{2}}{k+1} \,  \left[\frac{ s-1}{s\lambda+1} \right]^k \le  \left|\frac{ s-1}{s\lambda+1} \right|  C_\varepsilon(s),
\end{align*}
where the last identity holds if  $|\frac{ s-1}{s\lambda+1}|\le 1-\varepsilon$, for $C_\varepsilon$ a continuous function only depending on $\varepsilon>0$. Since, for $\xi$ satisfying \eqref{eq: BE_xi_bound_dir}, $|s\lambda + 1|>\lambda$, there exists $\bar{n} \in \N$ such that for every $n \ge \bar{n}$ 
$$
 \left|\frac{ s-1}{s\lambda+1} \right| \le \frac{\mathcal{C}}{\lambda} n^{\delta-\frac{1}{2}} \le 1-\varepsilon
$$
and
$$
\left|  \mathfrak{R}_2^{(n)}(s) -1 \right|  \le  C_\varepsilon(s) \, \frac{\mathcal{C}}{\lambda}\,  n^{\delta-\frac{1}{2}}.
$$
We conclude that
$$
\mathfrak{R}_2^{(n)}\left(s\right)  = 1+ O_{u.c.}\left(n^{\delta-\frac{1}{2}}\right).
$$
Concerning the behaviour of $\mathfrak{R}^{(n)}_3$, upon noting that the function $f$ is holomorphic in the disc of center 1 and radius 1, argue as for $\mathfrak{R}_4^{(n)}$ of the case $\alpha \in (0, 1)$ to conclude that
$$
\mathfrak{R}_3^{(n)}\left(s\right)  = 1+O_{u.c.}\left(n^{3\delta-\frac{1}{2}}\right).
$$

Altogether, we proved that, for $s = e^{\frac{i\xi}{\sqrt{n} \mathfrak{s}_{0,\lambda}(z)}}$ with $\xi$ satisfying \eqref{eq: BE_xi_bound_dir} and for every $n \ge \bar{n}$, 
\begin{align*}
\mathfrak{R}_1^{(n)}\left(s\right) &= 1+ O_{u.c.}\left(n^{-1}\right) \\
\mathfrak{R}_2^{(n)}\left(s\right) &= 1+ O_{u.c.}\left(n^{\delta-\frac{1}{2}}\right) \\
\mathfrak{R}_3^{(n)}\left(s\right) &= 1+O_{u.c.}\left(n^{3\delta-\frac{1}{2}}\right).
\end{align*}
Recalling that $\mathfrak{R}^{(n)}(\xi)$ is defined as the product of these three terms, and since $0> 3\delta-\frac{1}{2} > \delta-\frac{1}{2} > -1$ for any $\delta \in \left(0, \frac{1}{6}\right)$, 
we conclude that $\mathfrak{R}^{(n)}(z, \xi) = 1+O_{u.c.}\left(n^{3\delta-\frac{1}{2}}\right)$.



\begin{thebibliography}{9}

\bibitem[Balocchi et al.(2024)]{Bal(24)}
\textsc{Balocchi, C., Favaro, S. and Naulet, Z.} (2024). Bayesian nonparametric inference for ``species-sampling" problems. \textit{Statistical Science}, to appear.

\bibitem[Belki\`c(2019)]{Bel(19)}
\textsc{Belki\`c, D.} (2019) All the trinomial roots, their powers and logarithms from the Lambert series, Bell polynomials and Fox-Wright function: illustration for genome multiplicity in survival of irradiated cells. \emph{Journal of Mathematical Chemistry} \textbf{57}, 59--106.

\bibitem[Bercu and Favaro(2024)]{Ber(24)}
\textsc{Bercu, B. and Favaro, S.} (2024). A martingale approach to Gaussian fluctuations and laws of iterated logarithm for Ewens-Pitman model. \textit{Stochastic Processes and Their Applications}. To appear.

\bibitem[Bingham et al.(1989)]{Bin(89)}
\textsc{Bingham, N.H., Goldie, C.M.  and Teugels, J. L.} (1989). \textit{Regular variation}. Cambridge University Press

\bibitem[Charalambides(2005)]{Cha(05)}
\textsc{Charalambides} (2005) \textit{Combinatorial methods in discrete distributions.} Wiley.

\bibitem[Charalambides(2007)]{Cha(07)}
\textsc{Charalambides, C.A.} (2007). Distributions of random partitions and their applications. \textit{Methodology and Computing in Applied Probability} \textbf{9}, 163--193.

\bibitem[Chow and Teicher(1997)]{CT(97)} 
\textsc{Chow, Y.S. and  Teicher, H.} (1997) \textit{Probability theory. Independence, interchangeability, martingales.} Springer.

\bibitem[Crane(2016)]{Cra(16)}  
\textsc{Crane, H.} (2016). The ubiquitous Ewens sampling formula. \textit{Statistical Science} \textbf{31}, 1--19.

\bibitem[Contardi et al.(2024)]{Con(24)} 
\textsc{Contardi, C., Dolera, E. and Favaro, S.} (2024). Gaussian credible intervals in Bayesian nonparametric estimation of the unseen. \textit{Technical report}.

\bibitem[Dawson and Feng(2006)]{Daw(06)}
\textsc{Dawson, D.A. and Feng, S.} (2006). Asymptotic behavior of Poisson-Dirichlet distribution for large mutation rate. \textit{The Annals of Applied Probability} \textbf{16}, 562--582.

\bibitem[Dolera and Favaro(2020)]{Dol(20)}
\textsc{Dolera, E. and Favaro, S.} (2020). A Berry--Esseen theorem for Pitman's $\alpha$--diversity. \textit{The Annals of Applied Probability} \textbf{30}, 847--869.

\bibitem[Dolera and Favaro(2021)]{Dol(21)}
\textsc{Dolera, E. and Favaro, S.} (2021). A compound Poisson perspective of Ewens-Pitman sampling model. \textit{Mathematics} \textbf{9}, 2820.

\bibitem[Dudley(2002)]{Dud(02)}
\textsc{Dudley, R.M.} (2002). \textit{Real Analysis and Probability}. Cambridge University Press.

\bibitem[Ewens(1972)]{Ewe(72)}
\textsc{Ewens, W.} (1972). The sampling theory or selectively neutral alleles. \textit{Theoretical Population Biology} \textbf{3}, 87--112.

\bibitem[Favaro et al.(2009)]{Fav(09)}  
\textsc{Favaro, S., Lijoi, A., Mena, R.H. and Pr\"unster, I.} (2009). Bayesian nonparametric inference for species variety with a two parameter Poisson-Dirichlet process prior. \textit{Journal of the Royal Statistical Society Series B} \textbf{71}, 992--1008.

\bibitem[Favaro and Feng(2014)]{Fav(14)}
\textsc{Favaro, S. and Feng, S} (2014). Asymptotics for the number of blocks in a conditional Ewens-Pitman sampling model. \textit{Electronic Journal of Probability} \textbf{19}, 1--15.

\bibitem[Favaro et al.(2018)]{Fav(18)}
\textsc{Favaro, S., Feng, S. and Gao. F.} (2018). Moderate deviations for Ewens-Pitman sampling models. \textit{Sankhya A} \textbf{80}, 330--341.

\bibitem[Feng(2007)]{Fen(07)}
\textsc{Feng, S.} (2007). Large deviations associated with Poisson–Dirichlet distribution and Ewens sampling formula. \textit{The Annals of Applied Probability}\textbf{17}, 1570--1595.

\bibitem[Feng(2007a)]{Fen(07a)}
\textsc{Feng, S.} (2007). Large deviations for Dirichlet processes and Poisson-Dirichlet distribution with two parameters. \textit{Electronic Journal of Probability}\textbf{12}, 787--807.

\bibitem[Feng(2010)]{Feng(10)}
\textsc{Feng, S.} (2010). \textit{The Poisson-Dirichlet distribution and Related Topics}. Springer.

\bibitem[Feng and Gao(2008)]{Fen(08)}
\textsc{Feng, S. and Gao, F.Q.} (2008). Moderate deviations for Poisson-Dirichlet distribution. \textit{The Annals of Applied Probability}\textbf{18}, 1794--1824.

\bibitem[Feng and Gao(2010)]{Fen(10)}
\textsc{Feng, S. and Gao, F.Q.} (2010). Asymptotic results for the two-parameter Poisson-Dirichlet distribution. \textit{Stochastic Processes and their Applications}\textbf{120}, 1159--1177.

\bibitem[Feng and Hoppe(1998)]{Fen(98)}
\textsc{Feng, S. and Hoppe, F.M.} (1998). Large deviation principles for some random combinatorial structures in population genetics and Brownian motion. \textit{The Annals of Applied Probability}\textbf{8}, 975--994.

\bibitem[Griffiths(1979)]{Gri(79)}
\textsc{Griffiths, R.C.} (1979). On the distribution of allele frequencies in a diffusion model. \textit{Theoretical Population Biology}\textbf{15}, 140--158.

\bibitem[Huber(1981)]{Hub(81)}  
\textsc{Huber, P.} (1981). \textit{Robust statistics}. Wiley.

\bibitem[Joyce et al.(2002)]{Joy(02)}
\textsc{Joyce, P., Krone, S.M. and Kurtz, T.G.} (2002). Gaussian limits associated with the Poisson--Dirichlet distribution and the Ewens sampling formula. \textit{The Annals of Applied Probability} \textbf{12}, 101--124.

\bibitem[Kilbas and Saigo(2002)]{KiSa(04)}
\textsc{Kilbas, A.A. and Saigo, M.} (2004). \textit{H-transforms. Theory and applications}. Chapman \& Hall/CRC. 

\bibitem[Kingman(1975)]{Kin(75)}
\textsc{Kingman, J.F.C.} (1975). Random discrete distributions. \textit{Journal of the Royal Statistical Society Series B} \textbf{37}, 1--15.

\bibitem[Korwar and Hollander(1973)]{Kor(73)}
\textsc{Korwar, R.M. and Hollander, M.} (1973). Contribution to the theory of Dirichlet processes. \textit{The Annals of Probability} \textbf{1}, 705--711.

\bibitem[Lijoi et al.(2007)]{Lij(07)}   
\textsc{Lijoi, A., Mena, R.H. and Pr\"unster, I.} (2007). Bayesian nonparametric estimation of the probability of discovering new species. \textit{Biometrika} \textbf{94}, 769--786.

\bibitem[DLMF(2024)]{nist} \textit{NIST Digital Library of Mathematical Functions. } https://dlmf.nist.gov/, Release 1.2.1.

\bibitem[Paris(2021)]{Par(21)}
\textsc{Paris, R.B.} (2021). The asymptotic expansion of Kr\"atzel’s integral and an integral related to an extension of the Whittaker function. \emph{Preprint ArXiv: 2112.02928}

\bibitem[Paris and Kaminski(2001)]{PaKa(01)}
\textsc{Paris, R.B. and Kaminski, D.} (2001). \textit{Asymptotics and Mellin--Barnes integrales}. Encyclopedia of Mathematics and its Applications, 85. 

\bibitem[Pereira et al.(2022)]{Per(22)}
\textsc{Pereira, A., Oliveira, R.I. and Ribeiro, R.} (2022). Concentration in the generalized Chinese restaurant process. \textit{Sankhya A} \textbf{80}, 628-670.

\bibitem[Perman et al.(1992)]{Per(92)}
\textsc{Perman, M., Pitman, J. and Yor, M.} (1992). Size-biased sampling of Poisson point processes and excursions. \textit{Probability Theory and Related Fields} \textbf{92}, 21--39.

\bibitem[Petrov(1975)]{Pet(75)}
\textsc{Petrov, V.V.} (1975). \textit{Sums of independent random variables}. Springer.

\bibitem[Pitman(1995)]{Pit(95)}
\textsc{Pitman, J.} (1995). Exchangeable and partially exchangeable random partitions. \textit{Probability Theory and Related Fields} \textbf{102}, 145--158.

\bibitem[Pitman and Yor(1997)]{Pit(97)}
\textsc{Pitman, J. and Yor, M.} (1997). The two parameter Poisson-Dirichlet distribution derived from a stable subordinator. \textit{The Annals of Probability} \textbf{25}, 855--900.

\bibitem[Pitman(2006)]{Pit(06)}
\textsc{Pitman, J.} (2006). \textit{Combinatorial stochastic processes}. Lecture Notes in Mathematics, Springer Verlag.

\bibitem[Pollard(1946)]{Poll(46)}
\textsc{Pollard, H.} (1946). The representation of $e^{-x^{\lambda}}$ as a Laplace integral. \textit{Bulletin of the American Mathematical Society} \textbf{52}, 908--910.

\bibitem[Strassen(1967)]{Str(67)}
\textsc{Strassen, V.} (1967).  Almost sure behavior of sums of independent random variables and martingales. In \textit{Berkeley Symposium on Mathematical Statistics and Probability} \textbf{5}, 315--343.

\bibitem[Tricomi and Erd\`elyi(1951)]{TE(51)}
\textsc{Tricomi, F.G. and Erd\`elyi, A.} (1951). The asymptotic expansion of a ratio of Gamma functions. \textit{Pacific Journal of Mathematics} \textbf{1}, 133--142.

\bibitem[Watterson and Guess(1977)]{Wat(77)}
\textsc{Watterson. G.A. and Guess, H.A.} (1977). Is the most frequent allele the oldest? \textit{Theoretical Population Biology} \textbf{11}, 141--160.

\bibitem[Zabell(2005)]{Zab(05)}
\textsc{Zabell, S.L.} (2005). \textit{Symmetry and its discontents: essays on the history of inductive probability}. Cambridge University Press.

\bibitem[Zolotarev(1986)]{Zol(86)}
\textsc{Zolotarev, V.M.} (1986). \textit{One-dimensional Stable distributions.} American Mathematical Society.

\end{thebibliography}
\end{document}